\documentclass[11pt]{amsart}
\usepackage{latexsym,amssymb,amsmath,hyperref}
\usepackage{amsthm, dsfont}
\usepackage[all]{xy}
\usepackage[short,nodayofweek]{datetime}
\usepackage[active]{srcltx}
\leftmargin  \leftmargini
\def\NN{{\mathbb N}}

\def\NN{\operatorname N}

\newcommand{\vol}{\operatorname{vol}}

\def\d{{\rm d}}

\theoremstyle{plain}
\newtheorem{thm}{Theorem}[section]
\newtheorem{cor}[thm]{Corollary}
\newtheorem{lem}[thm]{Lemma}
\newtheorem{prop}[thm]{Proposition}

\newtheorem{eg}[thm]{Example}
\newtheorem{assumption}[thm]{Hypothesis}

\theoremstyle{definition}
\newtheorem{defn}[thm]{Definition}
\theoremstyle{remark}

\newtheoremstyle{Acknowledgements}% name
  {}% {\topsep}%      Space above
    {}% {\topsep}%      Space below
     {}%         Body font
     {}%         Indent amount (empty = no indent, \parindent = para indent)
    {\bfseries}% Thm head font
    {}%        Punctuation after thm head
     {.5em}%     Space after thm head: " " = normal interword space;
     {\thmname{#1}\thmnumber{ }\thmnote{ (#3)}}%         Thm head spec (can be left empty, meaning `normal')
\theoremstyle{Acknowledgements}

\date{\today, \currenttime} 

\begin{document}

\title[a generalisation of local Gross-Zagier]{A generalisation of Zhang's local Gross-Zagier formula}

\author{Kathrin Maurischat}
\address{\rm {\bf Kathrin Maurischat}, Mathematics Center  Heidelberg
  (MATCH), Heidelberg University, Im Neuenheimer Feld 288, 69120 Heidelberg, Germany }
\curraddr{}
\email{\sf maurischat@mathi.uni-heidelberg.de}
\thanks{}

% AMS-Classification

% Keywords

%-----------------------------------------------------------  
\begin{abstract}

On the background of Zhang's local Gross-Zagier formulae for $\operatorname{GL}(2)$ \cite{zhang}, we study some $\wp$-adic problems. The local Gross-Zagier formulae give identities of very special local geometric data (local linking numbers) with certain local Fourier coefficients of a Rankin $L$-function.
The local linking numbers are local coefficients of a geometric (height) pairing. The Fourier coefficients are products of the local Whittaker functions of two automorphic representations of $\operatorname{GL}(2)$ \cite{zhang}.
We establish a matching of the space of  local linking numbers with  the space of all those Whittaker products. 
Further, we construct a universally defined operator on the local linking numbers which reflects the behavior of the analytic Hecke operator. Its suitability is shown by recovering from it an equivalent of the local Gross-Zagier formulae.
Our methods  are throughout constructive and computational.
\end{abstract}
%-----------------------------------------------------------  

\maketitle
%---------------------------------------------------------------------
% Beginn des eigentlichen Artikels
%---------------------------------------------------------------------

%%%%%%%%%%%%%%%%%%%%%%%%%%%%%%%%%%%%%%%%%%%%%%%%%%%%%%%%%%%%%%%%%%%%%%%%%%%%%%%%%%%%%%%%%%%%%%%%%%%%%%%%%%%%%%%%%%%%%%%%%%%%%%%%%%%%%%%%%%%%%%%%%%%%%%%%%%%%%%%%%%%%%%%%%%%%%%%%%%%%%%%%%%%%%%%%%%%%%%%%%%%%%%%%%%%%%%%%%%%%%%%%%%%%%%%%%%%%%%%%%%
% Section: Introduction

%%%%%%%%%%%%%%%%%%%%%%%%%%%%%%%%%%%%%%%%%%%%%%%%%%%%%%%%%%%%%%%%%%%%%%%%%%%%%%%%%%%%%%%%%%%%%%%%%%%%%%%%%%%%%%%%%%%%%%%%%%%%%%%%%%%%%%%%%%%%%%%%%%%%%%%%%%%%%%%%%%%%%%%%%%%%%%%%%%%%%%%%%%%%%%%%%%%%%%%%%%%%%%%%%%%%%%%%%%%%%%%%%%%%%%%%%%%%%%%%%%

\section{Introduction}
In this paper we study local $\wp$-adic problems which have their origin in the famous work of Gross and Zagier~\cite{gross-zagier}. There, Gross and Zagier give a relation between a Heegner point of discriminant $D$ on $X_0(N)$ and the $L$-function attached to a modular newform $f$ of weight $2$ and level $N$ and a character $\chi$ of the class group of $K=\mathbb Q(\sqrt D)$, in case $D$ is squarefree and prime to $N$:
\begin{equation}\label{einleitung_gross-zagier}
 L'(f,\chi,s=\frac{1}{2})=const\cdot\hat h(c_{\chi,f})
\end{equation}
Here, $\hat h$ is the height on $\operatorname{Jac}(X_0(N))$ and $c_{\chi,f}$ is a component of the Heegner class depending on $\chi$ and $f$.
The proof involves a detailed study of the local heights as well as Rankin's method and  holomorphic projection. The Fourier coefficients of the Rankin kernel happen to equal certain coefficients of the cycle under the action of Hecke correspondences.
Applying this formula  to elliptic curves attached to such $f$ sa\-tis\-fying $\operatorname{ord}_{s=1}L(E,s)=1$, Gross and Zagier find points of infinite order on $E(\mathbb Q)$. 
\medskip

Zhang in \cite{zhang}, \cite{zhang2}  improves the results in different aspects. 
He gives a kernel function for the Rankin-Selberg $L$-function which satisfies a functional equation similar to that for $L$. This enables him to compute the central values and the holomorphic projection without the technical difficulty in \cite{gross-zagier} of taking the trace down to the level of the newform. Moreover, he can compare this kernel directly with the height pairing.

Zhang switches the point of view to a completely local one. Thus, modular forms are now viewed as automorphic representations of $\operatorname{GL}_2$ and $\operatorname{CM}$-points are studied on the Shimura variety. (See Chapter~\ref{section_definitionen} for the concrete definitions.) The height pairing of $\operatorname{CM}$-cycles is replaced by a geometric pairing of Schwartz functions $\phi,\psi\in\mathcal S(\chi,\mathbf G(\mathbb A))$,
\begin{equation*}
 <\phi,\psi>=\sum_\gamma m_\gamma <\phi,\psi>_\gamma.
\end{equation*}
While the geometric input is included in the multiplicities $m_\gamma$, the coefficients \mbox{$<\phi,\psi>_\gamma$} now are pure ad\`elic integrals.
In that way, finding a $\operatorname{CM}$-point fitting in (\ref{einleitung_gross-zagier}) is replaced by giving local Schwartz functions for which the local components (the so called local linking numbers) of $<\phi,\psi>_\gamma$ correspond to the local Fourier coefficients of the kernel. These identities are called local Gross-Zagier formulae. (See for example Theorem~\ref{zhangs_local_Gross-Zagier} or the original \cite{zhang} Lemma~4.3.1.)
In the meanwhile, Zhang et al \cite{zhang3} proved these results with no level constraint at all.
\medskip

In this paper, we study the local correspondences between the local linking numbers and the local Fourier coefficients qualitatively at finite places not dividing $2$. We look at the general local linking numbers as well as the Fourier coefficients of the general Mellin transforms of which the $L$-function is the common divisor.
We characterize the local linking numbers as functions on the projective line satisfying certain properties (Proposition~\ref{eigenschaften_nonsplit}, \ref{eigenschaften_split}  and \ref{charakterisierung_der_LLN}).
The Fourier coefficients  are products of Whittaker functions of the automorphic representations  occuring in the Rankin-Selberg convolution, that is the ``theta series'' $\Pi(\chi)$ and the ``Eisenstein series'' $\Pi(\lvert\cdot\rvert^{s-\frac{1}{2}},\lvert\cdot\rvert^{\frac{1}{2}-s}\omega)$. We find that the spaces of these two kinds of functions are essentially the same (Theorem~\ref{satz_matching}).
We construct an operator on  local linking numbers that reflects the Hecke operator on the analytic side (Theorem~\ref{satz_matching_operator}).
This ``geometric Hecke'' operator is tested  quantitatively in the setting of Zhang's local Gross-Zagier formula \cite{zhang} afterwards. It is seen that it produces concrete results equivalent to the local Gross-Zagier formulae (Theorems~\ref{neuformulierung_local_Gross-Zagier} resp. \ref{zhangs_local_Gross-Zagier}).
\medskip

Our techniques used in the proofs are throughout  computational and constructively explicit.
While the background is highly technical and requires big theoretical input, we achieve an insight into the local identities by in parts vast but underlyingly elementary computations ($\wp$-adic integration).

In adapting Zhang's setting we  give evidence that the local Gross-Zagier formulae are not exceptional or unique but included in a more general identity of spaces.
In other words, we provide the geometric (height) pairing with a vast (indeed exploited) class of local test vectors.

By this, the need for a ``geometric Hecke'' operator which is defined universally on any local linking number but not only for a special one  as Zhang's (Section~\ref{section_zhang_referiert} resp. \cite{zhang}, Chap.~4.1)
becomes evident. In giving such a universal operator 
which in addition reproduces Zhang's local Gross-Zagier formulae, we achieve a quite general notion of what is going on around  these formulae.

It stands to reason whether  analogous results for higher genus are possible.
At the moment we have no suggestion how to achieve those. The method of brute force computation tracked here will certainly produce computational overflow not manageable any more.
\medskip

This article is the aggregation of a doctoral thesis~\cite{diss}.
In Section~\ref{section_geom_hintergrund} we review the geometric background needed to define the local linking numbers (Definition~\ref{def_lln}). This is mainly due to Zhang (\cite{zhang}, Chapter~4.1). Note that we have no level constraint  because we make no use of the newform at all.
The only obstruction of the data is that the conductor of the id\`ele character $\chi$ is coprime to the discriminant of the imaginary quadratic field $\mathbb K$ over the totally real field $\mathbb F$. 
 We prefer for computations a parametrization (in the variable $x\in F^\times$) of the local linking numbers which differs form Zhang's one in $\xi$ but is related to that by $x=\frac{\xi}{\xi-1}$.
Roughly speaking, when (locally at a finite place) $G$  is  an inner form of the projective group $\operatorname{PGL}_2(F)$ having a maximal torus $T$ isomorphic to $F^\times\backslash K^\times$ and when $\mathcal S(\chi, G)$ is the Schwartz space of functions transforming like $\chi$ under $T$, then  the local linking number of $\phi,\psi\in\mathcal S(\chi,G)$ is given by
\begin{equation}\label{gl_einleitung_lln}
 <\phi,\psi>_x=\int_{T\backslash G}\int_T\phi(t^{-1}\gamma(x)ty)~dt~\bar\psi(y)~dy,
\end{equation}
where $\gamma(x)$ is a carefully chosen representative of a double coset in $T\backslash G/T$.
These local linking numbers are studied in Chapter~\ref{section_charakterisierung}. They are functions on $F^\times$ characterized by four properties. If $K/F$ is a splitting algebra, then they are exactly the locally constant functions vanishing around $1$ and owning certain chracteristic behavior for $\lvert x\rvert\to 0$ resp. $\lvert x\rvert\to\infty$ (Proposition~\ref{eigenschaften_split}). Similar conditions hold for a field extension $K/F$ (Proposition~\ref{eigenschaften_nonsplit}).

The data used from the theory of automorphic forms are summarized in Section~\ref{section_autom_formen_hintergrund}.
The Rankin-Selberg $L$-series is the common divisor of all the associated Mellin transforms. Their Fourier coefficients are given by products of Whittaker functions both from the theta series $\Pi(\chi)$ and the Eisenstein series. While for the $L$-function essentially the newforms are used (\cite{zhang}, Chapter~3.3), here all Whittaker functions are taken into account. More precisely, only the Kirillov functions are needed which are described by classical results. In this way the results on the analytic side are direct conclusions from well-known facts on automorphic representations of $\operatorname{GL}_2(F)$.

The matching of local linking numbers and Whittaker products is shown in Chapter~\ref{section_matching}.

The rest of the paper is devoted to the exploration of a geometric Hecke operator. This operator is roughly speaking a weighted average of translations of local linking numbers, namely of
\begin{equation*}
 <\phi,\begin{pmatrix}b&0\\0&1\end{pmatrix}.\psi>_x,
\end{equation*}
where $b\in F^\times$. This translation is a first natural candidate for a possible operator since the Hecke operator acts on Whittaker products essentially via translation by $b\in F^\times$ (Proposition~\ref{prop_analytischer_Hecke_allgemein}), too.
The translated local linking numbers are studied in Chapter~\ref{section_translation}.
There is a crucial difference in their properties as well as in the proofs according to whether the torus $T=F^\times\backslash K^\times$ is or is not compact. 
In the first case, the inner integral of the local linking numbers (\ref{gl_einleitung_lln}) has compact support which allows a quick insight into the behavior of the translation (Section~\ref{section_compact}): Fixing $x$, the translation is a compactly supported locally constant function in $b\in F^\times$ (Proposition~\ref{prop_translatiert_kompakt}). This case is completed by an explicit Example~\ref{bsp_lln_translatiert_kompakt} which is proved in Appendix~A. 
In case of a noncompact torus $T$ (i.e., $K/F$ splits) the inner integral is not compactly supported anymore which complicates  study and results enormously. There are terms in the absolute value $\lvert b\rvert^{\pm 1}$ and the valuation $v(b)$ occuring (Theorem~\ref{satz_translatiert_nicht_kompakt}). The proof in~\cite{diss} takes one hundred pages of vast $\wp$-adic integration which cannot be reproduced here. We sketch the outline of this proof and refer to \cite{diss}, Chapter~8, for computations. What is more, we include an Example~\ref{bsp_lln_nichtkompakt} proved in Appendix~B, to give a flavour of what is going on. Moreover, the functions considered in Examples~\ref{bsp_lln_translatiert_kompakt} and \ref{bsp_lln_nichtkompakt} are those used in \cite{zhang} for local Gross-Zagier.

The geometric Hecke operator is studied in Chapter~\ref{section_geom_Hecke}. 
The translations itself do not realize the leading terms of the  asymptotical behavior of the translated Whittaker products, that is the behavior of the analytic Hecke operator.
The operator is constructed such that it satisfies this requirement  (Theorem~\ref{satz_matching_operator}).  Such an operator is not uniquely determined. We choose it such that further results become quite smooth.

Finally, this operator is tested by rewriting the local Gross-Zagier formula in terms of it.
For this, we first report the results of \cite{zhang}, Chapter~4, as far as we need them (Section~\ref{section_zhang_referiert}). Thus, Zhang's results can be compared directly to ours by the reader. We also give shorter proofs than those in \cite{zhang}. 
The action of the Hecke operator  constructed in Chapter~\ref{section_geom_Hecke} on the local linking numbers used in \cite{zhang} (resp. Examples~\ref{bsp_lln_translatiert_kompakt} and \ref{bsp_lln_nichtkompakt}) is given in Section~\ref{section_meine_Gross-Zagier_formel}. It produces the leading term of the local Gross-Zagier formula (Theorem~\ref{neuformulierung_local_Gross-Zagier}).  Moreover, in case of a compact torus $T$ the result  equals exactly that of \cite{zhang}.

%\begin{ack}
 
%\end{ack}
%preliminary
\section{Terminology and preparation}\label{section_definitionen}
\subsection{Embedding in the geometric background}\label{section_geom_hintergrund}
The local linking numbers were defined by Zhang~\cite{zhang} and the concrete geometric setting here goes back to that there (Chapter~4).
%%%
\subsubsection{Global data} 
We start with a swoop from the global framework to the local data we are after.
Let $\mathbb F$ be a totally real algebraic number field and let $\mathbb K$ be a imaginary quadratic extension of $\mathbb F$. Further, let $\mathbb D$ be a division quaternion algebra over $\mathbb F$ which contains $\mathbb K$ and splits at the archimedean places. Let $\mathbf G$ denote the inner form of the projective  group $\operatorname{PGL}_2$ over $\mathbb F$ which is given by the multiplicative group $\mathbb D^\times$,
\begin{equation*}
 \mathbf G(\mathbb F)=\mathbb F^\times\backslash \mathbb D^\times.
\end{equation*}
Let $\mathbf T$ be the maximal torus of $\mathbf G$ given by $\mathbb K^\times$, i.e. $\mathbf T(\mathbb F)=\mathbb F^\times\backslash \mathbb K^\times$. Let $\mathbb A_{\mathbb F}$ (resp.~$\mathbb A_{\mathbb K}$) be the ad\`eles of $\mathbb F$ (resp.~$\mathbb K$) and let $\mathbb A_{\mathbb F,f}=\prod_{v\mid\infty}1\prod_{v\nmid\infty}'\mathbb F_v$ be the subset of finite ad\`eles.
On $\mathbf T(\mathbb F)\backslash\mathbf G(\mathbb A_{\mathbb F,f})$ there is an action of $\mathbf T(\mathbb A_{\mathbb F,f})$ from the left and an action of $\mathbf G(\mathbb A_{\mathbb F,f})$ from the right. The factor space $\mathbf T(\mathbb F)\backslash\mathbf G(\mathbb A_{\mathbb F,f})$ can be viewed as the set of $\operatorname{CM}$-points of the Shimura variety defined by the inverse system of
\begin{equation*}
 \operatorname{Sh}_K := \mathbf G(\mathbb F)^+\backslash \mathcal H_1^n \times \mathbf G(\mathbb A_{\mathbb F,f})/K.
\end{equation*}
Here $K$ runs though the sufficiently small compact open subgroups of $\mathbf G(\mathbb A_{\mathbb F,f})$, $\mathcal H_1$ is the upper halfplane, and $n$ is the number of the infinite places of $\mathbb F$. The $\operatorname{CM}$-points are embedded in $\operatorname{Sh}_K$ by mapping the coset of $g\in\mathbf G(\mathbb A_{\mathbb F,f})$ to the coset of $(z,g)$, where $z\in\mathcal H_1^n$ is fixed by $\mathbf T$.

Let $\mathcal S(\mathbf T(\mathbb F)\backslash\mathbf G(\mathbb A_{\mathbb F,f}))$ be the Schwartz space, i.e. the space of complex valued functions on $\mathbf T(\mathbb F)\backslash\mathbf G(\mathbb A_{\mathbb F,f})$ which are locally constant and of compact support. A character of $\mathbf T$ shall be a character $\chi$ of $\mathbf T(\mathbb F)\backslash\mathbf T(\mathbb A_{\mathbb F,f})$, that is a character of $\mathbb A_{\mathbb K,f}^\times/\mathbb K^\times$ trivial on $\mathbb A_{\mathbb F,f}^\times/\mathbb F^\times$. Especially, $\chi=\prod\chi_v$ is the product of its local unitary components. One has
\begin{equation*}
 \mathcal S(\mathbf T(\mathbb F)\backslash\mathbf G(\mathbb A_{\mathbb F,f}))=\oplus_\chi \mathcal S(\chi,\mathbf T(\mathbb F)\backslash\mathbf G(\mathbb A_{\mathbb F,f})),
\end{equation*}
where $\mathcal S(\chi,\mathbf T(\mathbb F)\backslash\mathbf G(\mathbb A_{\mathbb F,f}))$ is the subspace of those functions $\phi$ transforming under $\mathbf T(\mathbb A_{\mathbb F,f})$ by $\chi$, i.e.  for $t\in \mathbf T(\mathbb A_{\mathbb F,f})$ and $g\in \mathbf G(\mathbb A_{\mathbb F,f})$: $\phi(tg)=\chi(t)\phi(g)$. Any such summand is made up by its local components, 
\begin{equation*}
 \mathcal S(\chi,\mathbf T(\mathbb F)\backslash\mathbf G(\mathbb A_{\mathbb F,f}))=\otimes_v \mathcal S(\chi_v,\mathbf G(\mathbb A_{\mathbb F_v})).
\end{equation*}
A pairing on $\mathcal S(\chi,\mathbf T(\mathbb F)\backslash\mathbf G(\mathbb A_{\mathbb F,f}))$ can be defined as follows.
For functions $\phi,\psi$ in $\mathcal S(\chi,\mathbf T(\mathbb F)\backslash\mathbf G(\mathbb A_{\mathbb F,f}))$ and a double coset $[\gamma]\in \mathbf T(\mathbb F)\backslash \mathbf G(\mathbb F)/\mathbf T(\mathbb F)$ define the {\bf linking number}
\begin{equation}\label{def_globallinkingnumber0}
 <\phi,\psi>_\gamma:=\int_{\mathbf T_\gamma(\mathbb F)\backslash \mathbf G(\mathbb A_{\mathbb F,f})}\phi(\gamma y)\bar\psi(y)~dy,
\end{equation}
where $\mathbf T_\gamma=\gamma^{-1}\mathbf T\gamma\cap\mathbf T$. For $\gamma$ which normalize $\mathbf T$ one has $\mathbf T_\gamma=\mathbf T$. Otherwise $\mathbf T_\gamma$ is trivial. Here $dy$ denotes the quotient measure of nontrivial Haar measures on $\mathbf G$ and $\mathbf T$ adapted adequately later on. Further, let
\begin{equation*}
 m:\mathbf T(\mathbb F)\backslash \mathbf G(\mathbb F)/\mathbf T(\mathbb F)\rightarrow \mathbb C
\end{equation*}
be a multiplicity function. Then
\begin{equation*}
 <\phi,\psi>:=\sum_{[\gamma]}m([\gamma])<\phi,\psi>_\gamma
\end{equation*}
defines a sesquilinear pairing on $\mathcal S(\chi,\mathbf T(\mathbb F)\backslash\mathbf G(\mathbb A_{\mathbb F,f}))$. While determining the multiplicity function is an essential global problem, the coefficients $<\phi,\psi>_\gamma$ are the data linking global height pairings on curves and local approaches. Their local components are subject of this paper.

\subsubsection{Local data}
In studying the local components of the linking numbers~(\ref{def_globallinkingnumber0}), we restrict to the nondegenerate case, i.e. the case that $\gamma$ does not normalize the torus $\mathbf T$. First notice that
\begin{equation}\label{def_globallinkingnumber}
<\phi,\psi>_\gamma:=\int_{\mathbf T(\mathbb A_{\mathbb F,f})\backslash \mathbf G(\mathbb A_{\mathbb F,f})}\int_{\mathbf T(\mathbb A_{\mathbb F,f})}\phi(t^{-1}\gamma ty)~dt~\bar\psi(y)~dy.
\end{equation}

Assume $\phi=\prod_v\phi_v$ and $\psi=\prod_v\psi_v$. Then
\begin{equation*}
\int_{\mathbf T(\mathbb A_{\mathbb F,f})}\phi(t^{-1}\gamma ty)~dt=\prod_v\int_{\mathbf T(F_v)}\phi_v(t_v^{-1}\gamma_v t_vy_v)~dt_v
\end{equation*}
as well as $<\phi,\psi>_\gamma=\prod_v<\phi,\psi>_{\gamma,v}$, where
\begin{equation}
 <\phi,\psi>_{\gamma,v}:=\int_{\mathbf T(\mathbb F_v)\backslash \mathbf G(\mathbb F_v)}\int_{\mathbf T(\mathbb F_v)}\phi_v(t_v^{-1}\gamma_v t_vy_v)~dt_v~\bar\psi_v(y_v)~dy_v.
\end{equation}
In here one has to observe that the local components $<\phi,\psi>_{\gamma,v}$ depend on the choice $\gamma$ while $<\phi,\psi>_v$ does not. Thus, one has to work a little  to get a neatly definition.

As  all the following is local, one simplyfies notation:
Let $F$ denote a localization of $\mathbb F$ at a finite place which does not divide $2$. Then $K$ is the quadratic extension of $F$ coming from $\mathbb K$. $K$ can be a field, $K=F(\sqrt A)$, or a splitting algebra $K=F\oplus F$. For $t\in K$, let $\bar t$ denote the Galois conjugate of $t$ (resp. $\overline{(x,y)}=(y,x)$ in the split case). The local ring of $F$ (resp.~$K$) is $\mathbf o_F$ (resp.~$\mathbf o_K$). It contains the maximal ideal $\wp_F$ (resp.~$\wp_K$, where in the split case  $\wp_K:=\wp_F\oplus\wp_F$).  Let $\wp_F$ be a prime element for $\mathbf o_F$. If it can't be mixed up, one writes $\wp$ (resp.~$\pi$) for $\wp_F$ (resp.~$\pi_F$). The residue class field of $F$ has characteristic $p$ and $q$ elements. Further, let $\omega$ be the quadratic character of $F^\times$ given by the extension $K/F$ that is, $\omega(x)=-1$ if $x$ is not in the image  of the the norm of $K/F$. 
Let $D:=\mathbb D(F)$, $T:=\mathbf T(F)$ and $G:=\mathbf G(F)$.
By Wedderburn-Artin there are two possibilities: Either the quaternion algebra $D$ is split, i.e.  $D\cong M_2(F)$ and $G\cong \operatorname{PGL}_2(F)$. Or $D$ is not split, i.e. a division ring over $F$. Then $G=F^\times\backslash D^\times$ is a nonsplit inner form of $\operatorname{PGL}_2(F)$. One defines
\begin{equation}\label{def_delta(D)}
 \delta(D):=\left\{\begin{matrix}0,&\textrm{if $D$ is split}\\1,&\textrm{if $D$ is not split}\end{matrix}\right..
\end{equation}
Generally, there is exists $\epsilon\in D^\times$, such that for all $t\in K$ one has $\epsilon t=\bar t\epsilon$ and such that
\begin{equation*}
 D=K+\epsilon K.
\end{equation*}
Then $c:=\epsilon^2\in F^\times$. Let $\NN$ denote the reduced norm on $D$. Restricted to $K$ this is the norm of the extension $K/F$. One has for $\gamma_1+\epsilon\gamma_2\in D$
\begin{equation*}
 \NN(\gamma_1+\epsilon\gamma_2)=\NN(\gamma_1)-c\NN(\gamma_2),
\end{equation*}
as $\NN(\epsilon)=-\epsilon^2=-c$. Thus, $D$ splits exactily in the case $c\in\NN(K^\times)$.
With this notations, one can parametrize the double cosets $[\gamma]\in T\backslash G/T$ by the projective line:
\begin{defn}\label{def_P_funktion}
 Let $P:T\backslash G/T\rightarrow \mathbb P^1(F)$ be defined by 
\begin{equation*}
 P(\gamma_1+\epsilon\gamma_2):=\frac{c\NN(\gamma_2)}{\NN(\gamma_1)}
\end{equation*}
for $\gamma_1+\epsilon\gamma_2\in D^\times$.
\end{defn}
We check that this in fact is well-defined: $P(t(\gamma_1+\epsilon\gamma_2)t')=P(\gamma_1+\epsilon\gamma_2)$ for all $t,t'\in K^\times$. The non-empty fibres of $P$ not belonging to $0$ or $\infty$ are exactly the nondegenerate double cosets.
In case that  $K/F$ is a field extension, $P$ is injective with range $c\NN(K^\times)\cup\{0,\infty\}$.
In case $K/F$ split, the range of $P$ is $F^\times\backslash\{1\}\cup\{0,\infty\}$ and the fibres of $F^\times\backslash\{1\}$ are single double cosets (\cite{jacquet}).

Of course this is just one possibility of parametrization. Zhang~\cite{zhang} (Chapter~4) for example uses $\xi:=\frac{P}{P-1}$ to which we will come back eventually.
\begin{lem}
(\cite{zhang} Chapter~4) Let $\gamma\in D^\times$. In each double coset $T\gamma T$ of $G$ there exists exactly one $T$-conjugacy class of trace zero.
\end{lem}
Now the local components $<\phi,\psi>_\gamma$ of the linking numbers can be declared precisely:
\begin{defn}\label{def_lln}
Let $\phi, \psi\in\mathcal S(\chi,G)$. For $x\in F^\times$ the {\bf local linking number} is defined by
\begin{equation*}
<\phi,\psi>_x:=<\phi,\psi>_{\gamma(x)}
\end{equation*}
if there is a tracefree preimage $\gamma(x)\in D^\times$ of $x$ under $P$. If there doesn't exist  a preimage, then $<\phi,\psi>_x:=0$. Thus, for $x\in c\NN:=c\NN(K^\times)$
\begin{equation*}
<\phi,\psi>_x=\int_{T\backslash G}\int_T\phi(t^{-1}\gamma(x)ty)~dt~\bar\psi(y)~dy.
\end{equation*}
\end{defn}
Notice that this definition is independent of the choice of the element $\gamma(x)$ of trace zero by  unimodularity of the Haar measure on $T$.

There is one general assumption on the character $\chi$ which  will be assumed in all the following.
The conductor of $\chi$ and the discriminant of $K/F$ shall be coprime:
\begin{assumption}\label{voraussetzung_an_chi}
The character $\chi$ of $T$ may only be ramified  if $\omega$ is not.
%The character $\chi$ of $T=F^\times\backslash K^\times$ may only be ramified  if $\omega$, that is $K/F$, is not.
\end{assumption}
The conductor $f(\chi)<\mathbf o_K$ of $\chi$ can be viewed as an ideal of $\mathbf o_F$: If $K=F\oplus F$, then $\chi=(\chi_1,\chi_1^{-1})$ for a character $\chi_1$ of $F^\times$ and $f(\chi)=f(\chi_1)$. If $K/F$ is a ramified field extension, then $\chi$ is unramified, thus $f(\chi)\cap\mathbf o_F=\mathbf o_F$. Lastly, if $K/F$ is an unramified field extension, then $f(\chi)=\pi^{c(\chi)}\mathbf o_K$, where $\pi$ is an uniformizing element for $K$ as well as $F$. That is, $f(\chi)\cap\mathbf o_F=\pi^{c(\chi)}\mathbf o_F$.

There are some simple properties for $\chi$ following from the hypothesis.
\begin{lem}\label{lemma_chi_quadratisch}
Let $\chi$ be as in \ref{voraussetzung_an_chi}. Equivalent are:

(a) $\chi$ is quadratic.

(b) $\chi$ factorizes via the norm.
\end{lem}
\begin{cor}\label{cor_chi}
Assume \ref{voraussetzung_an_chi}.
If  $K/F$ is a ramified field extension, then  the character $\chi$ is quadratic.

If for  an unramified field extension  $K/F$ the charcacter $\chi$ is unramified, then $\chi=1$.
\end{cor}
One concluding remark on Haar measures is in due. Let $da$ be a nontrivial additive Haar measure on $F$. Associated volumes are denoted by $\vol$. The measure $d^\times a$ of the multiplicative group $F^\times$ shall be compatible with $da$. Associated volumes are denoted by $\vol^\times$. Thus,
\begin{equation*}
\vol^\times(\mathbf o_F^\times)=(1-q^{-1})\vol(\mathbf o_F^\times).
\end{equation*}
The measure on $T\backslash G$ shall be the quotient measure induced of those on $G$ and $T$.
%%%%%%%%%%%%%%%%%%%%%%%%%%
%%%
% %%%
% %%%'
%%%%%%%%%%%%%%%%%%%%%%%%%%%%%

\subsection{Automorphic forms}\label{section_autom_formen_hintergrund}
The central object on the automorphic side is the Rankin-Selberg convolution of two automorphic representations.
The Gross-Zagier formula is interested in the central order of its $L$-function.

Let $\Pi_1$ be a cuspidal representation of $\operatorname{GL}_2(\mathbb A_{\mathbb F})$ with trivial central character (i.e. an irreducible component of the discrete spectrum of the right translation on $L^2(\operatorname{GL}_2(\mathbb F)\backslash \operatorname{GL}_2(\mathbb A_{\mathbb F})),1))$ and conductor $N$.

Further, let $\Pi(\chi)$ be the irreducible component belonging to $\chi$ of the Weil representation  of $\operatorname{GL}_2(\mathbb A_{\mathbb F})$ for the norm form of $\mathbb K/\mathbb F$ (e.g.~\cite{gelbart}~\S 7). It has conductor $f(\chi)^2f(\omega)$ and central character $\omega$.

The Rankin-Selberg convolution of $\Pi_1$ and $\Pi(\chi)$ produces (see~\cite{jacquet2}) the (local) Mellin transform
\begin{equation*}
 \Psi(s,W_1,W_2,\Phi)=\int_{Z(F)N(F)\backslash\operatorname{GL}_2(F)}W_1(g)W_2(eg)f_\Phi(s,\omega,g)~dg
\end{equation*}
for Whittaker functions $W_1$ of $\Pi_1$ (resp.~$W_2$ of $\Pi(\chi)$) for an arbitrary nontrivial character of $F$. One defines $e:=\begin{pmatrix}-1&0\\0&1\end{pmatrix}$. In here, the Eisenstein series
\begin{equation*}
 f_\Phi(s,\omega,g)=\lvert\det g\rvert^s\int_{F^\times}\Phi\left((0,t)g\right)\lvert t\rvert^{2s}\omega(t)~d^\times t
\end{equation*}
for a function $\Phi\in\mathcal S(F^2)$ occures. $f_\Phi$ is an element of the principal series $\Pi(\lvert\cdot\rvert^{s-\frac{1}{2}},\omega\lvert\cdot\rvert^{\frac{1}{2}-s})$. Of course, there is an ad\`elic analogon of this.
Analytical continuation of $\Psi$ leads to the $L$-function, the greatest common divisor of all $\Psi$. It is defined by newforms $\phi$ for $\Pi_1$ and $\theta_\chi$ of $\Pi(\chi)$ as well as a special form $E$ of $\Pi(\lvert\cdot\rvert^{s-\frac{1}{2}},\omega\lvert\cdot\rvert^{\frac{1}{2}-s})$:
\begin{eqnarray*}
 L(s,\Pi_1\times\Pi(\chi)) 
&=& \int_{Z(\mathbb A_{\mathbb F})\operatorname{GL}_2(\mathbb F)\backslash\operatorname{GL}_2(\mathbb A_{\mathbb F})}\phi(g)\theta_\chi(g)E(s,g)~dg\\
&=& \int_{Z(\mathbb A_{\mathbb F})\operatorname{GL}_2(\mathbb F)\backslash\operatorname{GL}_2(\mathbb A_{\mathbb F})}W_\phi(g)W_{\theta_\chi}(g)f_E(s,\omega,g)~dg,
\end{eqnarray*}
where $W_\phi$ etc. denotes the associated Whittaker function. This $L$-function satisfies the functional equation
\begin{equation*}
 L(s,\Pi_1\times\Pi(\chi))=\epsilon(s,\Pi_1\times\Pi(\chi))L(1-s,\Pi_1\times\Pi(\chi)),
\end{equation*}
as $\Pi_1$ and $\Pi(\chi)$ are selfdual. For places where $c(\chi)^2c(\omega)\leq v(N)$, the form $E$ (resp. $W_E$) is the newform of the Eisenstein series. In \cite{zhang} (chap.~1.4) an integral kernel $\Xi(s,g)$ is constructed which has a functional equation analogous to that of $L$ and for which
\begin{equation*}
 L(s,\Pi_1\times\Pi(\chi))=\int_{Z(\mathbb A_{\mathbb F})\operatorname{GL}_2(\mathbb F)\backslash\operatorname{GL}_2(\mathbb A_{\mathbb F})}\phi(g)\Xi(s,g)~dg.
\end{equation*}
We remark, that such a kernel depends on the newform of the theta series $\Pi(\chi)$ as well as the Eisenstein series, but not on the special choice of $\Pi_1$. While the construction of the kernel shall not be reported here, its local nonconstant Fourier coefficients are defined by
\begin{equation}\label{def_fourierkoeff}
 W(s,\xi,\eta,g):=W_\theta(\begin{pmatrix}\eta&0\\0&1\end{pmatrix}g) W_E(s,\begin{pmatrix}\xi&0\\0&1\end{pmatrix}g).
\end{equation}
Here $\eta:=1-\xi$. These Fourier coefficients are exactly those analytic functions which are compared to special local linking numbers in the local Gross-Zagier formula (\cite{zhang} Lemma~4.3.1).
In this paper, the restriction to newforms in (\ref{def_fourierkoeff}) will be cided. For this, one looks at the Kirillov models of the representations: Starting from the Whittaker model $\mathcal W(\Pi,\psi)$ of an irreducible admissible representation $\Pi$ for an additive character $\psi$, the Kirillov space $\mathcal K(\Pi)$ is given by
\begin{eqnarray*}
 \mathcal W(\Pi,\psi) &\to& \mathcal K(\Pi),\\
W&\mapsto& k:(a\mapsto W\begin{pmatrix}a&0\\0&1\end{pmatrix}).
\end{eqnarray*}
\begin{prop}\label{prop_kirillov}
 (\cite{godement}, I.36)
Let $\Pi$ be an infinite dimensional irreducible admissible representation of $\operatorname{GL}_2(F)$. The Kirillov space $\mathcal K(\Pi)$ is generated by the Schwartz space $\mathcal S(F^\times)$ along with the following stalks around zero:

(a) If $\Pi$ is supercuspidal, this stalk is zero.

(b) If $\Pi=\Pi(\mu_1,\mu_2)$ is a principle series representation, then it is given by representatives of the form
\begin{itemize}
 \item[$\bullet$] $\left(\lvert a\rvert^{\frac{1}{2}}c_1\mu_1(a)+\lvert a\rvert^{\frac{1}{2}}c_2\mu_2(a)\right)\mathbf 1_{\wp^n}(a)$, if $\mu_1\not=\mu_2$,
\item[$\bullet$]
$\lvert a\rvert^{\frac{1}{2}}\mu_1(a)\left(c_1+c_2v(x)\right)\mathbf 1_{\wp^n}(a)$, if $\mu_1=\mu_2$.
\end{itemize}
Here $c_1,c_2\in\mathbb C$.

(c) If $\Pi=\Pi(\mu_1,\mu_2)$ is special, it is given by representatives
\begin{itemize}
 \item[$\bullet$] $\lvert a\rvert^{\frac{1}{2}}\mu_1(a)\mathbf 1_{\wp^n}(a)$, if $\mu_1\mu_2^{-1}=\lvert \cdot\rvert$,
\item[$\bullet$]
$\lvert a\rvert^{\frac{1}{2}}\mu_2(a)\mathbf 1_{\wp^n}(a)$, if $\mu_1\mu_2^{-1}=\lvert\cdot\rvert^{-1}$.
\end{itemize}
\end{prop}
Now one defines the so called Whittaker products which are products of Kirillov functions actually. The name keeps in mind the origin of these functions as Fourier coefficients.
\begin{defn}\label{def_whittaker_prod}
 Let (locally) $\Pi(\chi)$ be the theta series and $\Pi(1,\omega)$ be the Eisenstein series at the central place $s=\frac{1}{2}$. Then the products
\begin{equation*}
 W(\xi,\eta)=W_\theta(\eta)W_E(\xi)
\end{equation*}
of Kirillov functions $W_\theta\in\mathcal K(\Pi(\chi))$ and $W_E\in\mathcal K(\Pi(1,\omega))$ %for $\eta=1-\xi$
 are called {\bf Whittaker products}.
\end{defn}
Being a component of a Weil representation the theta series $\Pi(\chi)$  is completly described (\cite{jacquet-langlands} \S 1, \cite{gelbart} \S 7). Ad\`elically, it is a Hilbert modular form of conductor $f(\chi)^2f(\omega)$ and of weight $(1,\dots,1)$ at the infinite places. If $K=F\oplus F$ is split, then $\chi=(\chi_1,\chi_1^{-1})$ and $\Pi(\chi)=\Pi(\chi_1,\omega\chi_1^{-1})=\Pi(\chi_1,\chi_1^{-1})$ is a principle series representation. If $K/F$ is a field extension and $\chi$ does not factorize via the norm, then $\Pi(\chi)$ is supercuspidal. While if $\chi=\chi_1\circ\NN$ it is the principle series representation $\Pi(\chi_1,\chi_1^{-1}\omega)=\Pi(\chi_1,\chi_1\omega)$, as $\chi_1^2=1$ by Lemma~\ref{lemma_chi_quadratisch}.
Thus, by Proposition~\ref{prop_kirillov}:
\begin{prop}\label{prop_characterisierung_theta_functionen}
 Let $\Pi(\chi)$ be the theta series.

(a) If $K/F$ is a quadratic field extension and $\chi$ is not quadratic, then the Kirillov space $\mathcal K(\Pi(\chi))$ is given by $\mathcal S(F^\times)\cup\{0\}$.

(b) If $K/F$ is a quadratic field extension and $\chi^2=1$, then the Kirillov space $\mathcal K(\Pi(\chi))$ as a function space in one variable $\eta$ is generated by $\mathcal S(F^\times)$ along with functions around zero of the form
\begin{equation*}
 \lvert\eta\rvert^{\frac{1}{2}}\chi_1(\eta)\left(a_1+a_2\omega(\eta)\right).
\end{equation*}
(c) If $K/F$ is split, then the Kirillov space $\mathcal K(\Pi(\chi))$ as a function space in one variable $\eta$ is generated by $\mathcal S(F^\times)$ along with functions around zero of the form
\begin{itemize}
 \item [$\bullet$] $\lvert\eta\rvert^{\frac{1}{2}}\left(a_1\chi_1(\eta)+a_2\chi_1^{-1}(\eta)\right)$, if $\chi_1^2\not=1$,
\item[$\bullet$] $\lvert\eta\rvert^{\frac{1}{2}}\chi_1(\eta)\left(a_1+a_2v(\eta)\right)$, if $\chi_1^2=1$.
\end{itemize}
\end{prop}
For later use we collect some properties of principal series. For an automorphic form $f\in \Pi(\mu_1\lvert\cdot\rvert^{s-\frac{1}{2}},\mu_2\lvert\cdot\rvert^{\frac{1}{2}-s})$ there is $\phi\in\mathcal S(F^2)$ such that 
\begin{equation}\label{funktion_principal_series}
 f(s,g)=\mu_1(\det g)\lvert \det g\rvert^s\int_{F^\times}\phi\left((0,t)g\right)(\mu_1\mu_2^{-1})(t)\lvert t\rvert^{2s}~d^\times t.
\end{equation}
Conversely, any $\phi\in\mathcal S(F^2)$ defines a form $f_\phi\in \Pi(\lvert\cdot\rvert^{s-\frac{1}{2}},\omega\lvert\cdot\rvert^{\frac{1}{2}-s})$ in that way (e.g.~\cite{bump} chap.~3.7). 
The Whittaker function belonging to $f$ (in a Whittaker model with unramified character $\psi$) is given by the first Fourier coefficient,
\begin{equation*}
 W_f(s,g,\psi)=\int_Ff(s,\begin{pmatrix}0&-1\\1&0\end{pmatrix}\begin{pmatrix}1&x\\0&1\end{pmatrix}g)\bar\psi(x)~dx.
\end{equation*}
Read in the Kirillov model , the form  for  $s=\frac{1}{2}$ is given by evaluation at $g=\begin{pmatrix}a&0\\0&1\end{pmatrix}$, thus
\begin{equation*}
 W_f(a):=W_f(\frac{1}{2},\begin{pmatrix}a&0\\0&1\end{pmatrix},\psi).
\end{equation*}
For $\mu_i$ unramified the newform is obtained by choosing in (\ref{funktion_principal_series}) concretely
\begin{equation*}
% \phi(x,y):=\left\{\begin{matrix}\mathbf 1_{\mathbf o_F}(x)\mathbf 1_{\mathbf o_F}(y), &\textrm{ if $K/F$ nonramified}\\
%                    \omega(y)\mathbf 1_{\mathbf o_F}(x)\mathbf 1_{\mathbf o_F^\times}(y), & \textrm{if $K/F$ ramified}
  %                 \end{matrix}\right..
\phi(x,y)=\mathbf 1_{\mathbf o_F}(x)\mathbf 1_{\mathbf o_F}(y).
\end{equation*}
Thus,
\begin{align}
 W_{\operatorname{new}}(a) &=
\mu_1(a)\lvert a\rvert^{\frac{1}{2}} \int_F\int_{F^\times}\mathbf 1_{\mathbf o_F}(at)\mathbf 1_{\mathbf o_F}(xt)\mu_1\mu_2^{-1}(t)\lvert t\rvert~d^\times t~\bar\psi(x)~dx\nonumber\\
&=
\mu_1(a)\lvert a\rvert^{\frac{1}{2}}\mathbf 1_{\mathbf o_F}(a)\vol(\mathbf o_F)\vol^\times(\mathbf o_F^\times)\sum_{j=-v(a)}^0\mu_1\mu_2^{-1}(\pi^j)\nonumber\\
&=
\lvert a\rvert^{\frac{1}{2}}\mathbf 1_{\mathbf o_F}(a)\vol(\mathbf o_F)\vol^\times(\mathbf o_F^\times)
\left\{\begin{matrix}\frac{\mu_1(a\pi)-\mu_2(a\pi)}{\mu_1(\pi)-\mu_2(\pi)},\textrm{ if }\mu_1\not=\mu_2\\\mu_1(a)(v(a)+1),\textrm{ if }\mu_1=\mu_2\end{matrix}\right..\label{gleichung_whittaker_neuform}
\end{align}

By  Proposition~\ref{prop_kirillov}, we have
\begin{prop}\label{prop_characterisierung_eisenstein_functionen}
 At the central place $s=\frac{1}{2}$ the Eisenstein series is the principle series representation $\Pi(1,\omega)$. Its Kirillov space as a function space in the variable $\xi$ is generated by $\mathcal S(F^\times)$ along with the functions around zero of the form
\begin{itemize}
 \item [$\bullet$] $\lvert\xi\rvert^{\frac{1}{2}}\left(a_1+a_2\omega(\xi)\right)$, if $K/F$ is a field extension,
\item[$\bullet$] $\lvert\xi\rvert^{\frac{1}{2}}\left(a_1+a_2v(\xi)\right)$, if $K/F$ is split.
\end{itemize}
\end{prop}
Let us recall a property of the Hecke operator. For a finite set $S$ of places of $\mathbb F$, let $\mathbb A^S:=\prod_{v\notin S}\mathbf o_{\mathbb F_v}$ and $\mathbb A_S:=\prod_{v\in S}\mathbb F_v\cdot\mathbb A^S$.
\begin{prop}\label{prop_analytischer_Hecke_allgemein}
 (\cite{zhang} Chapter~2.4)
Let $\mu$ be a character of $\mathbb A^\times/\mathbb F^\times$. Let $\phi\in L^2(\operatorname{GL}_2(\mathbb F)\backslash \operatorname{GL}_2(\mathbb A),\mu)$, and let $W_\phi$ be the Whittakerfunction of $\phi$ in some Whittaker model. Let $S$ be the finite  set of infinite places and of those finite places $v$ for which $\phi_v$ is not invariant under the maximal compact subgroup $\operatorname{GL}_2(\mathbf o_{\mathbb F_v})$. For $b\in \mathbb A^s\cap\mathbb A^\times$ define
\begin{equation*}
 H(b):=\left\{g\in M_2(\mathbb A^S)\mid\det(g)\mathbb A^S=b\mathbb A^S\right\}. 
\end{equation*}
Then the Hecke operator $\mathbf T_b$ is well defined for $g\in\operatorname{GL}_2(\mathbb A_S)$:
\begin{equation*}
 \mathbf T_bW_\phi(g):=\int_{H(b)}W_\phi(gh)~dh.
\end{equation*}
If $y\in \mathbb A^S$ and $(b,y_f)=1$, then
\begin{equation*}
 \mathbf T_bW_\phi(g\begin{pmatrix} y&0\\0&1\end{pmatrix})=\lvert b\rvert^{-1}W_\phi(g\begin{pmatrix} yb&0\\0&1\end{pmatrix}).
\end{equation*}
\end{prop}
That is, the operation of the (local) Hecke operator $\mathbf T_b$ on some Whittaker product is essentially translation by $b$:
\begin{equation}\label{hecke_auf_whittaker_prod}
 \mathbf T_bW(\xi,\eta)=\lvert b\rvert^{-2}W(b\xi,b\eta).
\end{equation}
%%%%%%%%%%%%%%%%%%%%%%%%%%%%%%%%%%%%%%%%%%%%%%%%%%%%%%%%%%%%%%%%%%%%%%%%%%%%%%%%%%%%%%%%%%%%%%%%%%%%%%%%%%%%%%%%%%%%%%%%%%%%%%%%%%%%%%%%%%%%%%%%%%%%%%%%%%%%%%%%%%%%%%%%%%%%%%%%%%%%%%%%%%%%%%%%%%%%%%%%%%%%%%%%%%%%%%%%%%%%%%%%%%%%%%%%%%%%%%%%%%
% Section: 
%%%%%%%%%%%%%%%%%%%%%%%%%%%%%%%%%%%%%%%%%%%%%%%%%%%%%%%%%%%%%%%%%%%%%%%%%%%%%%%%%%%%%%%%%%%%%%%%%%%%%%%%%%%%%%%%%%%%%%%%%%%%%%%%%%%%%%%%%%%%%%%%%%%%%%%%%%%%%%%%%%%%%%%%%%%%%%%%%%%%%%%%%%%%%%%%%%%%%%%%%%%%%%%%%%%%%%%%%%%%%%%%%%%%%%%%%%%%%%%%%%
\section{Characterisation of the local linking numbers}\label{section_charakterisierung}

Here, the local linking numbers are characterized as functions on $F^\times$. The characterizing properties are very near by those satisfied by the orbital integrals of \cite{jacquet}. Thus, establishing and proving proposition~\ref{eigenschaften_nonsplit} resp.~\ref{eigenschaften_split} is  influenced by the methods there.
Before stating these properties, two useful lemmas:

\begin{lem}\label{lemma1}
 Let $\phi\in\mathcal S(\chi,G)$.

(a) For each $y\in G$ there is an open set $V\ni y$ such that for all $g\in\operatorname{supp}(\phi)y^{-1}$ and all $\tilde y\in V$
\begin{equation*}
 \phi(g\tilde y)=\phi(gy).
\end{equation*}
(b) Let $C\subset G$ be compact. For each $g\in G$ there is an open set $U\ni g$ such that for all $\tilde g\in U$ and all $y\in TC$
\begin{equation}\label{lemma1_int_gleichung}
 \int_T\phi(t^{-1}\tilde gty)~dt=\int_T\phi(t^{-1}gty)~dt.
\end{equation}
\end{lem}
\begin{proof}[Proof of Lemma \ref{lemma1}]
(a) It is enough to prove the statement for $y=\operatorname{id}$. As $\phi$ is locally constant, for every $g\in G$ there is an open set $U_g\ni\operatorname{id}$ with $\phi(gU_g)=\phi(g)$. Let $C\subset G$ be compact such that $\operatorname{supp}\phi=TC$. Then one can cover $C\subset \cup gU_g$ by finitely many $gU_g$. Define $U$ to be the intersection of those $U_g$ to get $\phi(gU)=\phi(g)$ for all $g\in TC$.\\
(b) It is enough to prove the statement for $y\in C$ rather than $y\in TC$, as a factor $s\in T$ just changes the integral by a factor $\chi(s)$. By (a) there is an open set $V_y\ni y$ such that $\phi(t^{-1}gt\tilde y)=\phi(t^{-1}gty)$ for $\tilde y\in V_y$ and $t^{-1}gt\in\operatorname{supp}(\phi)y^{-1}$. Take finitely many $y\in C$ such that the $V_y$ cover $C$. It is enough to find open sets $U_y\ni g$ for these $y$ so that  eqn.~(\ref{lemma1_int_gleichung}) is fulfilled. Then $\cap U_y$ is an open set such that eqn.~(\ref{lemma1_int_gleichung}) is satisfied for all $y\in TC$. 
Write $g=g_1+\epsilon g_2$ and describe a neighborhood $U_y$ of $g$ by $k_1,k_2>0$ depending on $y$ and the obstructions  $\lvert \tilde g_i-g_i\rvert<k_i$, $i=1,2$, for $\tilde g$ lying in $U_y$. Write $t^{-1}\tilde gt=g_1+\epsilon g_2t\bar t^{-1}+(\tilde g_1-g_1)+\epsilon(\tilde g_2-g_2)t\bar t^{-1}$. As $\phi$ is locally constant, one can choose $k_1,k_2$ depending on $y$ such that
\begin{equation*}
 \phi(t^{-1}\tilde gt)=\phi((g_1+\epsilon g_2t\bar t^{-1})y)=\phi(t^{-1}gty).
\end{equation*}
These constants are independent from $t$ as $\lvert (\tilde g_2-g_2)t\bar t^{-1}\rvert=\lvert\tilde g_2-g_2\rvert$.
\end{proof}
\begin{lem}\label{lemma2}
 Let $\phi\in\mathcal S(F\oplus F)$.

(a) There are $A_1,A_2\in\mathcal S(F)$ such that
\begin{equation*}
 \int_{F^\times}\phi(a^{-1}y,a)~d^\times a =A_1(y) + A_2(y)v(y).
\end{equation*}
(b) Let $\eta$ be a nontrivial (finite) character of $F^\times$. Then there are $B_1,B_2\in\mathcal S(F)$ and $m\in \mathbb Z$ such that for $0\not= y\in\wp^m$
\begin{equation*}
 \int_{F^\times}\phi(a^{-1}y,a)\eta(a)~d^\times a = B_1(y)+B_2(y)\eta(y).
\end{equation*}
\end{lem}
\begin{proof}[Proof of Lemma \ref{lemma2}]
(a) Any $\phi\in\mathcal S(F\oplus F)$ is a finite linear combination of the following elementary functions: $\mathbf 1_{\wp^n}(a)\mathbf 1_{\wp^n}(b)$, $\mathbf 1_{x+\wp^n}(a)\mathbf 1_{\wp^n}(b)$, $\mathbf 1_{\wp^n}(a)\mathbf 1_{z+\wp^n}(b)$, $\mathbf 1_{x+\wp^n}(a)\mathbf 1_{z+\wp^n}(b)$ for suitable $n\in \mathbb Z$ and $v(x),v(z)>n$. It is enough to prove the statement for these elementary functions. One gets
\begin{equation*}
 \int_{F^\times}\mathbf 1_{\wp^n}(a^{-1}y)\mathbf 1_{\wp^n}(a)~d^\times a =\mathbf 1_{\wp^{2n}}(y)v(y\pi^{-2n+1})\operatorname{vol}^\times(\mathbf o_F^\times).
\end{equation*}
Thus, if $0\in\operatorname{supp}\phi$, then the integral has a pole in $y=0$, otherwise it hasn't:
\begin{equation*}
 \int_{F^\times}\mathbf 1_{x+\wp^n}(a^{-1}y)\mathbf 1_{\wp^n}(a)~d^\times a =\mathbf 1_{\wp^{v(x)+n}}(y)\operatorname{vol}^\times(1+\wp^{n-v(x)}),
\end{equation*}
\begin{equation*}
 \int_{F^\times}\mathbf 1_{\wp^n}(a^{-1}y)\mathbf 1_{z+\wp^n}(a)~d^\times a =\mathbf 1_{\wp^{v(z)+n}}(y)\operatorname{vol}^\times(1+\wp^{n-v(z)})
\end{equation*}
and
\begin{equation*}
 \int_{F^\times}\mathbf 1_{x+\wp^n}(a^{-1}y)\mathbf 1_{z+\wp^n}(a)~d^\times a =\mathbf 1_{xz(1+\wp^{m})}(y)\operatorname{vol}^\times(1+\wp^{m}),
\end{equation*}
where $m:=n-\operatorname{min}\{v(x),v(z)\}$.\\
(b) Similar computations to those of (a).
\end{proof}

In describing the properties of the local linking numbers, one has to distinguish between the case of a compact torus $T=F^\times \backslash K^\times$, i.e. $K/F$ is a field extension, and the case of a noncompact torus $T$, i.e. $K/F$ is split.

\begin{prop}\label{eigenschaften_nonsplit}
Let $K=F(\sqrt A)$ be a field extension and let $\omega$ be the associated quadratic character. Let $\phi,\psi\in \mathcal S(\chi,G)$. The local linking number $<\phi,\psi>_x$ is a function of $x\in F^\times$ having the following properties:

(a) $<\phi,\psi>_x$ is zero on the complement of $c\operatorname N$.

(b) $<\phi,\psi>_x$ is zero on a neighborhood of $1\in F^\times$.

(c) There is a locally constant function $A$ on a neighborhood $U$ of $0$ such that for all $0\not= x\in U$: $<\phi,\psi>_x=A(x)(1+\omega(cx))$.

(d) The behavior around infinity is described as follows: There is an open set $V\ni 0$ such that for all $x^{-1}\in V\cap c\operatorname N$
\begin{equation*}
 <\phi,\psi>_x=\delta(\chi^2=1)\chi_1(\frac{A}{c})\chi_1(x)\int_{T\backslash G}\phi(\epsilon y)\bar\psi(y)~dy.
\end{equation*}
Here the character $\chi_1$ of $F^\times$ is given by $\chi=\chi_1\circ\operatorname N$ if $\chi^2=1$. Especially, the local linking number vanishes in a neighborhood of infinity if $\chi^2\not= 1$.
\end{prop}

\begin{prop}\label{eigenschaften_split}
Let $K=F\oplus F$ be a split algebra. Let $\chi=(\chi_1,\chi_1^{-1})$ and let $\phi,\psi\in\mathcal S(\chi,G)$. The local linking number $<\phi,\psi>_x$ is a function of $x\in F^\times$ carrying the following properties:

(a) $<\phi,\psi>_x$ is zero on a neighborhood of $1\in F^\times$.

(b) $<\phi,\psi>_x$ is locally constant on $F^\times$.

(c) There is an open set $U\ni 0$ and locally constant functions $A_1,A_2$ on $U$ such that for $0\not= x\in U$: $<\phi,\psi>_x=A_1(x)+A_2(x)v(x)$.

(d) There is an open set $V\ni 0$ and locally constant functions $B_1,B_2$ on $V$ such that for $x^{-1}\in V$:
\begin{equation*}
 <\phi,\psi>_x=\left\{\begin{matrix}
                       \chi_1(x)(B_1(x^{-1})+B_2(x^{-1})v(x)), &\textrm{if } \chi_1^2=1\\
			\chi_1(x)B_1(x^{-1})+\chi_1^{-1}(x)B_2(x^{-1}), & \textrm{if } \chi^2\not=1
                      \end{matrix}.\right.
\end{equation*}
For $\chi_1^2=1$ the term $B_2(x^{-1})v(x)$ only occures   if $\operatorname{id}\in \operatorname{supp}\phi(\operatorname{supp}\psi)^{-1}$.
\end{prop}

\begin{proof}[Proof of Proposition~\ref{eigenschaften_nonsplit}]
(b) Assume $1\in c\operatorname N$, otherwise this property is trivial. One has to show that for all  $\gamma$ with $P(\gamma)\in U$, where $U$ is a sufficiently small neighborhood of $1$,
\begin{equation*}
 \int_{T\backslash G}\int_t\phi(t^{-1}\gamma ty)~dt\bar\psi(y)~dy = 0.
\end{equation*}
This is  done by showing that the even inner integral is zero. Let $C\subset G$ be compact such that $\operatorname{supp}\phi\subset TC$. Then $\phi$ obviously vanishes outside of $TCT$. It is enough to show that there is $k>0$ such that $\lvert P(\gamma)-1\rvert>k$ holds for all $\gamma\in TCT$. Assume there isn't such $k$. Let $(\gamma_i)_i$ be a sequence such that $P(\gamma_i)$ tends to $1$. Multiplying by elements of $T$ and enlarging $C$ occasionally (this is possible as $T$ is compact!), one can assume $\gamma_i=1+\epsilon t_i=z_ic_i$,
where $t_i\in T$, $c_i\in C$, $z_i\in Z$. Then $P(\gamma_i)=ct_i\bar t_i =1+a_i$, where $a_i\rightarrow 0$. We have $\det \gamma_i=1-ct_i\bar t_i=-a_i$ as well as $\det \gamma_i =z_i^2\det c_i$. As $C$ is compact, $(z_i)$ is forced to tend to zero. This implies $\gamma_i\rightarrow 0$ contradicting $\gamma_i=1+\epsilon t_i$.

(c) A $\gamma \in F^\times \backslash D^\times$ of trace zero has a representative of the form $\gamma=\sqrt A +\epsilon \gamma_2$ (by abuse of notation). Thus,
\begin{equation*}
 <\phi,\psi>_x=\int_{T\backslash G}\int_T\phi((\sqrt A +\epsilon\gamma_2 t\bar t^{-1})y)~dt~\bar\psi(y)~dy.
\end{equation*}
As $\phi\in\mathcal S(\chi,G)$, there exists an ideal $\wp_K^m$ of $K$ such that for all $y\in G$ and all $l\in \wp_K^m$ one has $\phi((\sqrt A+\epsilon l)y)=\phi(\sqrt A y)$. Let $x=P(\gamma)$ be near zero, i.e. $x$ belongs to an ideal $U$ of $F$ which is given by the obstruction that  $\frac{cl\bar l}{-A}\in U$ implies $l\in\wp_K^m$. For such $x$ one has
\begin{equation*}
 <\phi,\psi>_x=\operatorname{vol}_T(T)\chi(\sqrt A)\int_{T\backslash G}\phi(y)\bar\psi(y)~dy.
\end{equation*}
Taking into account that $x$ must not belong to the image of $P$ (see (a)), one gets
\begin{equation*}
 <\phi,\psi>_x=\frac{1}{2}\operatorname{vol}_T(T)\chi(\sqrt A)\bigl(\phi(y),\psi(y)\bigr)(1+\omega(cx)),
\end{equation*}
where $\bigl(\cdot,\cdot\bigr)$ is the $L^2$-scalar product.

(d) Again, let $\gamma=\sqrt A+\epsilon\gamma_2$ denote a preimage of $x$ under $P$. Then
\begin{equation*}
\int_T\phi(t^{-1}\gamma ty)~dt=\chi(\gamma_2)\int_T\phi((\sqrt A\gamma_2^{-1} +t^{-1}\bar t\epsilon)y)~dt.
\end{equation*}
As $\phi$ is locally constant, by Lemma~\ref{lemma1} there exists $k>0$ such that for $\lvert\gamma_2\rvert>k$ and for $y\in \operatorname{supp} \psi$ one has $\phi((\sqrt A\gamma_2^{-1} +t^{-1}\bar t\epsilon)y)=\phi(t^{-1}\bar t\epsilon y)$. Thus, for $\lvert x\rvert >\lvert cA^{-1}\rvert k^2$,
\begin{equation*}
 <\phi,\psi>_x=\chi(\gamma_2)\int_T\chi(t^{-1}\bar t)~dt\int_{T\backslash G}\phi(\epsilon y)\bar\psi(y)~dy.
\end{equation*}
As $\chi(t^{-1}\bar t)$ defines the trivial character of $T$ if and only if $\chi^2=1$, the statement follows by noticing that in this case $\chi(\gamma_2)=\chi_1(\frac{Ax}{c})$.
\end{proof}

\begin{proof}[Proof of Proposition~\ref{eigenschaften_split}]
Here $D^\times$ is isomorphic to $\operatorname{GL}_2(F)$, an obvious isomorphism is given by embedding $K^\times$ diagonally and sending $\epsilon$ to $\begin{pmatrix}0&1\\1&0\end{pmatrix}$.
Then $P$ is given by
\begin{equation*}
 P\begin{pmatrix}a&b\\c&d\end{pmatrix}=\frac{bc}{ad}.
\end{equation*}
The only value  not contained in the image of $P$ is $1$. A preimage of $x\not=1$ of trace zero is given by
\begin{equation*}
 \gamma=\begin{pmatrix}-1&x\\-1&1\end{pmatrix}.
\end{equation*}
(a) First, one shows that for $\phi\in\mathcal S(\chi,G)$ there is a constant $k>0$ such that for all $\gamma\in\operatorname{supp}\phi$: $\lvert P(\gamma)-1\rvert >k$. Using  Bruhat-Tits decomposition for $\operatorname{SL}_2(F)$, $G=\operatorname{PGL}_2(F)=TNN'\cup TNwN$, where $N$ is the group of uniponent upper triangular matrices, $N'$ its transpose and $w=\begin{pmatrix}0&-1\\1&0\end{pmatrix}$. Thus, there is $c>0$ such that
\begin{eqnarray*}
 \operatorname{supp}\phi &\subset&T\left\{\begin{pmatrix}1&u\\0&1\end{pmatrix}\begin{pmatrix}1&0\\v&1\end{pmatrix}\mid \lvert u\rvert <c, \lvert v\rvert <c\right\}\\
&&\bigcup T\left\{\begin{pmatrix}1&u\\0&1\end{pmatrix}w\begin{pmatrix}1&v\\0&1\end{pmatrix}\mid \lvert u\rvert <c, \lvert v\rvert <c\right\}.
\end{eqnarray*}
On the first set $P$ has the shape $P=\frac{uv}{1+uv}$. On the second one its shape is $P=\frac{uv-1}{uv}$. Thus, for all $\gamma\in\operatorname{supp}\phi$ one has $\lvert P(\gamma)-1\rvert \geq \operatorname{min}\{1,c^{-2}\}$.
Now, one shows that there even is a constant $k>0$ such that $\lvert P(\gamma)-1\rvert >k$ for all $\gamma y\in \operatorname{supp}\phi$ for all $y\in \operatorname{supp}\psi$. This implies that $<\phi,\psi>_x=0$ in the neighborhood $\lvert x-1\rvert <k$ of $1$. One knows there is such a constant $k_y$ for any $y\in \operatorname{supp}\psi$. By Lemma~\ref{lemma1}(a) this constant is valid for all $\tilde y$ in a neighborhood $V_y$. Modulo $T$ the support of $\psi$ can be covered by finitely many  those $V_y$. The minimum of the associated $k_y$ then is the global constant one was looking for.\\
(b) By Lemma~\ref{lemma1}(b), there is for every $x\in F^\times\backslash\{1\}$ a neighborhood $U_x$ such that for  all $y\in \operatorname{supp}\psi$ the inner integral
\begin{equation*}
 \int_T\phi(t^{-1}\gamma(\tilde x)ty)~dt
\end{equation*}
is constant in $\tilde x\in U_x$. Even more the hole local linking number is locally constant on $F^\times\backslash\{1\}$. That it is constant around $1$ as well was part~(a).\\
For (c) and (d) one regards the inner integral separately first. One has for representatives
\begin{eqnarray*}
 t^{-1}\gamma(x)t &=&\begin{pmatrix}a^{-1}&0\\0&1\end{pmatrix}\begin{pmatrix}-1&x\\-1&1\end{pmatrix}\begin{pmatrix}a&0\\0&1\end{pmatrix}\\
&=&\begin{pmatrix}(x-1)&0\\0&1\end{pmatrix}\begin{pmatrix} 1&\frac{x}{a(x-1)}\\0&1\end{pmatrix}\begin{pmatrix}1&0\\-a&1\end{pmatrix}\in K^\times NN'\\
&=&\begin{pmatrix}\frac{1-x}{a}&0\\0&-a\end{pmatrix}\begin{pmatrix}1&\frac{a}{x-1}\\0&1\end{pmatrix}w\begin{pmatrix}1&-a^{-1}\\0&1\end{pmatrix}\in K^\times NwN.
\end{eqnarray*}
As $\operatorname{supp}\phi$ is compact modulo $T$, the intersections $\operatorname{supp}\phi\cap NN'$ and $\operatorname{supp}\phi\cap NwN$ are compact. Write $\phi^y:=\rho(y)\phi$ as a sum $\phi^y=\phi_1^y+\phi_2^y$, $\phi_i^y\in\mathcal S(\chi,G)$, with $\operatorname{supp}\phi_1^y\subset TNN'$ and $\operatorname{supp}\phi_2^y\subset TNwN$. Using the transformation under $T$ by $\chi$, one can actually regard $\phi_i^y$, $i=1,2$, as functions on $F\oplus F$ identifying $N$ with $F$. Thus, $\phi_i^y\in \mathcal S(F\oplus F)$. For the inner integral one gets the formula
\begin{eqnarray}
 \int_T\phi(t^{-1}\gamma(x)ty)~dt &=&\chi_1(x-1)\int_{F^\times}\phi_1^y(\frac{x}{a(x-1)},-a)~d^\times a\label{inners_integral_umformung}\\
&+&\chi_1(1-x)\int_{F^\times}\chi_1(a^{-2})\phi_2^y(\frac{a}{x-1},-a^{-1})~d^\times a.\nonumber
\end{eqnarray}
(c) 
One has $\chi_1(x-1)=\chi_1(-1)$ if $x\in\wp^{c(\chi_1)}$, the leader of $\chi_1$. By lemma~\ref{lemma2}, the first integral of (\ref{inners_integral_umformung}) for small $x$ equals 
\begin{equation*}
 A_1(\frac{x}{x-1})+A_2(\frac{x}{x-1})v(\frac{x}{x-1}),
\end{equation*}
where $A_1,A_2$ are locally constant functions on a neighborhood of zero depending on $y$. $\tilde A_i(x):=A_i(\frac{x}{x-1})$ are locally constant functions on a neighborhood $U_1$ of zero as well.
The second integral of (\ref{inners_integral_umformung})  is constant on a neighborhood $U_2$ of $x=0$ depending on $y$, as $\phi_2^y$ is locally constant for $(x-1)^{-1}\rightarrow -1$. Thus, the complete inner integral can be expressed on $U_y:=\wp^{c(\chi_1)}\cap U_1\cap U_2$ as
\begin{equation*}
 A_y(x):= \tilde A_1(x)+\tilde A_2(x)v(x) +B.
\end{equation*}
By lemma~\ref{lemma1}(a), there is a neighborhood $V_y$ of $y$ where the inner integral has the same value. Take $V_y$ that small that $\psi$ is constant there, too, and cover $\operatorname{supp}\psi$ modulo $T$ by finitely many such $V_y$, i.e.  $y\in I$, $\lvert I\rvert$ finite. The local linking number for $x\in U=\cap_{y\in I}U_y$ now is computed as
\begin{equation*}
 <\phi,\psi>_x=\sum_{y\in I}\operatorname{vol}_{T\backslash G}(TV_y)\bar\psi(y)A_y(x).
\end{equation*}
That is, there are locally constant functions $B_1,B_2$ on $U$ such that for $x\in U$
\begin{equation*}
 <\phi,\psi>_x=B_1(x)+B_2(x)v(x).
\end{equation*}
(d) Let $x^{-1}\in\wp^{c(\chi_1)}$, then $\chi_1(x-1)=\chi(x)$. As $\phi_1^y$ is locally constant, the first integral of (\ref{inners_integral_umformung}) equals a locally constant function $A_1(x^{-1})$ for $x^{-1}\in U_1$, a neighborhood of zero depending on $y$. For the second integral, one has to differentiate between $\chi_1^2=1$ or not. To start with, let $\chi_1^2\not=1$. Applying lemma~\ref{lemma2}(b) for $\eta=\chi_1^2$, one gets locally constant functions $A_2,A_3$ on a neigborhood $U_2$ of zero depending on $y$ such that the second integral equals $A_2(x^{-1})+\chi_1^2(x^{-1})A_3(x^{-1})$. Thus, for fixed $y$ the inner integral for $x^{-1}\in U_y=U_1\cap U_2\cap\wp^{c(\chi_1)}$ is given by
\begin{equation*}
 A_y(x):=\int_T\phi^y(t^{-1}\gamma(c)t)~dt=\chi_1(x)\bigl(A_1(x^{-1})+A_2(x^{-1})+A_3(x^{-1})\chi_1^{-1}(x)\bigr).
\end{equation*}
Proceeding as in (c), one gets the assertion. \\
Now, let $\chi_1^2=1$. By lemma~\ref{lemma2}(a), one has locally constant functions $A_2,A_3$ on a neighborhood $U_2$ of zero such that for $x^{-1}\in U$ the second integral of (\ref{inners_integral_umformung}) is given by $A_2(x^{-1})+A_2(x^{-1})v(x)$. Thus, for $x^{-1}\in U_y:=U_1\cap U_2\cap\wp^{c(\chi_1)}$ the inner integral is given by
\begin{equation*}
 A_y(x):=\chi_1(x)\bigl( A_1(x^{-1})+A_2(x^{-1})+A_3(x^{-1})v(x)\bigr).
\end{equation*}
The term $A_3(x^{-1})v(x)$  by lemma~\ref{lemma2}(a) is obtainted from functions $\phi_2^y(a,b)$ having the shape $\mathbf 1_{\wp^n}(a)\mathbf 1_{\wp^n}(b)$ around zero. Those function can only occure if $y$ is contained in $\operatorname{supp}\phi$.
Again proceeding as in part (c), the local linking number for $x^{-1}$ in a sufficently small neighborhood $U$ of zero is 
\begin{equation*}
 <\phi,\psi>_x=\chi_1(x)\bigl(B_1(x^{-1})+B_2(x^{-1})v(x)\bigr),
\end{equation*}
where $B_1,B_2$ are locally constant on $U$ and $B_2$ doesn't vanish only if $\operatorname{id}\in (\operatorname{supp}\phi)(\operatorname{supp}\psi)^{-1}$.
\end{proof}
\begin{prop}\label{charakterisierung_der_LLN}
The properties (a) to (d) of proposition~\ref{eigenschaften_nonsplit} resp. \ref{eigenschaften_split} characterize the local linking numbers. That is, given a function on $F^\times$ satisfying these properties, one can realize it as a local linking number.
\end{prop}
The  proof of proposition~\ref{charakterisierung_der_LLN} is totally constructive. Let us first describe the appoach in general before going into detail in the case of a field extension $K/F$. The case of a split algebra $K=F\oplus F$ will be omitted and referred to~\cite{diss} chap.~2, as the computations there are quite similar to those presented here and straighed forward after them. 

Firstly,  choose a describtion of the function $H$ satisfying the properties (a) to (d) in the following manner
\begin{equation*}
 H(x)=\mathbf 1_{c\operatorname{N}}(x)\bigl(A_0(x)\mathbf 1_{V_0}(x)+A_1(x)\mathbf 1_{V_1}(x) +\sum_{j=2}^{M} H(x_j)\mathbf 1_{V_j}(x)\bigr),
\end{equation*}
where for $j=2,\dots, M$
\begin{equation*}
 V_j=x_j(1+\wp_F^{n_j})
\end{equation*}
are open sets in $F^\times$ on which $H$ is  constant. Further,
\begin{equation*}
 V_0=\wp_F^{n_0} \quad \textrm{ resp. } \quad V_1=F\backslash \wp^{-n_1}
\end{equation*}
are neighborhoods of $0$ (resp. $\infty$) where $H$ is characterized by $A_0$ (resp. $A_1$) according to property (c) (resp. (d)). One can assume without loss of generality that $n_j>0$ for $j=0, \dots , M$ and $V_i\cap V_j=\emptyset$ for $i\not= j$.\\
Secondly, construct functions $\phi_j$, $j=0,\dots, M$, and one function $\psi$ in $\mathcal S(\chi,G)$ such that $\operatorname{supp}\phi_i\cap\operatorname{supp}\phi_j=\emptyset$ if $i\not=j$ and such that
\begin{equation*}
 <\phi_j,\psi>_x=H(x_j)\mathbf 1_{V_j}(x)\quad\textrm{ resp. } <\phi_j,\psi>_x=A_j(x)\mathbf 1_{V_j}(x).
\end{equation*}
There is essentially one possibility to construct such functions in $\mathcal S(\chi,G)$: Take a compact open subset $C$ of $G$ which is {\bf fundamental} for $\chi$, i.e. if $t\in T$ and $c\in C$ as well as $tc\in C$, then $\chi(t)=1$. Then the function $\phi =\chi\cdot\mathbf 1_C$ given by $\phi(tg)=\chi(t)\mathbf 1_C(g)$ is well defined in $\mathcal S(\chi,G)$ with support $TC$.\\
The function $\psi$ is now chosen as $\psi=\chi\cdot \mathbf 1_U$, where $U$ is a compact open subgroup of $G$ that small that for $j=0,\dots,M$
\begin{equation*}
 P(P^{-1}(V_j)U)= V_j\cap c\operatorname N.
\end{equation*}
For $j\geq 2$ now take $C_j\subset P^{-1}(V_j)$ compact such that $C_jU$ is fundamental and $P(C_jU)=V_j$ and define $\phi_j:=H(x_j)\cdot\chi\cdot\mathbf 1_{C_jU}$.
The stalks of zero and infinity are constructed in a similar manner.\\
Thirdly, as the local linking numbers are linear in the first component and as the supports of the $\phi_j$ are disjoint by construction, one gets
\begin{equation*}
 H(x)=<\sum_{j=0}^{M}\phi_j,\psi>_x.
\end{equation*}

\begin{proof}[Proof of Proposition~\ref{charakterisierung_der_LLN} in the case $K$ a field]
Let $K=F(\sqrt A)$ be a quadratic field extension. Let the function $H$ satisfying (a) to (d) of Prop.~\ref{eigenschaften_nonsplit} be given as
\begin{equation*}
 H(x)=\mathbf 1_{c\operatorname N}(x)\bigl(A_0(x)\mathbf 1_{V_0}(x)+A_1(x)\mathbf 1_{V_1}(x) +\sum_{j=2}^{M} H(x_j)\mathbf 1_{V_j}(x)\bigr),
\end{equation*}
where
\begin{eqnarray*}
 && V_0=\wp^{n_0} \textrm{ and } A_0(x)=a_0,\\
 && V_1=F\backslash \wp^{-n_1} \textrm{ and } A_1(x)=\left\{\begin{matrix}\chi_1(x)a_1, \textrm{ if } \chi^2=1\\
                                                             0, \textrm{ if } \chi^2\not=1\\
                                                            \end{matrix},\right.\\
 && V_j=x_j(1+\wp^{n_j}) \textrm{ for } j=2,\dots, M,
\end{eqnarray*}
with $a_0, a_1, H(x_j)\in\mathbb C$, and $n_j>0$ for $j=0,\dots,M$. One can further assume
\begin{eqnarray*}
 n_0-v(\frac{c}{A})>0, \quad n_1+v(\frac{c}{A})>0 \textrm{ and both even,}
\end{eqnarray*}
as well as $V_i\cap V_j=\emptyset$ for $i\not= j$. One defines
\begin{eqnarray*}
 \tilde n_0 &=& \left\{\begin{matrix}\frac{1}{2}(n_0-v(\frac{c}{A})), & \textrm{if } K/F \textrm{ unramified}\\
                        n_0-v(\frac{c}{A}), &\textrm{if } K/F \textrm{ ramified}
                       \end{matrix},\right.
\end{eqnarray*}
\begin{eqnarray*}
 \tilde n_1 &=& \left\{\begin{matrix}\frac{1}{2}(n_1+v(\frac{c}{A})), & \textrm{if } K/F \textrm{ unramified}\\
                        n_1+v(\frac{c}{A}), &\textrm{if } K/F \textrm{ ramified}
                       \end{matrix},\right.
\end{eqnarray*}
as well as for $j=2,\dots,M$
\begin{equation*}
 \tilde n_j = \left\{\begin{matrix}n_j, & \textrm{if } K/F \textrm{ unramified}\\
                        2n_j, &\textrm{if } K/F \textrm{ ramified}
                       \end{matrix}.\right.
\end{equation*}
Then $\operatorname N(1+\wp_K^{\tilde n_j})=1+\wp_F^{n_j}$, $j\geq 2$, and $\wp_K$  the prime ideal of $K$. Define
\begin{equation*}
 U:= 1+\wp_K^k+\epsilon\wp_K^m,
\end{equation*}
where $k>0$ and $m>0$ are chosen such that
\begin{eqnarray}
&& k\geq c(\chi),\quad m\geq c(\chi)\nonumber\\
&& k\geq \tilde n_j,\quad m\geq \tilde n_j+1, \textrm{ for } j=0,\dots,M,\nonumber\\
&&m\geq c(\chi)+1-\frac{1}{2}v(x_j), \textrm{ for } j=2,\dots,M,\nonumber\\
&&m\geq \tilde n_j+1+\frac{1}{2}\lvert v(x_j)\rvert, \textrm{ for } j=2,\dots,M.\label{k-m-bedingungen}
\end{eqnarray}
As $k,m>0$ and $k,m\geq c(\chi)$, $U$ is fundamental. Define 
\begin{equation*}
 \psi:=\chi\cdot\mathbf 1_U.
\end{equation*}
Now realize the stalks for $x_j$, $j\geq 2$, as local linking numbers. To begin with, let $\sqrt A +\epsilon\gamma_j$ be a preimage of $x_j$, i.e.
\begin{equation*}
 P(\sqrt A+\epsilon\gamma_j)=\frac{c\operatorname N(\gamma_j)}{-A}=x_j.
\end{equation*}
Then the preimage of $V_j$ is given by
\begin{equation*}
 P^{-1}(V_j)=T\bigl(\sqrt A+\epsilon\gamma_j(1+\wp_K^{\tilde n_j})\bigr)T=T\bigl(\sqrt A+\epsilon\gamma_j(1+\wp_K^{\tilde n_j})\operatorname N_K^1\bigr).
\end{equation*}
Let $C_j:=\sqrt A +\epsilon\gamma_j(1+\wp_K^{\tilde n_j})\operatorname N_K^1$ and look at the compact open set $C_jU$,
\begin{equation*}
 C_jU=\sqrt A(1+\wp_K^k)+c\bar \gamma_j\wp_K^m+\epsilon\bigl(\gamma_j(1+\wp_K^k+\wp_K^{\tilde n_j})\operatorname N_K^1+\sqrt A\wp_K^m\bigr).
\end{equation*}
Due to the choices (\ref{k-m-bedingungen}), $C_jU$ is fundamental. To prove this, one has to check that if $t\in T$, $c\in C_j$ and $tc\in C_jU$, then $\chi(t)=1$ (observe that $U$ is a group). So, let
\begin{equation*}
 tc=t\sqrt A +\epsilon\bar t\gamma_j(1+\pi_K^{\tilde n_j}c_1)l\in C_jU.
\end{equation*}
The first component forces $t\in 1+\wp_K^k+\frac{c}{A}\bar\gamma_j\wp_K^m$. For those $t$, the choices (\ref{k-m-bedingungen}) imply $\chi(t)=1$. For the image $P(C_jU)$ one finds again by (\ref{k-m-bedingungen})
\begin{equation*}
 P(C_jU)=\frac{c\operatorname N(\gamma_j)\operatorname N(1+\wp_K^k+\wp_K^{\tilde n_j}+\wp_K^m\frac{\sqrt A}{\gamma_j})}{-A\operatorname N(1+\wp_K^k+\frac{c}{\sqrt A}\bar\gamma_j\wp_K^m)}=V_j.
\end{equation*}
The functions $\phi_j:= \chi\cdot\mathbf 1_{C_jU}\in\mathcal S(\chi,G)$ are now well defined. Let us compute the local linking number
\begin{equation*}
 <\phi_j,\psi>_x=\int_{T\backslash G}\int_T\phi_j(t^{-1}\gamma(x)ty)~dt~\bar\psi(y)~dy.
\end{equation*}
The integrand doesn't vanish only if there is $s\in K^\times$ such that
\begin{equation*}
 st^{-1}\gamma(x)t=s\sqrt A+\epsilon\bar s\gamma_2(x)t\bar t^{-1}\in C_jU.
\end{equation*}
The first component inplies $s\in 1+\wp_K^{\tilde n_j}$. The second component implies $\gamma_2(x)\in\gamma_j(1+\wp_K^{\tilde n_j})\operatorname N_K^1$, which is equivalent to $\gamma(x)\in C_jU$ or $x\in V_j$. In this case one can take $s=1$ and gets
%\begin{eqnarray*}
% <\phi_j,\psi>_x&=&\mathbf 1_{V_j}(x)\int_{T\backslash G}\int_T 1~dt~\bar\psi(y)~dy\\
%&=& \mathbf 1_{V_j}(x)\operatorname{vol}_T(T)\operatorname{vol}_G(U).
%\end{eqnarray*}
\begin{equation*}
 <\phi_j,\psi>_x=\mathbf 1_{V_j}(x)\int_{T\backslash G}\int_T 1~dt~\bar\psi(y)~dy= \mathbf 1_{V_j}(x)\operatorname{vol}_T(T)\operatorname{vol}_G(U).
\end{equation*}
Normalizing $\tilde \phi_j:=\frac{H(x_j)}{\operatorname{vol}_T(T)\operatorname{vol}_G(U)}\phi_j$, one finally gets
$%\begin{equation*}
 H\vert_{V_j}(x)=<\tilde\phi_j,\psi>_x$.
%\end{equation*}
\\
Now regard the stalk of zero. One findes $P(\sqrt A+\epsilon\wp_K^{\tilde n_0})=\wp_F^{n_0}\cap c\operatorname N$.
Define $C_0:=\sqrt A+\epsilon\wp_K^{\tilde n_0}$. The preimage $P^{-1}(V_0)$ equals $TC_0T=TC_0$. The set open and compact set $C_0U$ is easily seen to be fundamental and to satisfy $P(C_0U)=V_0\cap cN$.
 Define $\phi_0:=\chi\cdot\mathbf 1_{C_0U}$ and compute the local linking number $<\phi_0,\psi>_x$. Again, this doesn't vanish only if there is $s\in K^\times$ such that
\begin{equation*}
 st^{-1}\gamma(x)t=s\sqrt A+\epsilon \bar s\gamma_2(x)t\bar t^{-1}\in C_0U.
\end{equation*}
This forces $\gamma_2(x)\in\wp_K^{\tilde n_0}$. Assuming this, one can take $s=1$ and gets
\begin{equation*}
 <\phi_0,\psi>_x=\mathbf 1_{V_0\cap c\operatorname N}(x)\operatorname{vol}_T(T)\operatorname{vol}_G(U).
\end{equation*}
That is, $H\vert_{V_0}(x)=a_0=<\frac{a_0}{\operatorname{vol}_T(T)\operatorname{vol}_G(U)}\phi_0,\psi>_x$.\\
It remains to construct the stalk of infinity. One can assume that $\chi^2=1$, as otherwise the function vanishes for big $x$. Thus, $\chi=\chi_1\circ \operatorname N$. The preimage of $V_1=F\backslash\wp_F^{-n_1}$ is given by
\begin{equation*}
 P^{-1}(V_1)=T\bigl(\sqrt A+\epsilon(\wp_K^{\tilde n_1})^{-1}\bigr)T=T\bigl(\sqrt A\wp_K^{\tilde n_1}+\epsilon\operatorname N_K^1\bigr).
\end{equation*}
Take $C_1=\sqrt A\wp_K^{\tilde n_1}+\epsilon\operatorname N_K^1$ to get 
a fundamental compact open set
%a compact open set
\begin{equation*}
 C_1U=\sqrt A\wp_K^{\tilde n_1}+c\wp_K^m+\epsilon\bigl(\operatorname N_K^1(1+\wp_K^k)+\sqrt A\wp_K^{m+\tilde n_1}\bigr),
\end{equation*}
By the choices (\ref{k-m-bedingungen}) one gets $P(C_1U)=V_1\cap c\operatorname N$.
Taking $\phi_1:=\chi\cdot\mathbf 1_{C_1U}$ this time, we get
$ H\vert_{V_1}(x)=\frac{a_1}{\operatorname{vol}_T(T)\operatorname{vol}_G(U)}<\phi_1,\psi>_x$.
\end{proof}
%%%%%%%%%%%%%%%%%%%%%%%%%%%%%%%%%%%%%%%%%%%%%%%%%%%%%%%%%%%%%%%%%%%%%%%%%%%%%%%%%%%%%%%%%%%%%%%%%%%
%%%
%%%%%%%%%%%%%%%%%%%%%%%%%%%%%%%%%%%%%%%%%%%%%%%%%%%%%%%%%%%%%%%%%%%%%%%%%%%%%%%%%%%%%%%%%%%%%%%%%%%%%
\section{A Matching}\label{section_matching}
Having characterized the local linking numbers, we can easily compare them to the Whittaker products of Definition~\ref{def_whittaker_prod}. For this, we use the parametrization $\xi=\frac{x}{x-1}$ rather than $x$ itself. The properties of the local linking numbers (Propositions~\ref{eigenschaften_nonsplit} and \ref{eigenschaften_split}) transform accordingly. For example, the property of vanishing around $x=1$ means that the local linking numbers as functions in $\xi$ have compact support. The behavior around infinity is replaced by the behavior around $\xi=1$.

The reason why the parametrization $\xi$ is not used all over this paper is simply that it was more comfortable to do the calculations using the coordinate $x$. In view of section~\ref{section_translation}, this is reasonable.
\begin{thm}\label{satz_matching}
 The local linking numbers and the Whittaker products match, that is: If $\eta=1-\xi$ and if $\omega(-\xi\eta)=(-1)^{\delta(D)}$, then
\begin{eqnarray*}
&&\Bigl\{\lvert\xi\eta\rvert^{\frac{1}{2}}<\phi,\psi>_{x=\frac{\xi}{\xi-1}}\Big\vert \phi,\psi\in\mathcal S(\chi,G)\Bigr\} =\\
&&\hspace*{3cm} \Bigl\{W_\theta(\eta)W_E(\xi)\Big\vert W_\theta\in\mathcal K(\Pi(\chi)),W_E\in\mathcal K(\Pi(1,\omega))\Bigr\}.
\end{eqnarray*}
\end{thm}
Recall the definition~(\ref{def_delta(D)}) of $\delta(D)$. Notice that the term $\omega(-\xi\eta)$ is just $\omega(x)$.
\begin{proof}[Proof of Theorem~\ref{satz_matching}]
 The Whittaker products are products of Kirillov functions characterized in Propositions~\ref{prop_characterisierung_theta_functionen} and \ref{prop_characterisierung_eisenstein_functionen}. Comparing their properties  to those of the local linking numbers (Propositions~\ref{eigenschaften_nonsplit} resp. \ref{eigenschaften_split}) yields the Proposition. For example, by Prop.~\ref{prop_characterisierung_theta_functionen} for $K/F$ split and $\chi_1^2\not=1$, the Whittaker products for $\xi\to 1$ ($\eta\to 0$) are given by
\begin{equation*}
 \lvert\xi\eta\rvert^{\frac{1}{2}}\left(a_1\chi_1(\eta)+a_2\chi_1^{-1}(\eta)\right),
\end{equation*}
which corresponds to Prop~\ref{eigenschaften_split}~(d). For $\xi\to 0$ ($\eta\to 1$) we apply Prop.~\ref{prop_characterisierung_eisenstein_functionen}: The Whittaker products have the shape
 $ \lvert\xi\eta\rvert^{\frac{1}{2}}\left(a_1+a_2v(\xi)\right)$.
This is property~(c) of Prop.~\ref{eigenschaften_split}. Away from $\xi\to 1$ and $\xi\to 0$, the Whittaker products are locally constant with compact support. This is equivalent to (a) and (b) of Prop.~\ref{eigenschaften_split}.
\end{proof}

%%%%%%%%%%%%%%%%%%%%%%%%%%%%%%%%%%%%%%%%%%%%%%%%%%%%%%%%%%%%%%%%%%%%%%%%%%%%%
%
% %
% %
% %
% %
% %
% 
% %%
% %%%%%%%%%%%%%%%%%%%%%%%%%%%%%%%%%%%%%%%%%%%%%%%%%%%%%%%%%%%%%%%%%%%%%%%%%%%%%%%%
\section{Translated linking numbers}\label{section_translation}

In the remaining, the quaternion algebra $D$  is assumed to be split, that is $G=F^\times\backslash D^\times$ is isomorphic to the projective group $\operatorname{PGL}_2(F)$.
The aim is to give an operator on the local linking numbers realizing the Hecke operator on the analytic side. 
As the analytic Hecke operator essentially is given by translation by $b\in F^\times$ (Proposition~\ref{prop_analytischer_Hecke_allgemein}), the first  candidate for this study surely is the translation by $b$, i.e.
\begin{equation*}%\label{gl_lln_transl_allg}
 <\phi,\begin{pmatrix}b&0\\0&1\end{pmatrix}.\psi>_x=\int_{T\backslash G}\int_T\phi(t^{-1}\gamma(x)ty)~dt~\bar\psi(y\begin{pmatrix}b&0\\0&1\end{pmatrix})~dy.
\end{equation*}
Let
\begin{equation}\label{gl_inneres_int_def}
 I_\phi(y) =\int_T\phi(t^{-1}\gamma(x)ty)~dt
\end{equation}
be the inner integral of this translated local linking number.
It will be seen eventually that this translation does not realize the Hecke operator completely but that there are operators made up from it which do.

In studying this translation the difference between the case of a compact torus and the case of a noncompact one becomes cucial. While one can describe what is going on in the compact case in a few lines (at least if one fixes $x$, viewing the translated linking number as a function of $b$ alone), our approach in the noncompact case would blast any article because of computational overflow. This case will be sketched here and made clear by examples. The complete computations are done in~\cite{diss}.

What is more, we have to reduce ourselves to the case in which the first variable $x$ is fixed. Again, the reason for that is manageability of computations. But at least we get some hints of what is going on in two variables by examples.

\subsection{The compact case}\label{section_compact}
Let $K=F(\sqrt A)$ be a quadratic field extension of $F$. That is, the torus $F^\times\backslash K^\times$ is compact.
As functions $\phi\in\mathcal S(\chi,G)$ have compact support modulo $T$, they have compact support absolutely.  As $\phi$ is locally constant, the inner integral $I_\phi$  (\ref{gl_inneres_int_def}) is. Further, $T\gamma(x)T\operatorname{supp}\phi$ is compact, and left translation by $t'\in T$ yields $I_\phi(t'y)=\chi(t')I_\phi(y)$. Thus,  $I_\phi$ itself is  a function belonging to $\mathcal S(\chi,G)$.
Choose the following isomorphism of $D^\times=(K+\epsilon K)^\times$ with $\operatorname{GL}_2(F)$:
\begin{eqnarray*}
 \epsilon&\mapsto&\begin{pmatrix}0&-A\\1&0\end{pmatrix}\\
K^\times\ni t=a+b\sqrt A&\mapsto& \begin{pmatrix}a&bA\\b&a\end{pmatrix}
\end{eqnarray*}
\begin{lem}\label{lemma_mirobolische_zerlegung}
Let $M=\left\{\begin{pmatrix}y_1&y_2\\0&1\end{pmatrix}\mid y_1\in F^\times,y_2\in F\right\}$ be the mirabolic subgroup of the standard Borel group. The mapping $K^\times\times M\rightarrow \operatorname{GL}_2(F)$, $(t,m)\mapsto k\cdot m$, is a homeomorphism.
\end{lem}
\begin{proof}[Proof of Lemma~\ref{lemma_mirobolische_zerlegung}]
 One has to show that
\begin{equation*}
 \operatorname{GL}_2(F)\ni \begin{pmatrix}a&b\\c&d\end{pmatrix}=\begin{pmatrix}\alpha&\beta A\\\beta&\alpha\end{pmatrix}\begin{pmatrix}z&y\\0&1\end{pmatrix}
\end{equation*}
has exatly one solution $(\alpha,\beta,z,y)$ satisfying $(\alpha,\beta)\not=(0,0)$ and $z\not=0$. The first column yields $\alpha=az^{-1}$ and $\beta=cz^{-1}$. The second column now reads
\begin{equation*}
 \begin{pmatrix}b\\d\end{pmatrix}=\begin{pmatrix}z^{-1}cA+yz^{-1}a\\z^{-1}a+yz^{-1}c\end{pmatrix},
\end{equation*}
which is a system of linear equations in the variables $w_1=yz^{-1}$ and $w_2=z^{-1}$ with determinant $a^2-c^2A\not=0$. Thus, there is exactly one solution $(w_1,w_2)$. As $w_2=z^{-1}=0$ would imply that the columns $\begin{pmatrix}a\\c\end{pmatrix}$ and $\begin{pmatrix}b\\d\end{pmatrix}$ are linearly dependend, $z$ and $y$ are uniquely determined and so are $\alpha$ and $\beta$. The resulting continous mapping $K^\times\times M\rightarrow \operatorname{GL}_2(F)$ is bijective.  Being provided by polynomial equations its inverse is continous, too.
\end{proof}
The group $M$ is not unimodular anymore but carries a right invariant Haar measure $d^\times y_1~dy_2$, where $d^\times y_1$ resp. $dy_2$ are nontrivial compatible Haar measures on $F^\times$ resp. $F$. We normalize the quotient measure $dy$ on $T\backslash G$ so that
%\begin{equation*}
$ dy=d^\times y_1~dy_2$.
%\end{equation*}
By Lemma~\ref{lemma_mirobolische_zerlegung}, any $\phi\in\mathcal S(\chi,G)$ can be identified with a function in $\mathcal S(F^\times\times F)$,
\begin{equation*}
 \phi(y_1,y_2):=\phi\begin{pmatrix}y_1&y_2\\0&1\end{pmatrix}.
\end{equation*}
$\phi$ being locally constant with compact support, there are finitely many points $(z_1,z_2)\in F^\times\times F$ and  $m>0$ such that
\begin{equation*}
 \phi(y_1,y_2)=\sum_{(z_1,z_2)}\phi(z_1,z_2)\mathbf 1_{z_1(1+\wp^m)}(y_1)\mathbf 1_{z_2+\wp^m}(y_2).
\end{equation*}
Applying this for $I_\phi$ and $\psi$,
\begin{equation*}
 I_\phi(y_1,y_2)=\sum_{(z_1,z_2)}I_\phi(z_1,z_2)\mathbf 1_{z_1(1+\wp^m)}(y_1)\mathbf 1_{z_2+\wp^m}(y_2),
\end{equation*}
\begin{equation*}
 \psi(y_1,y_2)=\sum_{(w_1,w_2)}\psi(w_1,w_2)\mathbf 1_{w_1(1+\wp^m)}(y_1)\mathbf 1_{w_2+\wp^m}(y_2),
\end{equation*}
we compute the translated local linking number
\begin{align*}
<\phi,\begin{pmatrix}b&0\\0&1\end{pmatrix}.\psi>_x&=\int_{T\backslash G}I_\phi(y)\bar\psi(y\begin{pmatrix}b&0\\0&1\end{pmatrix})~dy\\
&=\sum_{(z_1,z_2),(w_1,w_2)}I_\phi(z_1,z_2)\bar\psi(w_1,w_2)\mathbf 1_{z_2+\wp^m}(w_2)\mathbf 1_{\frac{w_1}{z_1}(1+\wp^m)}(b)\\
&\quad\quad\cdot\operatorname{vol}^\times(1+\wp^m)\operatorname{vol}(\wp^m).
\end{align*}
We have proved:
\begin{prop}\label{prop_translatiert_kompakt}
Let $T$ be compact.
 For fixed $x$, the translated local linking number $<\phi,\begin{pmatrix}b&0\\0&1\end{pmatrix}\psi>_x$ is a locally constant function of $b\in F^\times$ with compact support.
\end{prop}
As the behavior of the translated local linking numbers as functions in $b$ as well as in $x$ is not studied here completely,  an example is in due. This example will be of further use later. Its calculation is banished to Appendix~A.
\begin{eg}\label{bsp_lln_translatiert_kompakt} 
 Let $K/F$ be an unramified field extension and let $\chi=1$. Then $\phi=\chi\cdot\mathbf 1_{\operatorname{GL}_2(\mathbf o_F)}$ is well defined in $\mathcal S(\chi,G)$ and
\begin{align*}
 &<\phi,\begin{pmatrix}b&0\\0&1\end{pmatrix}\phi>_x\cdot\operatorname{vol}^{-1}=\\
&\quad \quad\mathbf 1_{\operatorname N\backslash(1+\wp)}(x)\mathbf 1_{\mathbf o_F^\times}(b)+\mathbf 1_{1+\wp}(x)\bigl(\mathbf 1_{(1-x)\mathbf o_F^\times}(b)+\mathbf 1_{(1-x)^{-1}\mathbf o_F^\times}(b)\bigr)q^{-v(1-x)},
\end{align*}
where $\operatorname{vol}:=\operatorname{vol}_T(T)\operatorname{vol}^\times(\mathbf o_F^\times)\operatorname{vol}(\mathbf o_F)$.
\end{eg}
%%%%%%%%%%%%%%%%%%%%%%%%%%%%%
%%%
% %%%
%%%
%%%%%%%%%%%%%%%%%%%%%%%%%%%%%%%%
\subsection{The noncompact case}\label{section_nonkompakt}
Let $K=F\oplus F$ be a split algebra. The character $\chi$ is of the form $\chi=(\chi_1,\chi_1^{-1})$ for a character $\chi_1$ of $F^\times$. %We use the obvious isomorphism $D^\times\to\operatorname{GL}_2(F)$.
As in the proof of Proposition~\ref{eigenschaften_split}, $G=TNN'\cup TNwN$. Both of these open subset are invariant under right translation by $\begin{pmatrix}b&0\\0&1\end{pmatrix}$.
Choose coset representatives for $T\backslash TNN'$ of the form
\begin{equation*}
 y=\begin{pmatrix}1&y_2\\0&1\end{pmatrix}\begin{pmatrix}1&0\\y_3&1\end{pmatrix}
\end{equation*}
as well as coset representatives for $T\backslash TNwN$ of the form
\begin{equation*}
 y=\begin{pmatrix}1&y_1\\0&1\end{pmatrix}w\begin{pmatrix}1&0\\y_4&1\end{pmatrix}.
\end{equation*}
Any function $\psi\in\mathcal S(\chi,G)$ can be split into a sum $\psi=\psi_1+\psi_2$, $\psi_i\in\mathcal S(\chi,G)$, with $\operatorname{supp} \psi_1\subset TNN'$ (resp. $\operatorname{supp} \psi_2\subset TNwN$).
The function $\psi_1$ can be viewed as an element of $\mathcal S(F^2)$ in the variable $(y_2,y_3)$. Choose the quotient measure $dy$ on $T\backslash TNN'$ such that $dy=dy_2~dy_3$ for fixed Haar measure $dy_i$ on $F$. Proceed analogly for $\psi_2$.
For fixed $x$ the inner integral $I_\phi$ (\ref{gl_inneres_int_def}) of the local linking number is a locally constant function in $y$. Its support is not compact anymore, but $I_\phi$ is the locally constant limit of Schwartz functions.
This is the reason for this case being that more difficult than the case of a compact torus. The shape of the translated linking numbers is given in the next theorem.
\begin{thm}\label{satz_translatiert_nicht_kompakt}
Let $T$ be a noncompact torus.
For fixed $x$,  the local linking number
$ <\phi,\begin{pmatrix}b&0\\0&1\end{pmatrix}.\psi>_x$
is a function in $b\in F^\times$ of the form
\begin{align*}
& \chi_1^{-1}(b)\Bigl(\mathbf 1_{\wp^n}(b)\lvert b\rvert (a_{+,1}v(b)+a_{+,2}) + A(b) + \mathbf 1_{\wp^n}(b^{-1})\lvert b\rvert^{-1} (a_{-,1}v(b)+a_{-,2})\Bigr)\\
& +
 \chi_1(b)\Bigl(\mathbf 1_{\wp^n}(b)\lvert b\rvert (c_{+,1}v(b)+c_{+,2}) + C(b) + \mathbf 1_{\wp^n}(b^{-1})\lvert b\rvert^{-1} (c_{-,1}v(b)+c_{-,2})\Bigr),&
\end{align*}
with suitable constants $a_{\pm,i},c_{\pm,i}\in\mathbb C$, integral $n>0$ and functions $A,C\in\mathcal S(F^\times)$.
\end{thm}
{\it On the proof of Theorem~\ref{satz_translatiert_nicht_kompakt}.}
This is done by brute force computations, which need about 100 pages. Instead of giving these, we will outline the reduction to 	realizable $\wp$-adic integration here and refer to \cite{diss}, Ch.~8, for the computations. What is more, in Example~\ref{bsp_lln_nichtkompakt} we compute one special translated local linking number in detail which gives a good insight to the general calculations and which is used eventually.

We choose the functions $\phi,\psi$ locally as simple as possible,  that is: If $z\in\operatorname{supp}\phi$, then $z$ belongs to $TNN'$ or $TNwN$. Let us reduce ourselves to $z\in TNN'$, as the other case is done similarly. Actually it is seen (\cite{diss}, Ch.~8) that all the calculations for one case can be put down to those for the other.
There is a representative 
\begin{equation*}
\tilde z= \begin{pmatrix}1+z_2z_3&z_2\\z_3&1\end{pmatrix} 
\end{equation*}
of $z$ modulo $T$ and an open set
\begin{equation*}
 U_z=\tilde z +\begin{pmatrix}\wp^m&\wp^m\\\wp^m&\wp^m\end{pmatrix}
\end{equation*}
such that $\phi\vert_{U_z}=\phi(\tilde z)$. Choosing $m>0$ that big that $U_z$ is fundamental, $\phi$ locally has the shape $\phi_z:=\chi\cdot\mathbf 1_{U_z}$ up to some multiplicative constant.

For the exterior function $\psi$  proceed similarly. Evidently, it is enough to determine the behavior of the translated local linking numbers for functions of this type.
Thus, we are reduced to compute 
\begin{equation*}
 \int_{T\backslash G}\int_T\phi_z(t^{-1}\gamma(x)ty)~dt~\bar\psi_{\tilde z}(y\begin{pmatrix}b&0\\0&1\end{pmatrix})~dy.
\end{equation*}
According to whether $z_2$ or $z_3$ is zero or not, and $\operatorname{supp}\psi\subset TNN'$ or $\operatorname{supp}\psi\subset TNwN$, there are eight types of integrals to be done (\cite{diss}, Ch.~5.2 and 8).
\begin{eg}\label{bsp_lln_nichtkompakt}
Let $T$ be noncompact.
Let  $\chi=(\chi_1,\chi_1)$, where $\chi_1$ is unramified and quadratic.
Then $\phi=\chi\cdot\mathbf 1_{\operatorname{GL}_2(\mathbf o_F)}$ is well-defined in $\mathcal S(\chi, G)$.
The translated local linking number 
$<\phi,\begin{pmatrix} b&0\\0&1\end{pmatrix}.\phi>_x$
is given by
\begin{eqnarray*}
 && 
\chi_1(1-x)\chi_1(b)\vol^\times (\mathbf o_F^\times)\vol(\mathbf o_F)^2\cdot\\
&&\Biggl[  
\mathbf 1_{F^\times\backslash (1+\wp)}(x)\Biggl(\mathbf 1_{\mathbf o_F^\times}(b)\bigl(\lvert v(x)\rvert +1\bigr)(1+q^{-1}) +\mathbf 1_{\wp}(b)\lvert b\rvert \bigl(4v(b)+2\lvert v(x)\rvert\bigr)\Biggr.
\Biggr.\\
&& \quad\quad\quad\quad\quad\quad\quad\Biggl. +\mathbf 1_{\wp}(b^{-1})\lvert b^{-1}\rvert \bigl(-4v(b)+2\lvert v(x)\rvert\bigr)\Biggr)\\
&&
\quad+\quad \mathbf 1_{1+\wp}(x)\Biggl(\mathbf 1_{\wp^{v(1-x)+1}}(b)\lvert b\rvert \bigl(4v(b)-4v(1-x)\bigr) \Biggr.\\
&&
\quad\quad\quad\quad\quad\quad\quad\Biggl.\Biggl. +\mathbf 1_{v(1-x)\mathbf o_F^\times}(b)\lvert b\rvert +
\mathbf 1_{v(1-x)\mathbf o_F^\times}(b^{-1})\lvert b^{-1}\rvert \Biggr.\Biggr.\\
&&
\quad\quad\quad\quad\quad\quad\quad\Biggl.\Biggl. + \mathbf 1_{\wp^{v(1-x)+1}}(b^{-1})\lvert b^{-1}\rvert \bigl(-4v(b)-4v(1-x)\bigr)\Biggr)\Biggr].
\end{eqnarray*}
\end{eg}
The calculation of Example~\ref{bsp_lln_nichtkompakt} is given in Appendix~B.

\section{A geometric Hecke  operator}\label{section_geom_Hecke}
Here, an adequate operator on the local linking numbers is constructed that realizes the asymptotics ($b\to0$) of the Hecke operator on Whittaker products. The asymptotics of the second one is descibed by the following.
\begin{prop}\label{prop_hecke_auf_whittaker_prod}
 The Whittaker products $W(b\xi,b\eta)$ have the following behavior for $b\to 0$ and fixed $\xi=\frac{x}{x-1}$, $\eta=1-\xi$.

(a) In case of a compact Torus $T$ and $\chi$ not factorizing via the norm,
\begin{equation*}
 W(b\xi,b\eta)=0.
\end{equation*}
In case of a compact Torus $T$ and $\chi=\chi_1\circ\NN$,
\begin{equation*}
  W(b\xi,b\eta)=\lvert b\rvert\lvert\xi\eta\rvert^{\frac{1}{2}}\chi_1(b\eta)\left(c_1+c_2\omega(b\xi)\right)\left(c_3\mathbf 1_{\wp^m\cap(1-x)\NN}(b)+c_4\mathbf 1_{\wp^m\cap(1-x)z\NN}(b)\right),
\end{equation*}
where $z\in F^\times\backslash \NN$.

(b) In case of a noncompact Torus $T$,
\begin{equation*}
 W(b\eta,b,\xi)=\left\{\begin{matrix}
\lvert b\rvert\lvert\xi\eta\rvert^{\frac{1}{2}}\left(c_1\chi_1(b\eta)+c_2\chi_1^{-1}(b\eta)\right)\left(c_3v(b\xi)+c_4)\right),&\textrm{if } \chi_1^2\not=1\\
\lvert b\rvert\lvert\xi\eta\rvert^{\frac{1}{2}}\chi_1(b\eta)\left(c_1v(b\eta)+c_2\right)\left(c_3v(b\xi)+c_4\right),&\textrm{if }\chi_1^2=1\end{matrix}\right..
\end{equation*}
In here, $c_i\in\mathbb C$, $i=1,\dots,4$.
\end{prop}
\begin{proof}[Proof of Proposition~\ref{prop_hecke_auf_whittaker_prod}]
 For $b\to 0$ both arguments $b\xi\to 0$ and $b\eta\to 0$. The stated behavior is directly collected from Propositions~\ref{prop_characterisierung_theta_functionen} and \ref{prop_characterisierung_eisenstein_functionen}.
\end{proof}
We notice that the translation of the local linking numbers by $b$ studied above underlies this asymptotics (cf. Proposition~\ref{prop_translatiert_kompakt} and Theorem~\ref{satz_translatiert_nicht_kompakt}), but that it does not realize the leading terms in  case  $\chi$ is quadratic. In case of a noncompact torus $T$, the leading term is $v(b)^2$, while translation only produces $v(b)$. In case of a compact torus, the translated linking numbers have compact support, while the Hecke operator on Whittaker products has not.

In the following, we make the additional ``completely unramified'' assumption  which is satisfied at  all but finitely many places of a division quaternion algebra over  a number field. For applications to Gross-Zagier formula, the constructed  operator is required  in case of this hypothesis only.
\begin{assumption}\label{annahme_fuer_geom_hecke}
$D$ is a split algebra, i.e. $G$ is isomorphic to $\operatorname{GL}_2(F)$. In this,  $K/F$ is an unramified extension (split or nonsplit).  $\chi$ is an unramified character.
\end{assumption}
%%%%%%%%%
%%%
%%%
%%%
%%%%%%%%%%%%%%%%%%%%%
For a noncompact torus $T$, the translated local linking numbers (Theorem~\ref{satz_translatiert_nicht_kompakt}) split into sums of the form
\begin{equation*}
 <\phi,\begin{pmatrix}\beta&0\\0&1\end{pmatrix}.\psi>_x = <\phi,\begin{pmatrix}\beta&0\\0&1\end{pmatrix}.\psi>_x^++ <\phi,\begin{pmatrix}\beta&0\\0&1\end{pmatrix}.\psi>_x^-,
\end{equation*}
where
\begin{align}
&  <\phi,\begin{pmatrix}\beta&0\\0&1\end{pmatrix}.\psi>_x^\pm:=\chi_1^{\pm1}(\beta)\cdot\label{schranken_def_translatiert_noncompakt}\\
&
\Bigl(\mathbf 1_{\wp^n}\lvert\beta\rvert(c_{\pm,1}v(\beta)+c_{\pm,2})+C_{\pm}(\beta)+\mathbf 1_{\wp^n}(\beta^{-1})\lvert\beta\rvert^{-1}(d_{\pm,1}v(\beta)+d_{\pm,2})\Bigr)\nonumber
\end{align}
are the summands belonging to $\chi_1^{\pm1}$ respectively. In here, the constants $c_{\pm,i}, d_{\pm,i}$,  and $C_\pm\in\mathcal S(F^\times)$ as well as $n>0$ depend on $\phi,\psi$ and $x$.
If $\chi_1$ is a quadratic character, these two summands coinside.

To give an operator fitting all cases, define in case of a compact torus  
\begin{equation*}
  <\phi,\begin{pmatrix}\beta&0\\0&1\end{pmatrix}.\psi>_x^\pm:= <\phi,\begin{pmatrix}\beta&0\\0&1\end{pmatrix}.\psi>_x.
\end{equation*}
For $v(b)\geq 0$ define the operator $\mathbf S_b$  as
\begin{equation}
 \mathbf S_b:=\frac{1}{4}\left(\mathbf S_b^++\mathbf S_b^-\right),
\end{equation}
where
\begin{eqnarray*}
\mathbf S_b^\pm<\phi,\psi>_x &:=&
\sum_{s=0,1}\sum_{i=0}^{v(b)}\frac{\chi_1^{\mp 1}(\pi)^{i(-1)^s}\omega(b(1-x))^{i+s}}{\lvert\pi^{v(b)-i}\rvert}\\
&& \quad\quad\quad\quad\cdot<\phi,\begin{pmatrix}\pi^{(-1)^s(v(b)-i)}&0\\0&1\end{pmatrix}.\psi>_x^\pm.
\end{eqnarray*}
It is worthwhile remarking that this ``Hecke operator'' is not unique. For example, the summand for $s=0$ is an operator - call it $\mathbf T_b$ - owning the same properties than $\mathbf S_b$ itself. The crucial point seems to be that an averaging sum occurs. The operator $\mathbf S_b$ is chosen such that this sum includes negative exponents $-v(b)+i$ as well. This kind of symmetry will make the results on the local Gross-Zagier formula look quite smoothly (cf. Section~\ref{section_meine_Gross-Zagier_formel}). But these results could be obtained by $\mathbf T_b$ as well. 
\begin{prop}\label{prop_geom_hecke_kompakt}
 Let $T$ be a compact torus. Then the operator $\mathbf S_b$ reduces to
\begin{equation*}
\mathbf S_b<\phi,\psi>_x=\frac{1}{2}\sum_{s=0,1}\sum_{i=0}^{v(b)}\frac{\omega(b(1-x))^{i+s}}{\lvert \pi^{v(b)-i}\rvert}<\phi,\begin{pmatrix}\pi^{(-1)^s(v(b)-i)}&0\\0&1\end{pmatrix}.\psi>_x.
\end{equation*}
Let $x\in c\NN$ be fixed. For $\phi,\psi\in\mathcal S(\chi, G)$ there are constants $c_1,c_2\in\mathbb C$ and $n\in\mathbb N$ such that for $v(b)\geq n$
\begin{equation*}
 \mathbf S_b<\phi,\psi>_x=c_1\mathbf 1_{\wp^n\cap(1-x)\NN}(b)+c_2\mathbf 1_{\wp^n\cap(1-x)z\NN}(b).
\end{equation*}
\end{prop}

\begin{prop}\label{prop_geom_hecke_nonkompakt}
Let $T$ be a noncompact torus. The operators $\mathbf S_b^\pm$ reduce to
\begin{equation*}
\mathbf S_b^\pm<\phi,\psi>_x 
\sum_{s=0,1}\sum_{i=0}^{v(b)}\frac{\chi_1^{\mp 1}(\pi)^{i(-1)^s}}{\lvert\pi^{v(b)-i}\rvert}
<\phi,\begin{pmatrix}\pi^{(-1)^s(v(b)-i)}&0\\0&1\end{pmatrix}.\psi>_x^\pm.
\end{equation*}
Let $x\in F^\times$ be fixed.
For $\phi,\psi\in\mathcal S(\chi, G)$ there are constants $c_0,\dots,c_3\in\mathbb C$ and $n\in\mathbb N$ such that for $v(b)\geq n$
\begin{displaymath}
 \mathbf S_b <\phi,\psi>_x = \chi_1(b)\bigl( c_1 v(b) +c_0\bigr) + \chi_1^{-1}(b)\bigl( c_3 v(b) +c_2\bigr),
\end{displaymath}
if $\chi_1^2\not=1$, and
\begin{displaymath}
 \mathbf S_b <\phi,\psi>_x = \chi_1(b)\bigl( c_2v(b)^2+c_1v(b)+c_0\bigr),
\end{displaymath}
if $\chi_1^2=1$.
\end{prop}

\begin{thm}\label{satz_matching_operator}
 For fixed $x$, the local linking numbers $\mathbf S_b<\phi,\psi>_x$ realize the asymptotics of the translated Whittaker products $ W(b\xi,b\eta)$ up to a factor $\lvert b\rvert\lvert \xi\eta\rvert^{\frac{1}{2}}$.
\end{thm}
\begin{proof}[Proof of Theorem~\ref{satz_matching_operator}]
In case  $T$ compact,
combining Proposition~\ref{prop_geom_hecke_kompakt} with Pro\-po\-si\-tion~\ref{prop_hecke_auf_whittaker_prod}~(a) for $\chi=1$ yields the Theorem.
In case $T$ noncompact, combine Proposition~\ref{prop_geom_hecke_nonkompakt} with Proposition~\ref{prop_hecke_auf_whittaker_prod}~(b).
\end{proof}

\begin{proof}[Proof of Proposition~\ref{prop_geom_hecke_kompakt}]
Notice, that for $T$ compact Assumtion~\ref{annahme_fuer_geom_hecke} induces $\chi=1$ by Corollary~\ref{cor_chi}.
By Proposition~\ref{prop_translatiert_kompakt}, the translated linking number can be written as
\begin{equation*}
 <\phi,\begin{pmatrix}\beta&0\\0&1\end{pmatrix}.\psi>_x=\sum_i d_{a_i}\lvert a_i\rvert^{\operatorname{sign} v(a_i)}\mathbf 1_{a_i(1+\wp^m)}(\beta),
\end{equation*}
for finitely many $a_i\in F^\times$, $d_{a_i}\in\mathbb C$, and some $m>0$, where the sets $a_i(1+\wp^m)$ are pairwise disjoint. One can assume that in this sum all $\pi^l$, $-\operatorname{max}_i \lvert v(a_i)\rvert\leq l\leq \operatorname{max}_i \lvert v(a_i)\rvert$, occure. Let $n:=\operatorname{max}_i \lvert v(a_i)\rvert+1$. Then, for $v(b)\geq n$,
\begin{align*}
 \mathbf S_b<\phi,\psi>_x &=
\frac{1}{2}\sum_{i=0}^{v(b)}\left(\omega(b(1-x))^i\sum_{l=-n+1}^{n-1}\frac{d_{\pi^l}\lvert\pi^l\rvert^{\operatorname{sign}(l)}} {\lvert\pi^{v(b)-i}\rvert}\mathbf 1_{\pi^l(1+\wp^m)}(\pi^{v(b)-i})\right.\\
& \:+\left.\omega(b(1-x))^{i+1}\sum_{l=-n+1}^{n-1}\frac{d_{\pi^l}\lvert\pi^l\rvert^{\operatorname{sign}(l)}}{\lvert\pi^{v(b)-i}\rvert}\mathbf 1_{\pi^l(1+\wp^m)}(\pi^{i-v(b)})\right)\\
&=\frac{1}{2}\sum_{l=0}^{n-1}\omega(b(1-x))^{v(b)+l}d_{\pi^l}+\frac{1}{2}\sum_{l=-n+1}^{0}\omega(b(1-x))^{v(b)+l+1}d_{\pi^l}\\
&= c_1\mathbf 1_{\wp^n\cap(1-x)\NN}(b)+c_2\mathbf 1_{\wp^n\cap(1-x)z\NN}(b),
\end{align*}
where $c_1:=\frac{1}{2}\sum_{l=0}^{n-1}(d_{\pi^l}+d_{\pi^{-l}})$ and 
$c_2:=\frac{1}{2}\sum_{l=0}^{n-1}(-1)^l(d_{\pi^l}-d_{\pi^{-l}})$.
Notice, that for $b(1-x)\in z\NN$ one has $\omega(b(1-x))^{v(b)}=(-1)^{v(b)}=-\omega(1-x)$.
\end{proof}
\begin{proof}[Proof of Proposition~\ref{prop_geom_hecke_nonkompakt}]
Recall that $T$ noncompact induces $\omega=1$. 
First, one proves this asymptotics for the part $\mathbf T_b^-$ of $\mathbf S_b$ belonging to $\mathbf S_b^-$ and $s=0$,
\begin{equation*}
 \mathbf T_b^-<\phi,\psi>_x :=\sum_{i=0}^{v(b)}\frac{\chi_1(\pi)^i}{\lvert \pi^{v(b)-i}\rvert}<\phi,\begin{pmatrix}\pi^{v(b)-i}&0\\0&1\end{pmatrix}.\psi>_x^-.
\end{equation*}
Let $n>0$ be the integer as in  (\ref{schranken_def_translatiert_noncompakt}). Let $v(b)\geq n$. In the formula for $\mathbf T_b^-$, one distinguishes the summands whether $v(b)-i<n$ or not. 
If $v(b)-i<n$,
then
\begin{equation*}
 <\phi,\begin{pmatrix}\pi^{v(b)-i}&0\\0&1\end{pmatrix}.\psi>_x^-=\chi_1^{-1}(\pi^{v(b)-i})C_-(\pi^{v(b)-i}).
\end{equation*}
The function $\tilde C_-$ defined by
\begin{equation*}
 \tilde C_-(\beta):=\frac{\chi_1^{-2}(\beta)}{\lvert\beta\rvert}C_-(\beta)
\end{equation*}
again belongs to $\mathcal S(F^\times)$. The part of $\mathbf T_b^-$ made up by summands satisfying $v(b)-i<n$ is now simplyfied to
\begin{align*}
 &\sum_{i=v(b)-n+1}^{v(b)}\frac{\chi_1(\pi)^i}{\lvert\pi^{v(b)-i}\rvert}<\phi,\begin{pmatrix}\pi^{v(b)-i}&0\\0&1\end{pmatrix}.\psi>_x^-\\
&=\sum_{i=v(b)-n+1}^{v(b)}\chi_1(b)\tilde C_-(\pi^{v(b)-i})=\chi_1(b)\sum_{l=0}^{n-1}\tilde C_-(\pi^l).
\end{align*}
In here, the last sum is independent of $b$. Thus, this part of $\mathbf T_b^-$ satisfies the claim.
In the remaining part
\begin{equation*}
 T(i\leq v(b)-n):=\sum_{i=0}^{v(b)-n}\frac{\chi_1(\pi)^i}{\lvert \pi^{v(b)-i}\rvert}<\phi,\begin{pmatrix}\pi^{v(b)-i}&0\\0&1\end{pmatrix}.\psi>_x^-
\end{equation*}
all the translated local linking numbers occuring can be written as
\begin{equation*}
 <\phi,\begin{pmatrix}\pi^{v(b)-i}&0\\0&1\end{pmatrix}.\psi>_x^-=\chi_1^{-1}(\pi^{v(b)-i})\lvert\pi^{v(b)-i}\rvert\left(c_{-,1}(v(b)-i)+c_{-,2}\right).
\end{equation*}
Using this, the remaining part simplyfies to
\begin{equation*}
 T(i\leq v(b)-n)=\chi_1^{-1}(b)\sum_{i=0}^{v(b)-n}\chi_1(\pi)^{2i}\left(c_{-,1}(v(b)-i)+c_{-,2}\right).
\end{equation*}
In the following one has to distinguish between $\chi_1$ quadratic or not. First, let $\chi_1^2=1$. Then
\begin{equation*}
 T(i\leq v(b)-n)=\chi_1(b)(v(b)-n+1)\left( c_{-,2}+\frac{1}{2}c_{-,1}(v(b)+n) \right),
\end{equation*}

which owns the claimed asymptotics.
Let $\chi_1^2\not=1$. By enlarging $n$ one can assume that the order of $\chi_1$ divides $n$, i.e. $\chi_1^n=1$. The remaining part of $\mathbf T_b^-$ in this case is
\begin{align*}
 T(i\leq v(b)-n)&=(c_{-,1}v(b)+c_{-,2})\frac{\chi_1(b\pi)-\chi_1^{-1}(b\pi)}{\chi_1(\pi)-\chi_1^{-1}(\pi)}\\
&\quad-c_{-,1}\frac{\chi_1(b\pi)(v(b)-n+1)}{\chi_1(\pi)-\chi_1^{-1}(\pi)}
+ c_{-,1}\frac{\chi_1(b\pi^2)-1}{(\chi_1(\pi)-\chi_1^{-1}(\pi))^2}.
\end{align*}
Thus, the claim is satisfied in case $\chi_1^2\not=1$, too.

The other parts of $\mathbf S_b$ satisfy the claimed asymptotics as well, as is easily deduced from the statement for $\mathbf T_b^-$. First, if $\mathbf T_b^+$ denotes the part of $\mathbf S_b$ belonging to $\mathbf S_b^+$ and $s=0$, then the statement for $\mathbf T_b^+$ follows  from the proof for $\mathbf T_b^-$ replacing there $\chi_1^{-1}$ by $\chi_1$, $C_-$ by $C_+$, and $c_{-,i}$ by $c_{+,i}$, where the constants are given by (\ref{schranken_def_translatiert_noncompakt}).
Second, for $s=1$ notice that
\begin{equation*}
 \chi_1(\pi)^{i(-1)^s}\chi_1^{-1}(\pi^{(-1)^s(v(b)-i)})=\chi_1(b)\chi_1(\pi)^{-2i}.
\end{equation*}
In this case the claim follows  if one rewrites the proofs for $s=0$ substituting $\chi_1$ by $\chi_1^{-1}$ as well as $c_{\pm,i}$ by $d_{\pm,i}$ of (\ref{schranken_def_translatiert_noncompakt}).
\end{proof}

\section{Local Gross-Zagier formula rewritten}\label{section_local_Gross-Zagier}
We report Zhang's local Gross-Zagier  formulae \cite{zhang} in the notation used throughout this paper in order to compare them directly with  the results given by the operator $\mathbf S_b$ just defined afterwards.
We prefer to give short proofs of the results by Zhang for the sake of readability.
We assume Hypothesis~\ref{annahme_fuer_geom_hecke}.

\subsection{Local Gross-Zagier formula by Zhang}\label{section_zhang_referiert}
The local Gross-Zagier formula compares the Whittaker products of local newforms with a local linking number belonging to a very special function $\phi$ (\cite{zhang} Ch.~4.1). This function is given by
\begin{equation*}
 \phi=\chi\cdot\mathbf 1_{R^\times},
\end{equation*}
where $R^\times$ in general is the unit group of a carefully chosen order $R$ in $ D$.
In case of Hypothesis~\ref{annahme_fuer_geom_hecke}, $R^\times=\operatorname{GL}_2(\mathbf o_F)$ and the function $\phi$ is well-defined. The special local linking number  is then defined by
\begin{equation*}
 <\mathbb T_b\phi,\phi>_x,
\end{equation*}
where the ``Hecke operator'' $\mathbb T_b$ is defined as follows (\cite{zhang} 4.1.22 et sqq.). Let
\begin{equation*}
 H(b):=\{g\in M_2(\mathbf o_F)\mid v(\det g)=v(b)\}.
\end{equation*}
Then
\begin{equation*}
 \mathbb T_b\phi(g):=\int_{H(b)}\phi(hg)~dh.
\end{equation*}
Notice that this operator is well-defined on $\phi$ because $\phi$ essentially is the characteristic function of  $\operatorname{GL}_2(\mathbf o_F)$, but not for most other functions. In our construction of the operator $\mathbf S_b$, we continued in some way the idea that $\mathbb T_b$ has the flavor of summation over translates by coset representatives, as
\begin{equation*}
 H(b)=\bigcup\begin{pmatrix}y_1&0\\0&y_3\end{pmatrix}\begin{pmatrix}1&y_2\\0&1\end{pmatrix}\operatorname{GL}_2(\mathbf o_F),
\end{equation*}
where the union is over representatives $(y_1,y_3)\in\mathbf o_F\times\mathbf o_F$ with $v(y_1y_3)=v(b)$ and $y_2\in \wp^{-v(y_1)}\backslash \mathbf o_F$.
\begin{lem}\label{zhangs_lln_kompakt}
 (\cite{zhang} Lemma~4.2.2)
Let $K/F$ be a field extension and assume Hypothesis~\ref{annahme_fuer_geom_hecke}. Let $\phi=\chi\cdot\mathbf 1_{\operatorname{GL}_2(\mathbf o_F)}$. Then
\begin{equation*}
 <\mathbb T_b\phi,\phi>_x=\vol(\operatorname{GL}_2(\mathbf o_F))^2\vol_T(T)\mathbf 1_{\NN}(x)\mathbf 1_{\frac{1-x}{x}(\mathbf o_F\cap\NN)}(b)\mathbf 1_{(1-x)(\mathbf o_F\cap\NN)}(b).
\end{equation*}
\end{lem}
\begin{lem}\label{zhangs_lln_nichtkompakt}
 (\cite{zhang} Lemma~4.2.3) Let $K/F$ be split, let $\chi=(\chi_1,\chi_1^{-1})$ be an unramified character, and let $\phi=\chi\cdot\mathbf 1_{\operatorname{GL}_2(\mathbf o_F)}$. In case $\chi_1^2\not=1$,
\begin{align*}
 <\mathbb T_b\phi,\phi>_x&=\frac{\chi_1(b(1-x)^{-1}\pi)-\chi_1^{-1}(b(1-x)^{-1}\pi)}{\chi_1(\pi)-\chi_1^{-1}(\pi)}\vol(\operatorname{GL}_2(\mathbf o_F))^2\vol^\times(\mathbf o_F^\times)\\
&\quad\cdot\mathbf 1_{\frac{1-x}{x}\mathbf o_F\cap(1-x)\mathbf o_F}(b)\mathbf 1_{F^\times}(x)\left(v(b)+v(\frac{x}{1-x})+1\right).
\end{align*}
In case $\chi_1^2=1$,
\begin{align*}
 <\mathbb T_b\phi,\phi>_x&=\chi_1(b(1-x))\vol(\operatorname{GL}_2(\mathbf o_F))^2\vol^\times(\mathbf o_F^\times)
\mathbf 1_{\frac{1-x}{x}\mathbf o_F\cap(1-x)\mathbf o_F}(b)\\
&\quad\cdot\mathbf 1_{F^\times}(x)\bigl(v(b)-v(1-x)+1\bigr)\bigl(v(b)+v(\frac{x}{1-x})+1\bigr).
\end{align*}
\end{lem}
For the proofs of Lemma~\ref{zhangs_lln_kompakt} and \ref{zhangs_lln_nichtkompakt} we follow a hint given orally by Uwe Weselmann. Write
\begin{equation*}
 \phi(x)=\sum_{\tau\in  T(F)/T(\mathbf o_F)}\chi(\tau)\mathbf 1_{\tau\operatorname{GL}_2(\mathbf o_F)}(x).
\end{equation*}
For the Hecke operator we find 
\begin{equation*}
 \mathbb T_b\mathbf 1_{\tau\operatorname{GL}_2(\mathbf o_F)}(x) =\vol(\operatorname{GL}_2(\mathbf o_F))\mathbf 1_{\tau b^{-1}H(b)}(x),
\end{equation*}
as $b^{-1}H(b)=\{h\in \operatorname{GL}_2(F)\mid h^{-1}\in H(b)\}$. Noticing that the Hecke operator is right invariant under multiplication by $y\in \operatorname{GL}_2(\mathbf o_F)$, i.e. $\mathbb T_b\phi(xy)=\mathbb T_b\phi(x)$, we get
\begin{equation}\label{formel_beweis_zhangs_lln}
 <\mathbb T_b\phi,\phi>_x=\vol(\operatorname{GL}_2(\mathbf o_F))^2\sum_\tau\chi(\tau)\int_T\mathbf 1_{\tau b^{-1}H(b)}(t^{-1}\gamma(x)t)~dt.
\end{equation}
This formula  is evaluated in the different cases.
\begin{proof}[Proof of Lemma~\ref{zhangs_lln_kompakt}]
Let $K=F(\sqrt A)$, where $v(A)=0$. Choose a tracefree $\gamma(x)=\sqrt A +\epsilon(\gamma_1+\gamma_2\sqrt A)$, where $\NN(\gamma_1+\gamma_2\sqrt A)=x$. The conditions for the integrands of (\ref{formel_beweis_zhangs_lln}) not vanishing are
\begin{align*}
 \tau^{-1}b\sqrt A&\in\mathbf o_K\\
 \tau^{-1}b\bar t^{-1}t(\gamma_1+\gamma_2\sqrt A)&\in\mathbf o_K\\
 \det(t^{-1}\gamma(x)t)=A(x-1)&\in b^{-1}\NN(\tau)\mathbf o_F^\times.
\end{align*}
They are equivalent to $\lvert \NN(\tau)\rvert=\lvert b(1-x)\rvert$ and $\lvert b\rvert\leq \min\{\lvert\frac{1-x}{x}\rvert,\lvert 1-x\rvert\}$.
There is only one coset $\tau\in T(F)/T(\mathbf o_F)$ satisfying this, and this coset only exists if $b\in(1-x)\NN$. Thus,
\begin{eqnarray*}
 <\mathbb T_b\phi,\phi>_x&=&\vol(\operatorname{GL}_2(\mathbf o_F))^2\vol_T(T)\\
&&\cdot\left(\mathbf 1_{\NN\backslash (1+\wp)}(x)\mathbf 1_{\mathbf o_F\cap\NN}(b)+\mathbf 1_{1+\wp}(x)\mathbf 1_{(1-x)(\mathbf o_F\cap\NN)}(b)\right),
\end{eqnarray*}
which equals the claimed result. 
\end{proof}
\begin{proof}[Proof of Lemma~\ref{zhangs_lln_nichtkompakt}]
 First evaluate the integral
\begin{equation*}
 I_\tau(b,x):=\int_T\mathbf 1_{\tau b^{-1}H(b)}(t^{-1}\gamma(x)t)~dt.
\end{equation*}
Choose $\gamma(x)=\begin{pmatrix}-1&x\\-1&1\end{pmatrix}$ tracefree, and set $\tau=(\tau_1,\tau_2)\in K^\times/\mathbf o_K^\times$ as well as $t=(a,1)\in T$. 
The conditions for the integrand of (\ref{formel_beweis_zhangs_lln}) not vanishing are
\begin{align*}
(-\tau_1^{-1}b,\tau_2^{-1}b)&\in\mathbf o_K,\\
 (-\tau_1^{-1}a^{-1}bx,\tau_2^{-1}ab)&\in\mathbf o_K,\\
 \det(t^{-1}\gamma(x)t)=x-1&\in\NN(\tau)b^{-1}\mathbf o_K^\times.
\end{align*}
That is: Only if $v(\tau_2)=-v(\tau_1)+v(b)+v(1-x)$ satisfies $v(1-x)\leq v(\tau_2)\leq v(b)$, the integral does not vanish. Then the scope of integration  is given by $-v(b)+v(\tau_2)\leq v(a)\leq v(\tau_2)+v(x)-v(1-x)$ and the integral equals
\begin{equation*}
  I_\tau(b,x)=\vol^\times(\mathbf o_F^\times)\left(v(b)+v(x)-v(1-x)+1\right)\mathbf 1_{\mathbf o_F\cap\wp^{v(1-x)-v(x)}}(b).
\end{equation*}
Evaluating $\chi(\tau)$ we get $\chi(\tau)=\chi_1(b(1-x))\chi_1^{-2}(\tau_2)$, as $\chi$ is unramified. Summing up the terms of (\ref{formel_beweis_zhangs_lln}) yields the claim.
\end{proof}
The other constituents of the local Gross-Zagier formulae are the Whittaker products of newforms for both the Theta series $\Pi(\chi)$ and the Eisensteinseries $\Pi(1,\omega)$ at $s=\frac{1}{2}$. By Hypothesis~\ref{annahme_fuer_geom_hecke}, the Theta series equals $\Pi(\chi_1,\chi_1^{-1})$ if $K/F$ splits, and it equals $\Pi(1,\omega)$ if $K/F$ is a field extension. Thus, all occuring representations are principal series and the newforms read in the Kirillov model are given by (\ref{gleichung_whittaker_neuform}).
In case of a field extension we get
\begin{equation*}
 W_{\theta, new}(a)=W_{E, new}(a)=\lvert a\rvert^{\frac{1}{2}}\mathbf 1_{\mathbf o_F\cap\NN}(a)\vol(\mathbf o_F)\vol^\times(\mathbf o_F^\times).
\end{equation*}
In case $K/F$ splits one gets
\begin{equation*}
 W_{\theta,new}(a)=\lvert a\rvert^{\frac{1}{2}}\mathbf 1_{\mathbf o_F}(a)\vol(\mathbf o_F)\vol^\times(\mathbf o_F^\times)\left\{\begin{matrix}\frac{\chi_1(a\pi)-\chi_1^{-1}(a\pi)}{\chi_1(\pi)-\chi_1^{-1}(\pi)},&\textrm{if } \chi_1^2\not=1\\\chi_1(a)(v(a)+1),&\textrm{if }\chi_1^2=1\end{matrix}\right.,
\end{equation*}
while
\begin{equation*}
 W_{E,new}(a)=\lvert a\rvert^{\frac{1}{2}}\mathbf 1_{\mathbf o_F}(a)(v(a)+1)\vol(\mathbf o_F)\vol^\times(\mathbf o_F^\times).
\end{equation*}
Summing up, we get the following Lemma on the shape of Whittaker products for newforms. We give two expressions of them, the first one using the variables $\xi=\frac{x}{x-1}$ and $\eta=1-\xi$ the second one using the variable $x$. 
\begin{lem}\label{lemma_whittaker_neuform_produkte}
(\cite{zhang} Lemma~3.4.1)
Assume Hypothesis~\ref{annahme_fuer_geom_hecke}. Then the Whittaker products for the newforms of Theta series and Eisenstein series have the following form up to the factor $\vol(\mathbf o_F)^2\vol^\times(\mathbf o_F^\times)^2$.
If $K/F$ is a field extension, then
\begin{eqnarray*}
 W_{\theta, new}(b\eta)W_{E, new}(b\xi) &=&
\lvert \xi\eta\rvert^{\frac{1}{2}}\lvert b\rvert \mathbf 1_{\mathbf o_F}(b\xi)\mathbf 1_{\mathbf o_F}(b\eta)\\
&=& \lvert \xi\eta\rvert^{\frac{1}{2}}\lvert b\rvert\mathbf 1_{\frac{1-x}{x}(\mathbf o_F\cap\NN)}(b)\mathbf 1_{(1-x)(\mathbf o_F\cap\NN)}(b).
\end{eqnarray*}
If $K/F$ splits and $\chi$ is quadratic, then
\begin{eqnarray*}
 && W_{\theta, new}(b\eta)W_{E, new}(b\xi)\\
&&\quad\quad=\lvert \xi\eta\rvert^{\frac{1}{2}}\lvert b\rvert \mathbf 1_{\mathbf o_F}(b\xi)\mathbf 1_{\mathbf o_F}(b\eta)\chi_1(b\eta)\left(v(b\xi)+1\right)\left(v(b\eta)+1\right)\\
&&\quad\quad=\lvert \xi\eta\rvert^{\frac{1}{2}}\lvert b\rvert\mathbf 1_{\frac{1-x}{x}\mathbf o_F\cap(1-x)\mathbf o_F}(b)\chi_1(b(1-x))\bigl(v(b)+v(\frac{x}{1-x})+1\bigr)\cdot\\
&&\quad\quad\quad\quad\quad\quad\quad\quad\quad\quad\quad\quad\quad \cdot\bigl(v(b)-v(1-x)+1\bigr).
\end{eqnarray*}
 If $K/F$ splits and $\chi$ is not quadratic, then
\begin{eqnarray*}
 && W_{\theta, new}(b\eta)W_{E, new}(b\xi)\\
&&\quad\quad=\lvert \xi\eta\rvert^{\frac{1}{2}}\lvert b\rvert \mathbf 1_{\mathbf o_F}(b\xi)\mathbf 1_{\mathbf o_F}(b\eta)\left(v(b\xi)+1\right)\frac{\chi_1(b\eta\pi)-\chi_1^{-1}(b\eta\pi)}{\chi_1(\pi)-\chi_1^{-1}(\pi)}\\
&&\quad\quad=\lvert \xi\eta\rvert^{\frac{1}{2}}\lvert b\rvert\mathbf 1_{\frac{1-x}{x}\mathbf o_F\cap(1-x)\mathbf o_F}(b)\bigl(v(b)+v(\frac{x}{1-x})+1\bigr)\cdot\\
&&\quad\quad\quad\quad\quad\quad\quad\quad\quad \cdot\frac{\chi_1(b(1-x)^{-1}\pi)-\chi_1^{-1}(b(1-x)^{-1}\pi)}{\chi_1(\pi)-\chi_1^{-1}(\pi)}.
\end{eqnarray*}
\end{lem}
In comparing Lemma~\ref{zhangs_lln_kompakt} resp.~\ref{zhangs_lln_nichtkompakt} with Lemma~\ref{lemma_whittaker_neuform_produkte} one now gets the local Gross-Zagier formula by Zhang:
\begin{thm}\label{zhangs_local_Gross-Zagier}
(\cite{zhang} Lemma~4.3.1)
 Assume Hypothesis~\ref{annahme_fuer_geom_hecke}. Let $W_{\theta, new}$ resp. $W_{E,new}$ be the newform for the Theta series resp. Eisenstein series. Let $\phi=\chi\cdot\mathbf 1_{\operatorname{GL}_2(\mathbf o_F)}$. Then up to a factor of volumes,
\begin{equation*}
 W_{\theta, new}(b\eta)W_{E, new}(b\xi)=\lvert \xi\eta\rvert^{\frac{1}{2}}\lvert b\rvert <\mathbb T_b\phi,\phi>_{x=\frac{\xi}{\xi-1}}.
\end{equation*}
\end{thm}

\subsection{Rephrasing  local Gross-Zagier}\label{section_meine_Gross-Zagier_formel}
As a test for the effectivity of the operator $\mathbf S_b$  constructed in Section~\ref{section_geom_Hecke}, we rewrite Zhang's local Gross-Zagier formula in terms of $\mathbf S_b$.
\begin{thm}\label{neuformulierung_local_Gross-Zagier}
Assume Hypothesis~\ref{annahme_fuer_geom_hecke}. Assume further, that $\chi_1^2=1$ in case $K/F$ splits. Let $W_{\theta, new}$ resp. $W_{E,new}$ be the newform for the Theta series resp. Eisenstein series. Let $\phi=\chi\cdot\mathbf 1_{\operatorname{GL}_2(\mathbf o_F)}$. Then up to a factor of volumes, 
\begin{equation*}
 W_{\theta, new}(b\eta)W_{E, new}(b\xi)=\lvert \xi\eta\rvert^{\frac{1}{2}}\lvert b\rvert\mathbf S_b<\phi,\phi>_x+O(v(b)),
\end{equation*}
where in case $K/F$ a field extension the term of  $O(v(b))$ is actually zero, while in case $K/F$ split the term of $O(v(b))$ can be given precisely by collecting terms in the proof of Example~\ref{bsp_lln_nichtkompakt}.
\end{thm}
\begin{proof}[Proof of Theorem~\ref{neuformulierung_local_Gross-Zagier}]
One has to compare the Whittaker products for newforms given in Lemma~\ref{lemma_whittaker_neuform_produkte} with
the action of the operator $\mathbf S_b$ on the special local linking number belonging to $\phi$. This action is calculated in Lemma~\ref{mein_operator_spezielle_lln_kompakt} resp. \ref{mein_operator_spezielle_lln_nichtkompakt}  below. 
\end{proof}
\begin{lem}\label{mein_operator_spezielle_lln_kompakt}
 Let $K/F$ be a field extension. Assume Hypothesis~\ref{annahme_fuer_geom_hecke}. Let $\phi=\chi\cdot\mathbf 1_{\operatorname{GL}_2(\mathbf o_F)}$. Then up the factor $\vol_T(T)\vol^\times(\mathbf o_F^\times)\vol(\mathbf o_F)$, 
\begin{equation*}
 \mathbf S_b<\phi,\phi>_x=
\mathbf 1_{\NN}(x)\mathbf 1_{\frac{1-x}{x}(\mathbf o_F\cap\NN)}(b)\mathbf 1_{(1-x)(\mathbf o_F\cap\NN)}(b).
\end{equation*}
\end{lem}
\begin{proof}[Proof of Lemma~\ref{mein_operator_spezielle_lln_kompakt}]
 The translated local linking number $<\phi,\begin{pmatrix}\beta&0\\0&1\end{pmatrix}.\phi>_x$ was computed in Example~\ref{bsp_lln_translatiert_kompakt}. One has to compute the sum for the operator $\mathbf S_b$ given in Proposition~\ref{prop_geom_hecke_kompakt}. If $x\in\NN\backslash (1+\wp)$, then up to the factor $\vol_T(T)\vol^\times(\mathbf o_F^\times)\vol(\mathbf o_F)$,
\begin{equation*}
 \mathbf S_b<\phi,\phi>_x=\frac{1}{2}\left(\omega(b(1-x))^{v(b)}+\omega(b(1-x))^{v(b)+1}\right)=\mathbf 1_{\NN}(b).
\end{equation*}
If $x\in1+\wp$, then again up to the factor of volumes
\begin{align*}
 \mathbf S_b<\phi,\phi>_x&=
\frac{1}{2}\mathbf 1_{\wp^{v(1-x)}}(b)\omega(b(1-x))^{v(b)-v(1-x)}\left(1+\omega(b(1-x))\right)\\
&= \mathbf 1_{\wp^{v(1-x)}\cap(1-x)\NN}(b).\qedhere
\end{align*}
\end{proof}
In case $K/F$ we limit ourselves to the case $\chi_1^2=1$. 
\begin{lem}\label{mein_operator_spezielle_lln_nichtkompakt}
Let $K/F$ be split and assume Hypothesis~\ref{annahme_fuer_geom_hecke} as well as $\chi_1^2=1$. Let $\phi=\chi\cdot\mathbf 1_{\operatorname{GL}_2(\mathbf o_F)}$. Then  up the factor $\vol^\times(\mathbf o_F^\times)\vol(\mathbf o_F)^2$, 
\begin{align*}
 &\mathbf S_b<\phi,\phi>_x=\chi_1(b(1-x))\cdot\\
&\quad\left[\mathbf 1_{F^\times\backslash(1+\wp)}(x)\Bigl(2v(b)^2+2(\lvert v(x)\rvert+1)v(b)+(1+q^{-1})(\lvert v(x)\rvert+1)\Bigr)\right.\\
&\quad\left.+\mathbf 1_{1+\wp}(x)\mathbf 1_{\wp^{v(1-x)}}(b)\Bigl(2\bigl(v(b)-v(1-x)+1\bigr)\bigl(v(b)-v(1-x)\bigr)+1\Bigr)\right].
\end{align*}
\end{lem}
\begin{proof}[Proof of Lemma~\ref{mein_operator_spezielle_lln_nichtkompakt}]
 In order to evaluate the action of $\mathbf S_b$, one has to know the translated local linking numbers
$<\phi,\begin{pmatrix}\beta&0\\0&1\end{pmatrix}.\phi>_x$. These were computed in Example~\ref{bsp_lln_nichtkompakt}. The operator $\mathbf S_b$ is given in Proposition~\ref{prop_geom_hecke_nonkompakt}. As $\chi_1$ is quadratic, $\mathbf S_b=\frac{1}{2}\mathbf S_b^+$.
For $x\in 1+\wp$ we compute
\begin{align*}
 &\mathbf S_b<\phi,\phi>_x\\
&=
\chi_1(b(1-x))\mathbf 1_{\wp^{v(1-x)}}(b)\left(1+\sum_{i=0}^{v(b)-v(1-x)-1}4(v(b)-i-v(1-x))\right)\\
&=
\chi_1(b(1-x))\mathbf 1_{\wp^{v(1-x)}}(b)\Bigl(2(v(b)-v(1-x)+1)(v(b)-v(1-x))+1\Bigr),
\end{align*}
while for $x\in F^\times\backslash(1+\wp)$,
\begin{align*}
&\mathbf S_b<\phi,\phi>_x\\
&=
\chi_1(b(1-x))\left[(\lvert v(x)\rvert+1)(1+q^{-1})+\sum_{i=0}^{v(b)-1}\bigl(4(v(b)-i)+2\lvert v(x)\rvert\bigr)\right]\\
&=
\chi_1(b(1-x))\Bigl(2v(b)^2+(1+\lvert v(x)\rvert)(2v(b)+1+q^{-1})\Bigr).\qedhere
\end{align*}
\end{proof}
%%%%%%%%%%%%%%%%%%%%%%%%%%%%%%
%%%
%%%
%%%
%%%%%%%%%%%%%%%%%%%%%%%%%%%%%%%%%%%%%%%%
\section*{Appendix A\\ Proof of Example~\ref{bsp_lln_translatiert_kompakt}}
For $\phi=\chi\cdot\mathbf 1_{\operatorname{GL}_2(\mathbf o_F)}$ one has to compute 
\begin{equation*}
 <\phi,\begin{pmatrix}b&0\\0&1\end{pmatrix}\phi>_x=\int_{T\backslash G}\int_T\phi(t^{-1}\gamma(x)ty)~dt~\bar\phi(y\begin{pmatrix}b&0\\0&1\end{pmatrix})~dy,
\end{equation*}
where by Lemma~\ref{lemma_mirobolische_zerlegung}, one can assume $y$ to be of the form $y=\begin{pmatrix}y_1&y_2\\0&1\end{pmatrix}$ as well as $dy=d^\times y_1~dy_2$. Accordingly one factorizes $t^{-1}\gamma(x)t$, where $\gamma(x)=\sqrt A+\epsilon(\gamma_1+\gamma_2\sqrt A)$ is a tracefree preimage of $x$ under $P$ and $t=\alpha+\beta\sqrt A\in K^\times$:
\begin{equation*}
 t^{-1}\gamma(x)t=\tilde t\begin{pmatrix}g_1&g_2\\0&1\end{pmatrix},
\end{equation*}
where $\tilde t\in K^\times$ and
\begin{equation*}
 g_1=\frac{1+x}{1-x}+\frac{2((\alpha^2+\beta^2A)\gamma_1+2\alpha\beta A\gamma_2)}{(1-x)(\alpha^2-\beta^2A)},
\end{equation*}
\begin{equation*}
 g_2=\frac{2A((\alpha^2+\beta^2A)\gamma_2+2\alpha\beta\gamma_1)}{(1-x)(\alpha^2-\beta^2A)}.
\end{equation*}
The inner integrand is not zero if and only if $g_1y_1\in\mathbf o_F^\times$ as well as $g_1y_2+y_1\in\mathbf o_F$, while the outer integrand doesn't vanish only for $y_1\in b^{-1}\mathbf o_F^\times$ and $y_2\in\mathbf o_F$. Forcing this, one gets the following to conditions for the inner integrand:
\begin{equation}\label{bed_1_kompakt}
 g_1=\frac{1+x}{1-x}+\frac{2((\alpha^2+\beta^2A)\gamma_1+2\alpha\beta A\gamma_2)}{(1-x)(\alpha^2-\beta^2A)}\in b\mathbf o_F^\times,
\end{equation}
\begin{equation}\label{bed_2_kompakt}
 g_1y_2+\frac{2A((\alpha^2+\beta^2A)\gamma_2+2\alpha\beta\gamma_1)}{(1-x)(\alpha^2-\beta^2A)}\in\mathbf o_F.
\end{equation}
In the following, one distinguishes whether $v(\beta)-v(\alpha)\geq 0$ or not as well as whether $x\in\operatorname N$ is a square or not. The contributions and conditions for the scope of integration are determined and marked  by "$\bullet$" for final collection.
%%%%
%%%
%%%%

{\bf 1) Let $x\in F^{\times 2}$ and $v(\frac{\beta}{\alpha})\geq 0$.} Then $\gamma(x)$ can be chosen such that $x=\gamma_1^2$ and one can reduce (\ref{bed_1_kompakt}) and (\ref{bed_2_kompakt}) by $\alpha^2$  assuming $\beta\in\mathbf o_F$ and $\alpha=1$ instead, getting
\begin{equation}\label{bed_11_kompakt}
 g_1=\frac{1+x}{1-x}+\frac{2(1+\beta^2A)\gamma_1}{(1-x)(1-\beta^2A)}\in b\mathbf o_F^\times,
\end{equation}
\begin{equation}\label{bed_21_kompakt}
 g_1y_2+\frac{4A\beta\gamma_1}{(1-x)(1-\beta^2A)}\in\mathbf o_F.
\end{equation}
Now assume first, that $v(x)\not=0$. Then $v(\frac{1+x}{1-x})=0$ and
\begin{equation*}
 v\left(\frac{2(1+\beta^2A)\gamma_1}{(1-x)(1-\beta^2A)}\right)\geq\frac{1}{2}\lvert v(x)\rvert>0,
\end{equation*}
i.e. condition~(\ref{bed_11_kompakt}) is satisfied only for $b\in\mathbf o_F^\times$. In this case, condition~(\ref{bed_21_kompakt}) is satisfied as well and the contribution for $v(x)\not=0$ is given by
\begin{itemize}
 \item[$\bullet$] $v(x)\not=0$: $b\in\mathbf o_F^\times$ and $\beta\in\mathbf o_F$.
\end{itemize}
Now assume $v(x)=0$. Then (\ref{bed_11_kompakt}) is equivalent to
\begin{equation*}
 v\left((1+x)(1-\beta^2A)+2(1+\beta^2A)\gamma_1\right)=v(b)+v(1-x),
\end{equation*}
respectively
\begin{equation}\label{bed_111_kompakt}
 v(b)+v(1-x)=2\operatorname{min}\{v(1+\gamma_1),v(\beta)\}\geq 0.
\end{equation}
In here, $v(b)+v(1-x)>0$ is possible if and only if $\gamma_1\in -1+\wp$ and $\beta\in\wp$. 
One can omit this case by choosing the preimage $\gamma(x)=\sqrt A+\epsilon\gamma_1$ such that $\gamma_1\in 1+\wp$ if $x\in 1+\wp$. This doesn't influnce the result, as the local linking numbers are independent of the choice of the tracefree preimage $\gamma(x)$.

Thus, let $v(b)+v(1-x)=0$. 
Then the case $v(b)\geq 0$ even forces $0=v(b)=v(1-x)$. That is, condition~(\ref{bed_21_kompakt}) is
\begin{equation*}
 \frac{4\beta A\gamma_1}{(1-x)(1-\beta^2A)}\in\mathbf o_F,
\end{equation*}
which here is equivalent to $\beta\in\mathbf o_F$. This yields the contribution
\begin{itemize}
 \item[$\bullet$]
$x\in\mathbf o_F^\times\backslash (1+\wp)$: $b\in\mathbf o_F^\times$ and $\beta\in\mathbf o_F$. 
\end{itemize}
While in case $v(b)<0$, one has $-v(b)=v(1-x)>0$. One assumes again that the square root $\gamma_1$ of $x$  is in $1+\wp$: $\gamma_1\in 1+\wp$. Then condition~(\ref{bed_21_kompakt}) is equivalent to
\begin{equation*}
 \bigl(\beta^2A(1-\gamma_1)^2-(1+\gamma_1)^2\bigr)y_2\in 4\beta A\gamma_1+\wp^{v(1-x)}.
\end{equation*}
As $(1-\gamma_1)^2\in\wp^{2v(1-x)}$, that is $\beta\in\frac{-(1+\gamma_1)^2y_2}{4A\gamma_1}+\wp^{v(1-x)}$. Thus, the contribution in this case is
\begin{itemize}
 \item[$\bullet$]
$x\in 1+\wp$: $v(b)=-v(1-x)$ and $\beta\in\frac{-(1+\sqrt x)^2y_2}{4A\sqrt x}+\wp^{v(1-x)}$.
\end{itemize}
{\bf 2) Let $x\in F^{\times 2}$ and $v(\frac{\beta}{\alpha})< 0$.}
Again, take $x=\gamma_1^2$. By reducing conditions~(\ref{bed_1_kompakt}) and (\ref{bed_2_kompakt}) by $\beta^2$, one can assume $\beta=1$ and $\alpha\in\wp$. The conditions now have the shape
\begin{equation}\label{bed_12_kompakt}
 \frac{1+x}{1-x}+\frac{2(\alpha^2+A)\gamma_1}{(1-x)(\alpha^2-A)}\in b\mathbf o_F^\times,
\end{equation}
\begin{equation}\label{bed_22_kompakt}
 g_1y_2+\frac{4\alpha\gamma_1}{(1-x)(\alpha^2 -A)}\in\mathbf o_F.
\end{equation}
If one substitutes in condition~(\ref{bed_11_kompakt}) resp. (\ref{bed_21_kompakt}) $\beta\mapsto \alpha A^{-1}\in\wp$ and $\gamma_1\mapsto -\gamma_1$, one gets exactly (\ref{bed_12_kompakt}) resp. (\ref{bed_22_kompakt}). Thus, taking $\gamma_1\in -1+\wp$ if $x\in 1+\wp$ this time, one can read off the contributions here from those of the first case:
\begin{itemize}
 \item [$\bullet$]
$v(x)\not=0$ or $x\in\mathbf o_F^\times\backslash(1+\wp)$: $b\in\mathbf o_F^\times$ and $\alpha\in\wp$,
\item[$\bullet$]
$x\in 1+\wp$: $v(b)=-v(1-x)$, $\alpha\in\frac{(1-\sqrt x)^2y_2}{4A\sqrt x}+\wp^{v(1-x)}$ and $y_2\in\wp$.
\end{itemize}

{\bf 3) Let $x\in\operatorname{N}\backslash F^{\times2}$ and $-1\in F^{\times 2}$ and $v(\frac{\beta}{\alpha})\geq 0$.}
In this case, choose the tracefree preimage $\gamma(x)=\sqrt A+\epsilon \gamma_2$ to get $x=-\gamma_2A$ . Reducing the conditions~(\ref{bed_1_kompakt}) and (\ref{bed_2_kompakt}), one can assume $\alpha=1$ and $\beta\in\mathbf o_F$. The conditions now are
\begin{equation}\label{bed_13_kompakt}
 g_1=\frac{1+x}{1-x}+\frac{4\beta A\gamma_2}{(1-x)(1-\beta^2A)}\in b\mathbf o_F^\times,
\end{equation}
\begin{equation}\label{bed_23_kompakt}
 g_1y_2+\frac{2A(1+\beta^2A)\gamma_2}{(1-x)(1-\beta^2A)}\in\mathbf o_F.
\end{equation}
If $v(x)\not=0$, then $v(\frac{1+x}{1-x})=0$ and 
\begin{equation*}
 v\left(\frac{4\beta A\gamma_2}{(1-x)(1-\beta^2A)}\right)\geq \frac{1}{2}\lvert v(x)\rvert>0.
\end{equation*}
Thus, (\ref{bed_13_kompakt}) is equivalent to $v(b)=0$. Then (\ref{bed_23_kompakt}) is satisfied and the contribution is
\begin{itemize}
 \item [$\bullet$] $v(x)\not=0$: $b\in\mathbf o_F^\times$ and $\beta\in\mathbf o_F$.
\end{itemize}
If $v(x)=0$, then $v(1-x)=0$ as $x$ is not a square. Thus,
\begin{equation*}
 v\left(\frac{1+x}{1-x}+\frac{4\beta A\gamma_2}{(1-x)(1-\beta^2A)}\right)=v\bigl((1+x)(1-\beta^2A)+4\beta A\gamma_2\bigr)\geq 0
\end{equation*}
and condition~(\ref{bed_13_kompakt}) implies $v(b)\geq 0$. Then again, condition~(\ref{bed_23_kompakt}) is satisfied. But one has to look more exactly at (\ref{bed_13_kompakt}) to get a sharp condition for $b$. As
\begin{equation*}
 (1+x)(1-\beta^2A)+4\beta A\gamma_2=(1+\beta\gamma_2A)^2-(\beta-\gamma_2)^2A,
\end{equation*}
(\ref{bed_13_kompakt}) can be rewritten as
\begin{equation*}
 v(b)=2\operatorname{min}\{v(1+\beta\gamma_2A),v(\beta-\gamma_2)\}.
\end{equation*}
For $v(b)>0$, both $v(1+\beta\gamma_2A)>0$ and $v(\beta-\gamma_2)>0$ have to be satisfied, that is 
\begin{equation*}
 \beta\in(\gamma_2+\wp)\cap(-(\gamma_2A)^{-1}+\wp)=\emptyset.
\end{equation*}
Thus, $v(b)=0$ and the contribution is
\begin{itemize}
 \item [$\bullet$]
$v(x)=0$: $b\in\mathbf o_F^\times$ and $\beta\in\mathbf o_F$.
\end{itemize}

{\bf 4) Let $x\in\operatorname{N}\backslash F^{\times2}$ and $-1\in F^{\times 2}$ and $v(\frac{\beta}{\alpha})< 0$.}
By reducing conditions~(\ref{bed_1_kompakt}) and (\ref{bed_2_kompakt}), one can assume $\beta=1$ and $\alpha\in\wp$. This case now is done as the previous one substituting $\beta\mapsto\alpha A^{-1}$ and $\gamma_2\mapsto -\gamma_2$ there. The contribution is
\begin{itemize}
\item[$\bullet$]
$x\in \operatorname N\backslash F^{\times 2}$: $b\in\mathbf o_F^\times$ and $\alpha\in \wp$.
\end{itemize}

{\bf 5) Let $x\in\operatorname{N}\backslash F^{\times2}$ and $-1\notin F^{\times 2}$ and $v(\frac{\beta}{\alpha})\geq 0$.} One again assumes $\alpha=1$ and $\beta\in\mathbf o_F$. As $-1$ is not a square, one has without loss of generality $A=-1$. The norm is surjective on $\mathbf o_F^\times$, thus there is $\gamma_3\in\mathbf o_F^\times$ such that $1+\gamma_3^2$ is not a square. Choose $\gamma(x)=\gamma_1(1+\gamma_3\sqrt{-1})$ to get $x=\gamma_1^2(1+\gamma_3^2)$.  Conditions ~(\ref{bed_1_kompakt}) and (\ref{bed_2_kompakt}) now are read as
\begin{equation}\label{bed_15_kompakt}
\frac{1+x}{1-x}+\frac{2\gamma_1(1-\beta^2+2\beta\gamma_3)}{(1-x)(1+\beta^2)}\in b\mathbf o_F^\times,
\end{equation}
\begin{equation}\label{bed_25_kompakt}
g_1y_2-\frac{2\gamma_1\left((1-\beta^2)\gamma_3+2\beta\right)}{(1-x)(1+\beta^2)}\in\mathbf o_F.
\end{equation}
As seen earlier, $1-\beta^2+2\beta\gamma_3\in\mathbf o_F^\times$ and $(1-\beta^2)\gamma_3+2\beta\in\mathbf o_F^\times$. As neither $-1$ nor $x$ are squares in $F$, one has $v(\frac{1+x}{1-x})\geq 0$ as well as 
\begin{equation}\label{bed_151_kompakt}
v\left(\frac{2\gamma_1(1-\beta^2+2\beta\gamma_3)}{(1-x)(1+\beta^2)}\right)\geq 0.
\end{equation}
That is, (\ref{bed_15_kompakt}) implies $b\in\mathbf o_F$. Then (\ref{bed_25_kompakt}) is equivalent to 
\begin{equation*}
\frac{2\gamma_1\left((1-\beta^2)\gamma_3+2\beta\right)}{(1-x)(1+\beta^2)}\in\mathbf o_F,
\end{equation*}
which is true. One studies (\ref{bed_15_kompakt}) further: If $v(x)\not=0$, then in (\ref{bed_151_kompakt}) one even has $>$. Thus, (\ref{bed_15_kompakt}) means $v(b)=0$. If $v(x)=0$, then (\ref{bed_15_kompakt}) is equivalent to
\begin{equation}\label{bed_152_kompakt}
v\left((1+x)(1+\beta^2)+2\gamma_1(1-\beta^2+2\beta\gamma_3)\right)=v(b).
\end{equation}
Assuming $v(b)>0$ first, one gets out of (\ref{bed_152_kompakt}) the condition
\begin{equation*}
(1+x-2\gamma_1)\beta^2+4\gamma_1\gamma_2\beta+(1+x+2\gamma_1)\in\wp.
\end{equation*}
This is a quadratic equation modulo $\wp$, which has roots modulo $\wp$ only if its discriminante is a square. (Notice $v(x)=0$, thus $v(1+x+2\gamma_1)=0$, thus $v(\beta)=0$.) This discriminante is
\begin{equation*}
4\left(4\gamma_1^2\gamma_3^2-(1+x-2\gamma_1)(1+x+2\gamma_1)\right)=-4(1-x)^2,
\end{equation*}
which is not a square. By (\ref{bed_152_kompakt}), one always has $v(b)=0$. The contribution of this case again is
\begin{itemize}
\item[$\bullet$]$x\in\operatorname N\backslash F^{\times 2}$: $b\in\mathbf o_F^\times$ and $\beta\in\mathbf o_F$.
\end{itemize}
{\bf 6) Let $x\in\operatorname{N}\backslash F^{\times2}$ and $-1\notin F^{\times 2}$ and $v(\frac{\beta}{\alpha})< 0$.} 
Again, one can assume $\beta=1$ and $\alpha\in\wp$. As seen twice, this case follows from the previous one by substituting $\beta\mapsto \alpha A^{-1}$ and $\gamma_1\mapsto-\gamma_1$ there. This yields  the contribution
\begin{itemize}
\item[$\bullet$] $x\in\operatorname N\backslash F^{\times 2}$: $b\in\mathbf o_F^\times$ and $\alpha\in\wp$.
\end{itemize}
{\bf Computation of the local linking number.}
Now one collects all the contributions of the cases 1 to 6 marked by $\bullet$ and computes the integral on them. If there isn't given a range of $y_1$ resp. $y_2$, then it is arbitrary in the support of $\phi(by_1,y_2)$, i.e. $y_1\in b^{-1}\mathbf o_F^\times$ resp. $y_2\in\mathbf o_F$.

Notice first that the contributions of the cases $-1$ a square resp. $-1$ not a square are the same.
One gets
\begin{eqnarray*}
&&<\phi,\begin{pmatrix}b&0\\0&1\end{pmatrix}.\phi>_x\\
&=&
\mathbf 1_{\operatorname N\backslash (1+\wp)}(x)\mathbf 1_{\mathbf o_F^\times}(b)\int_{\mathbf o_F^\times}\int_{\mathbf o_F}\int_{T}~dt~dy_2~dy_1\\
&&
+\mathbf 1_{1+\wp}(x)\mathbf 1_{(1-x)\mathbf o_F^\times}(b)\operatorname{vol}_T(T_1)2q^{-v(1-x)}\operatorname{vol}^\times(\mathbf o_F^\times)\operatorname{vol}(\mathbf o_F)\\
&&
+\mathbf 1_{1+\wp}(x)\mathbf 1_{(1-x)^{-1}\mathbf o_F^\times}(b)\operatorname{vol}_T(T_1)2q^{-v(1-x)}\operatorname{vol}^\times(\mathbf o_F^\times)\operatorname{vol}(\mathbf o_F)\\
&=&
\left(\mathbf 1_{\operatorname N\backslash (1+\wp)}(x)\mathbf 1_{\mathbf o_F^\times}(b) 
+ \mathbf 1_{1+\wp}(x)\left(\mathbf 1_{(1-x)\mathbf o_F^\times}(b)+\mathbf 1_{(1-x)^{-1}\mathbf o_F^\times}(b)\right)q^{-v(1-x)}\right)\operatorname{vol},
\end{eqnarray*}
where $T_1:=\left\{\alpha+\beta\sqrt A\in T\mid v(\beta)\geq v(\alpha)\right\}$ and
\begin{equation*}
\operatorname{vol}:=\operatorname{vol}_T(T)\operatorname{vol}^\times(\mathbf o_F^\times)\operatorname{vol}(\mathbf o_F).
\end{equation*}
This finishes the proof of Example~\ref{bsp_lln_translatiert_kompakt}.
%%%%%%%%%%%%%%%%%%%%%%%%%%%%%%%%%%%%%%%
%%%
%%%
%%%
%%%%%%%%%%%%%%%%%%%%%%%%%%%%%%%%%%%%%%%%%%%%%%%%%%%%
\section*{Appendix B\\Proof of Example~\ref{bsp_lln_nichtkompakt}}
For the proof of Example~\ref{bsp_lln_nichtkompakt}, we need the following Lemma. Its short proof is left to the reader.
\begin{lem}\label{G_maximalkompakt_zerlegung}
Let $K/F$ be split. Then
\begin{displaymath}
 \operatorname{GL}_2(\mathbf o_F) = \mathbf o_K^\times N(\mathbf o_F)N'(\mathbf o_F) \quad{\bigcup^\bullet} \quad \mathbf o_K^\times N(\mathbf o_F)w N(\wp),
\end{displaymath}
where $N(X)$ is group of unipotent upper triangular matrices having nontrivial entries in $X$.
\end{lem}
By Lemma~\ref{G_maximalkompakt_zerlegung}, 
\begin{equation*}
 \phi =\chi\cdot\mathbf 1_{\operatorname{GL}_2(\mathbf o_F)}=\chi\cdot\mathbf 1_{N(\mathbf o_F)N'(\mathbf o_F)}+\chi\cdot\mathbf 1_{N(\mathbf o_F)wN(\wp)}.
\end{equation*}
For $y\in  TNN'$ take the following representative modulo $T$
\begin{displaymath}
 y=\begin{pmatrix}1&y_2\\0&1\end{pmatrix}\begin{pmatrix}1&0\\y_3&1\end{pmatrix}.
\end{displaymath}
For such $y$ there is
\begin{displaymath}
 \phi(y)=\phi_1(y_2,y_3):=\mathbf 1_{\mathbf o_F}(y_2)\mathbf 1_{\mathbf o_F}(y_3).
\end{displaymath}
Analogly, for  $y\in  TNwN$ take the following representative modulo $T$
\begin{displaymath}
 y=\begin{pmatrix}1&y_1\\0&1\end{pmatrix}w\begin{pmatrix}1&y_4\\ 0&1\end{pmatrix}
\end{displaymath}
to get
\begin{displaymath}
 \phi(y)=\phi_2(y_1,y_4):=\mathbf 1_{\mathbf o_F}(y_1)\mathbf 1_{\wp}(y_4).
\end{displaymath}
We will use this splitting of $\phi$ for the exterior function $\psi=\phi$:
\begin{align*}
 <\phi,\begin{pmatrix}b&0\\0&1\end{pmatrix}.\phi>_x &=
\int_F\int_F\int_T\phi(t^{-1}\gamma(x)ty)~dt~\bar\chi_1(b)\bar\phi_1(b^{-1}y_2,by_3)~dy_2dy_3 \\
&+
\int_F\int_F\int_T\phi(t^{-1}\gamma(x)ty)~dt~\chi_1(b)\bar\phi_2(by_1,b^{-1}y_4)~dy_1dy_4.
\end{align*}
Choosing Haar measures on $TNN'$ and $TNwN$  as in Section~\ref{section_nonkompakt}, we get
\begin{displaymath}
 \vol_{T\backslash G}(T\cdot\operatorname{GL}_2(\mathbf o_F)) =\vol(\mathbf o_F)^2(1+q^{-1}).
\end{displaymath}
%%%%%%%%%%%%%%%%%%%%%%%
%%
% %
% %
%%%%%%%%%%%%%%%%%%%%%%%%%%%
The inner integrand $\phi(t^{-1}\gamma(x)ty)$ does not vanish only if there is $(r\mathbf o_F^\times,s\mathbf o_F^\times)\subset T$ such that
\begin{equation}\label{bsp_nonkp_bed_inner_integrand}
 \begin{pmatrix}r&0\\0&s\end{pmatrix}t^{-1}\gamma(x)ty \in\operatorname{GL}_2(\mathbf o_F).
\end{equation}
As $\operatorname{GL}_2(\mathbf o_F)$ is fundamental for $\chi$ unramified, there is at most one class $(r\mathbf o_F^\times,s\mathbf o_F^\times)$ satisfying (\ref{bsp_nonkp_bed_inner_integrand}).
If there is, then the value of the inner integrand is $\phi(t^{-1}\gamma(x)ty)=\chi_1(r^{-1}s)$.
We choose once and for all representatives
\begin{equation*}
 t=\begin{pmatrix}a&0\\0&1\end{pmatrix}\quad\textrm{ and } \quad\gamma(x)=\begin{pmatrix}-1&x\\-1&1\end{pmatrix}.
\end{equation*}
For $y\in NN'$  condition (\ref{bsp_nonkp_bed_inner_integrand})  is
\begin{equation}\label{bspganzheitsbedingungen}
 \begin{pmatrix}
       -r(1+y_2y_3-a^{-1}xy_3)&r(-y_2+a^{-1}x)\\
	-s(a(1+y_2y_3)-y_3)&s(1-ay_2)
      \end{pmatrix}\in \operatorname{GL}_2(\mathbf o_F).
\end{equation}
That is, all components are integral and the determinant satisfies
\begin{equation*}
 rs\det (t^{-1}\gamma(x)ty) =rs(x-1)\in\mathbf o_F^\times.
\end{equation*}
So we can choose 
\begin{equation*}
 r=s^{-1}(1-x)^{-1}.
\end{equation*}
The value of the inner integrand now is $\chi_1(r^{-1}s)=\chi_1(1-x)$.

For $y\in NwN$ condition (\ref{bsp_nonkp_bed_inner_integrand})  is
\begin{equation}\label{bspganzheitsbedingungen'}
\begin{pmatrix}
       r(-y_1+a^{-1}x)&-r(1-y_1y_4+a^{-1}xy_4)\\
	s(1-ay_1)&-s(a(1-y_1y_4)+y_4)
      \end{pmatrix}\in \operatorname{GL}_2(\mathbf o_F).
\end{equation}
Replacing $(y_2,y_3)$ by $(y_1,-y_4)$ in (\ref{bspganzheitsbedingungen}) yields (\ref{bspganzheitsbedingungen'}).
That is, the inner integral for $y\in NwN$ can be deduced of that for $y\in NN'$.

For the exterior function we observe: If $y\in NwN$, then
\begin{equation*}
\begin{pmatrix}b&0\\0&1\end{pmatrix} .\bar\phi_2(y)=\bar\chi\begin{pmatrix}1&0\\0&b\end{pmatrix}\bar\phi_2\begin{pmatrix}by_1&-1+y_1y_4\\1&b^{-1}y_4\end{pmatrix}=\chi_1(b)\phi_2(by_1,b^{-1}y_4).
\end{equation*}
While if $y\in NN'$, then
\begin{equation*}
\begin{pmatrix}b&0\\0&1\end{pmatrix} .\bar\phi_1(y)=\bar\chi\begin{pmatrix}b&0\\0&1\end{pmatrix}\bar\phi_1\begin{pmatrix}1+y_2y_3&b^{-1}y_2\\by_3&1\end{pmatrix}=\chi_1^{-1}(b)\phi_2(b^{-1}y_2,by_3).
\end{equation*}
Thus, for deducing the case $y\in NwN$ of the case $y\in NN'$,
one has to substitute $(y_1,y_4)\mapsto (y_2,-y_3)$, and additionally one has to replace  $b$ by $b^{-1}$.

Further, there is an a-priori condition on $(y_2,y_3)$ given by the exterior function:
\begin{equation}\label{bsp_nonkompakt_apriori_bed}
 v(y_2)\geq -v(y_3).
\end{equation}
For assuming $v(y_2)<-v(y_3)$ and $\phi_1(b^{-1}y_2,by_3)\not=0$ implies the contradiction 
$v(b)\leq v(y_2)<-v(y_3)\leq v(b)$.

\subsection*{Conditions for the inner integrand}
From now on we assume  (\ref{bsp_nonkompakt_apriori_bed}).

{\it Claim.} 
In the case $y\in NN'$ the conditions (\ref{bspganzheitsbedingungen}) for the inner integrand not to vanish  imply exactly the following possible scopes.

For $x\in F^\times\backslash (1+\wp)$:
\begin{itemize}
 \item[$\bullet$] $-v(y_3)\leq v(y_2)<0$:
\begin{itemize}
 \item[$\ast$] $a\in\frac{1}{y_2}(1+\wp^{-v(y_2)})$ (for $v(s)=v(y_2)$)
 \item[$\ast$] $a\in\frac{x}{y_2}(1+\wp^{-v(y_2)})$ (for $v(s)=-v(1-x)$)
\end{itemize}
\item[$\bullet$] $v(1+y_2y_3)<-v(y_3)\leq v(y_2)$:
\begin{itemize}
 \item[$\ast$] $a\in\frac{y_3}{1+y_2y_3}(1+\wp^{-v(y_3)-v(1+y_2y_3)})$ (for $v(s)=v(1+y_2y_3)$)
 \item[$\ast$] $a\in\frac{xy_3}{1+y_2y_3}(1+\wp^{-v(y_3)-v(1+y_2y_3)})$ (for $v(s)=-v(y_3)-v(1-x)$)
\end{itemize}
\item[$\bullet$] $v(y_2)\geq 0$ and $v(y_3)\geq 0$: 
For $v(x)\geq 0$: $0\leq v(a)\leq v(x)$ (for $v(s)=0$); for $v(x)<0$: $0\leq v(s)\leq -v(x)$ (for $v(s)=-v(a)$).
\item[$\bullet$] $v(y_3)<0$ and $v(y_2)=-v(y_3)\leq v(1+y_2y_3)$: 
For $v(x)\geq 0$: $2v(y_3)\leq v(a)\leq v(x)+2v(y_3)$ (for $v(s)=0$); for $v(x)<0$: $-v(y_3)\leq v(s)\leq -v(y_3)-v(x)$ (for $v(s)=-v(a)+v(y_3)$).
\end{itemize}
For $x\in 1+\wp$:
\begin{itemize}
 \item[$\bullet$] $-v(y_3)\leq v(y_2)< -v(1-x)$:
 $a\in\frac{x}{y_2}(1+\wp^{-v(y_2)})$ (for $v(s)=-v(1-x)$)
 \item[$\bullet$] $-v(y_3)\leq v(y_2)\leq -v(1-x)$:
 $a\in\frac{1}{y_2}(1+\wp^{-v(y_2)})$ (for $v(s)=v(y_2)$)
\item[$\bullet$] $0\leq v(1+y_2y_3)<-v(y_3)-v(1-x)$: $a\in\frac{xy_3}{1+y_2y_3}(1+\wp^{-v(y_3)-v(1+y_2y_3)})$ (for $v(s)=-v(y_3)-v(1-x)$)
\item[$\bullet$] $0\leq v(1+y_2y_3)\leq-v(y_3)-v(1-x)$: $a\in\frac{y_3}{1+y_2y_3}(1+\wp^{-v(y_3)-v(1+y_2y_3)})$ (for $v(s)=v(1+y_2y_3)$)
\end{itemize}
There is no difficulty in checking that all these scopes satisfy condition (\ref{bspganzheitsbedingungen}).  That these  are the only possible ones is done by wearying	distinction of cases.  
As those are characteristic for the proof of Theorem~\ref{satz_translatiert_nicht_kompakt} and we skipped that, we include them here to give an insight of what is going on.
\begin{proof}[Proof of Claim]
Condition  (\ref{bspganzheitsbedingungen}) is equivalent to the following four conditions
\begin{equation}\label{bsp1}
 a^{-1}\in \frac{1+y_2y_3}{xy_3}(1+\frac{s(1-x)}{1+y_2y_3}\mathbf o_F),
\end{equation}
\begin{equation}\label{bsp2}
 a^{-1}\in \frac{y_2}{x}(1+\frac{s(1-x)}{y_2}\mathbf o_F),
\end{equation}
\begin{equation}\label{bsp3}
 a\in\frac{y_3}{1+y_2y_3}(1+\frac{1}{sy_3}\mathbf o_F),
\end{equation}
\begin{equation}\label{bsp4}
 a\in\frac{1}{y_2}(1+ s^{-1}\mathbf o_F).
\end{equation}
We now go through these conditions distinguishing different cases for $s$ and $x$.
{\bf 1)} Assume $v(s)<0$.
Then  $a=\frac{1}{y_2}(1+a')$ where $a'\in s^{-1}\mathbf o_F\subset \wp$ by (\ref{bsp4}). Inserting this in (\ref{bsp3}) we get
\begin{equation*}
 \frac{1}{y_2}(1+a')+y_3 a' \in s^{-1}\mathbf o_F,
\end{equation*}
Which by (\ref{bsp_nonkompakt_apriori_bed}) is equivalent to
\begin{equation}\label{bsp3v(s)<0}
 v(y_2)\leq v(s)<0.
\end{equation}
In particular, $v(y_3)>0$. Assuming this, the conditions (\ref{bsp1}) and (\ref{bsp2}) are equivalent to 
\begin{equation}\label{bsp1v(s)<0}
 1\in s(1-x)\mathbf o_F
\end{equation}
and
\begin{equation}\label{bsp2v(s)<0}
a^{-1}\in\frac{y_2}{x}(1+\frac{s(1-x)}{y_2}\mathbf o_F).
\end{equation}

{\bf 1.1)} If  $v(\frac{s(1-x)}{y_2})>0$, then combining (\ref{bsp4}) and (\ref{bsp2v(s)<0})   we get
\begin{equation}\label{bspv(s)<0,2gut,24kombination}
 a\in \frac{x}{y_2}(1+\frac{s(1-x)}{y_2}\mathbf o_F) \cap \frac{1}{y_2}(1+ s^{-1}\mathbf o_F).
\end{equation}
For this intersection to be nonempty, one has to assume $x\in 1+\wp$.
Collecting  conditions (\ref{bsp3v(s)<0}) and (\ref{bsp1v(s)<0}) as well as $v(\frac{s(1-x)}{y_2})>0$  we have
\begin{equation}\label{bspv(s)<0,2gut,s-bed}
 v(y_2)\leq v(s)\leq -v(1-x).
\end{equation}
If $v(s)=v(y_2)+j$, then (\ref{bspv(s)<0,2gut,24kombination}) is
\begin{equation*}
 a\in \frac{x}{y_2}(1+\wp^{v(1-x)+j}) \cap \frac{1}{y_2}(1+ \wp^{-v(y_2)-j}).
\end{equation*}
For $j=0$ this is $ a\in \frac{1}{y_2}(1+ \wp^{-v(y_2)})$,
because in this case $x\in 1+\wp^{v(1-x)}$.
For $j>0$ we have $x\notin 1+\wp^{v(1-x)+j}$. Then the scope for $a$ is nonempty only if $x\in1+ \wp^{-v(y_2)-j}$ is satisfied, that is $v(1-x)\geq -v(y_2)-j$. Together with (\ref{bspv(s)<0,2gut,s-bed}) we now get $v(y_2)+j=v(s)=-v(1-x)$. 
Summing up: In the case  $v(\frac{s(1-x)}{y_2})>0$ we have $x\in 1+\wp$ and the scopes are:
\begin{itemize}
 \item[$\bullet$] $v(s)=v(y_2)\leq -v(1-x)$: $a\in \frac{1}{y_2}(1+ \wp^{-v(y_2)})$,
\item[$\bullet$] $v(y_2)<v(s)=-v(1-x)$: $a\in \frac{x}{y_2}(1+ \wp^{-v(y_2)})$.
\end{itemize}
{\bf 1.2)} If  $v(\frac{s(1-x)}{y_2})\leq0$, then we find by (\ref{bsp3v(s)<0}), (\ref{bsp1v(s)<0}) and (\ref{bsp2v(s)<0}):
\begin{equation*}
v(s)\geq v(y_2)\left\{\begin{array}{l}=-v(a)\geq v(s)+v(1-x)-v(x)\\\geq v(s)+v(1-x)\end{array}          \right.
\end{equation*}
As we always have $\operatorname{max}\{v(1-x)-v(x),v(1-x)\}\geq 0$, this means $v(s)=v(y_2)$
and $\operatorname{max}\{v(1-x)-v(x),v(1-x)\}=0$. 
Thus, the scope in this case is nonzero only if $x\in F^\times\backslash (1+\wp)$. Then it is given by
\begin{itemize}
 \item $v(s)=v(y_2)<0$: $a\in \frac{1}{y_2}(1+ \wp^{-v(y_2)})$.
\end{itemize}
Now the case $v(s)<0$ is exhausted.

{\bf 2)} Assume  $v(s)\geq 0$.
This case is much more complicated than the previous.
Condition (\ref{bsp4}) now is
\begin{equation}\label{bsp4v(s)>0}
 v(a)\geq -v(y_2)-v(s).
\end{equation}
For condition (\ref{bsp3}) we distinguish further:

{\bf 2.1)} If $-v(s)-v(y_3)>0$: Then $\frac{1}{sy_3}\mathbf o_F\subset\wp$ and $a=\frac{y_3}{1+y_2y_3}(1+a')^{-1}$ where $a'\in\frac{1}{sy_3}\mathbf o_F$.
Inserting $a^{-1}$ in (\ref{bsp1}) we get the condition
\begin{equation}\label{bsp1v(s)>0,3gut}
 (1+y_2y_3)(1-x-xa')\in s(1-x)\mathbf o_F.
\end{equation}
{\bf 2.1.1)} If additionally $x\in F^\times\backslash (1+\wp)$, then collecting all conditions for $v(s)$ we get:
\begin{equation}\label{bspv(s)>0,xwegvon1,3gut,v(s)-bed}
 \left.\begin{array}{r}0\leq\\v(1+y_2y_3)-v(y_2y_3)\leq\end{array}\right\}
v(s)\left\{\begin{array}{l}<-v(y_3) \textrm{ by distinction of cases}\\\leq v(1+y_2y_3)\textrm{ by (\ref{bsp1v(s)>0,3gut}) and (\ref{bsp4v(s)>0})}\end{array}\right..
\end{equation}
It is easily seen that (\ref{bspv(s)>0,xwegvon1,3gut,v(s)-bed}) is satisfied only for $v(s)=v(1+y_2y_3)$. Thus, (\ref{bsp2}) means
\begin{displaymath}
 a^{-1}\in \frac{(1+y_2y_3)(1-x)}{x}\mathbf o_F,
\end{displaymath}
as  $v(y_2)\geq -v(y_3)>v(s)\geq v(s(1-x))$.
 This condition is because of $v(a)=v(\frac{y_3}{1+y_2y_3})$  equivalent to $-v(y_3)\geq v(1-x)-v(x)$. As we are studying the case $x\notin 1+\wp$ , this is always true.
Thus, the scope for $x\in F^\times\backslash (1+\wp)$ is
\begin{itemize}
 \item[$\bullet$] $v(s)=v(1+y_2y_3)$, $0\leq v(1+y_2y_3)<-v(y_3)$: 

$a\in \frac{y_3}{1+y_2y_3}(1+\wp^{-v(y_3)-v(1+y_2y_3)})$.
\end{itemize}
{\bf 2.1.2)} But if $x\in 1+\wp$, 
then we first show that the assumption $v(a^{-1}x)<v(s)+v(1-x)$ implies a contradiction:
For then (\ref{bsp2}) and (\ref{bsp3}) would imply
\begin{displaymath}
 a\in \frac{x}{y_2}(1+\wp)\cap\frac{y_3}{1+y_2y_3}(1+\wp),
\end{displaymath}
which is satisfied only for $\frac{x}{y_2}\in\frac{y_3}{1+y_2y_3}(1+\wp)$ or equivalently $1+y_2y_3\in y_2y_3(1+\wp)$. Which is a contradiction as for this, $1\in\wp$ must hold.
Thus, we have $v(a^{-1}x)\geq v(s)+v(1-x)$, and (\ref{bsp2}) together with (\ref{bsp3}) gives: $v(1+y_2y_3)-v(y_3)=-v(a)\geq v(s)+v(1-x)$. Collecting all the condition for $v(s)$ found so far:
\begin{equation}\label{bspv(s)>0,xnahe1,3gut,v(s)-bed}
 \left.\begin{array}{r}0\leq\\v(\frac{1+y_2y_3}{y_2y_3})\leq\\\end{array}\right\}
        v(s)
\left\{\begin{array}{l}<-v(y_3) \textrm{ by distinction of cases}\\ \leq v(y_2)-v(1-x)\textrm{ by (\ref{bsp4}) und (\ref{bsp2})}\\\leq -v(y_3)-v(1-x)+v(1+y_2y_3)\textrm{ by (\ref{bsp2})}\end{array}\right..
\end{equation}
It is easily seen that these conditions shrink to
\begin{displaymath}
 v(1+y_2y_3)\leq v(s)\leq -v(y_3)-v(1-x)
\end{displaymath}
because of $v(1+y_2y_3)<v(y_2)$.
Thus, $v(s)+v(1-x)-v(1+y_2y_3)>0$. Combing (\ref{bsp1}) and (\ref{bsp3}) we get
\begin{equation}\label{bspv(s)>0,xnahe1,3gut,1gut,13kombi}
 a\in \frac{xy_3}{1+y_2y_3}(1+\frac{s(1-x)}{1+y_2y_3}\mathbf o_F)\cap\frac{y_3}{1+y_2y_3}(1+\frac{1}{sy_3}\mathbf o_F).
\end{equation}
For $v(s)=v(1+y_2y_3)$ this is %(\ref{bspv(s)>0,xnahe1,3gut,1gut,13kombi})
\begin{displaymath}
 a\in \frac{y_3}{1+y_2y_3}(1+\wp^{-v(y_3)-v(1+y_2y_3)}),
\end{displaymath}
as in that case $x\in 1+\frac{s(1-x)}{1+y_2y_3}\mathbf o_F$. Let $v(s)=v(1+y_2y_3)+j$ where $j>0$. For the intersection in (\ref{bspv(s)>0,xnahe1,3gut,1gut,13kombi}) to be nonempty, we must have $v(1-x)\geq -v(1+y_2y_3)-v(y_3)-j$. Combined with the rest of the condition for $v(s)$ this is $v(s)=-v(y_3)-v(1-x)$. In particular, $v(1+y_2y_3)<-v(y_3)-v(1-x)$.

Summing up the conditions of this case, we get for $x\in 1+\wp$:
\begin{itemize}
 \item[$\bullet$] $v(y_2)>-v(y_3)$ and $v(y_3)\leq -v(1-x)$: $v(s)=0$ and $a\in\frac{y_3}{1+y_2y_3}(1+\wp^{-v(y_3)})$,
\item[$\bullet$] $v(y_2)>-v(y_3)$ and $v(y_3)<-v(1-x)$: $v(s)=-v(y_3)-v(1-x)$ and $a\in\frac{xy_3}{1+y_2y_3}(1+\wp^{-v(y_3)})$,
\item[$\bullet$] $v(y_2)=-v(y_3)$ and $v(1+y_2y_3)\leq -v(y_3)-v(1-x)$: $v(s)=v(1+y_2y_3)$ and $a\in\frac{y_3}{1+y_2y_3}(1+\wp^{-v(y_3)-v(1+y_2y_3)})$,
\item[$\bullet$] $v(y_2)=-v(y_3)$ and $0\leq v(1+y_2y_3)<-v(y_3)-v(1-x)$: $v(s)=-v(y_3)-v(1-x)$ and $a\in\frac{xy_3}{1+y_2y_3}(1+\wp^{-v(y_3)-v(1+y_2y_3)})$.
\end{itemize}
This finishes the case $-v(s)-v(y_3)>0$.
\medskip

{\bf 2.2)} If $-v(s)-v(y_3)\leq0$:

{\bf 2.2.1)} If additionally $v(s)+v(1-x)>v(y_2)$, then $\frac{s(1-x)}{y_2}\mathbf o_F\subset\wp$. Thus, by (\ref{bsp2}) we have $v(a)=v(x)-v(y_2)$. All conditions for  $v(s)$, which are given by the distinction of cases and (\ref{bsp3}) and  (\ref{bsp4}) are:
\begin{equation}\label{bspv(s)>0,3schlecht,2gut,v(s)-bed}
\left.\begin{array}{lr}
\textrm{Distinction of cases: }&0\leq\\\textrm{Distinction of cases: }&-v(y_3)\leq\\\textrm{by (\ref{bsp4}): }&-v(x)\leq\\\textrm{by (\ref{bsp3}): }&-v(1+y_2y_3)+v(y_2)-v(x)\leq\\\textrm{Distinction of cases: }&v(y_2)-v(1-x)<\end{array}\right\} v(s).
\end{equation}
We show that the assumption $v(s)+v(1-x)-v(1+y_2y_3)>0$ implies a contradiction:
Then (\ref{bsp1}) and (\ref{bsp2}) would imply
\begin{displaymath}
a\in\frac{xy_3}{1+y_2y_3}(1+\wp)\cap\frac{x}{y_2}(1+\wp).
\end{displaymath}
This intersection to be nonempty means $\frac{y_2y_3}{1+y_2y_3}\in 1+\wp$, or equvalently $1\in\wp$.

Thus, $v(s)+v(1-x)-v(1+y_2y_3)\leq 0$. Then by (\ref{bsp1}) we have
\begin{equation*}
v(s)\leq\left\{\begin{array}{l}v(1+y_2y_3)-v(1-x)\\v(y_2)+v(y_3)-v(1-x)\end{array},\right.
\end{equation*}
or equivalently $v(s)\leq- v(1-x)$. Conditions  (\ref{bspv(s)>0,3schlecht,2gut,v(s)-bed}) imply in particular $0\leq -v(1-x)$, that is $x\in F^\times\backslash(1+\wp)$. Moreover, $v(y_2)<0$. 
The conditions for $v(s)$ are now simplyfied to
\begin{equation*}
\left.\begin{array}{r}0\\-v(x)\end{array}\right\}\leq v(s)\leq -v(1-x),
\end{equation*}
which is equivalent to $v(s)=-v(1-x)$, for $x\notin 1+\wp$.
In this case the scope is given by
\begin{itemize}
\item[$\bullet$] $v(y_2)<0$: $v(s)=-v(1-x)$ and $a\in \frac{x}{y_2}(1+\wp^{-v(y_2)})$.
\end{itemize}
This finishes the case $v(s)+v(1-x)>v(y_2)$.
\medskip\\

{\bf 2.2.2)} But if $v(s)+v(1-x)\leq v(y_2)$, then  (\ref{bsp2}) is equivalent to
\begin{equation}\label{bspv(s)>0,3schlecht,2schlecht,2}
-v(a)\geq v(s)+v(1-x)-v(x).
\end{equation}We distingush further:

{\bf 2.2.2.1)} If additionally  $v(s)+v(1-x)-v(1+y_2y_3)>0$, then we have by (\ref{bsp1})
\begin{displaymath}
a\in \frac{xy_3}{1+y_2y_3}(1+\frac{s(1-x)}{1+y_2y_3}\mathbf o_F)\subset\frac{xy_3}{1+y_2y_3}(1+\wp).
\end{displaymath}
In particular, $v(a)=v(xy_3)-v(1+y_2y_3)$.
We collect the conditions for $v(s)$:
\begin{equation}\label{bspv(s)>0,3schlecht,1schlecht,v(s)-bed}
\left.\begin{array}{r}0\leq\\-v(y_3)\leq\\-v(x)-v(y_3)\leq\\-v(x)+v(\frac{1+y_2y_3}{y_2y_3})\leq\\v(\frac{1+y_2y_3}{1-x})<\end{array}\right\}v(s)\leq
\left\{\begin{array}{ll}v(y_2)-v(1-x)&\textrm{by distinction of cases}\\&\textrm{by distinction of cases}\\v(\frac{1+y_2y_3}{y_3(1-x)}) & \textrm{by (\ref{bsp3}) and (\ref{bsp2})}\\ & \textrm{by (\ref{bsp4})}\\ & \textrm{by distinction of cases}\end{array}.\right.
\end{equation}
The two conditions on the right combined are equivalent to
\begin{equation*}
v(s)\leq -v(y_3)-v(1-x),
\end{equation*}
by the general assumption $v(y_2)\geq -v(y_3)$ (\ref{bsp_nonkompakt_apriori_bed}).
This implies   $v(1+y_2y_3)<-v(y_3)$ as well as
 $0\leq -v(1-x)$, that is $x\in F^\times\backslash (1+\wp)$. Inserting this in (\ref{bspv(s)>0,3schlecht,1schlecht,v(s)-bed}) we finally get $v(s)=-v(y_3)-v(1-x)$. 
Thus, the scope of this case exists only for $x\in F^\times\backslash (1+\wp)$. It is then given by
\begin{itemize}
\item[$\bullet$] $0\leq v(1+y_2y_3)<-v(y_3)$: $v(s)=-v(y_3)-v(1-x)$ and $a\in \frac{xy_3}{1+y_2y_3}(1+\wp^{-v(y_3)-v(1+y_2y_3)})$.
\end{itemize}
{\bf 2.2.2.2)} If additionally  $v(s)+v(1-x)-v(1+y_2y_3)\leq 0$,
then (\ref{bsp1}) is equivalent to.
\begin{equation*}
-v(a)\geq v(s)+v(1-x)-v(x)-v(y_3).
\end{equation*}
By the distinction of cases and the conditions we now get the conditions
\begin{equation}\label{bspv(s)>0,allesschlecht,v(s)-bed}
\left.\begin{array}{r}0\\-v(y_3)\end{array}\right\}\leq v(s)\leq \left\{\begin{array}{l}v(y_2)-v(1-x)\\v(1+y_2y_3)-v(1-x)\end{array}\right..
\end{equation}
By(\ref{bsp1}) to (\ref{bsp4}) we see that $v(a)$ must satisfy:
\begin{equation}\label{bspv(s)>0,allesschlecht,a-bed}
\left.\begin{array}{r}-v(1+y_2y_3)-v(s)\\-v(y_2)-v(s)\end{array}\right\}\leq v(a) \leq \left\{\begin{array}{l}-v(s)+vx)-v(1-x)\\-v(s)+v(x)-v(1-x)+v(y_3)\end{array}\right. .
\end{equation}We distinguish further for $v(y_3)$:

{\bf 2.2.2.2 a)} If  $v(y_3)\geq 0$:
We show that $v(y_2)<0$ is not possible. For then (\ref{bspv(s)>0,allesschlecht,v(s)-bed}) would imply $v(1-x)\leq v(y_2)<0$. But  (\ref{bspv(s)>0,allesschlecht,a-bed}) would imply $v(1-x)-v(x)\leq v(y_2)<0$, which is a contradiction for $v(x)<0$.

Thus, $v(y_2)\geq 0$ and the conditions (\ref{bspv(s)>0,allesschlecht,v(s)-bed}) shrink to $0\leq v(s)\leq -v(1-x)$. This implies $x\notin 1+\wp$.
Condition (\ref{bspv(s)>0,allesschlecht,a-bed}) is reduced to $-v(s)\leq v(a)\leq -v(s)+v(x)-v(1-x)$.
We no can write down the scope of this case. The scope is nonempty only if
 $x\in F^\times\backslash(1+\wp)$ and is given by
\begin{itemize}
\item[$\bullet$] $v(y_2)\geq 0$ and $v(y_3)\geq 0$:
\begin{itemize}
\item[$\ast$] For $v(x)\geq 0$: $v(s)=0$ and $0\leq v(a)\leq v(x)$,
\item[$\ast$] For  $v(x)<0$: $v(a)=-v(s)$ and $0\leq v(s)\leq -v(x)$.
\end{itemize}
\end{itemize}
{\bf 2.2.2.2 b)} If $v(y_3)<0$: We show that $v(y_2)>-v(y_3)$ is not possible. Then (\ref{bspv(s)>0,allesschlecht,v(s)-bed}) would imply $v(1-x)\leq v(y_3)<0$, that is $v(x)<0$. By (\ref{bspv(s)>0,allesschlecht,a-bed}) we would get $0\leq v(x)-v(1-x)+v(y_3)$, which could be satisfied only for $v(x)\geq 0$.

Thus, $v(y_2)=-v(y_3)$. We show that $v(1+y_2y_3)<-v(y_3)$ is not possible, for   (\ref{bspv(s)>0,allesschlecht,v(s)-bed}) would again imply $v(1-x)<0$. Inserting this in (\ref{bspv(s)>0,allesschlecht,a-bed}), we get $-v(y_3)\leq v(1+y_2y_3)$, contradiction.

Thus, $v(1+y_2y_3)\geq -v(y_3)$. Then  (\ref{bspv(s)>0,allesschlecht,v(s)-bed}) means $-v(y_3)\leq v(s)\leq -v(y_3)-v(1-x)$, which is satisfied only if $x\notin 1+\wp$.  Condition (\ref{bspv(s)>0,allesschlecht,a-bed}) gives $v(y_3)-v(s)\leq v(a)\leq -v(s)+v(x)-v(1-x)+v(y_3)$. 
The scope of this case is nonempty only if
 $x\in F^\times\backslash(1+\wp)$. Then it is given by
\begin{itemize}
\item[$\bullet$]  $0<-v(y_3)\leq v(1+y_2y_3)$:
\begin{itemize}
\item[$\ast$] For $v(x)\geq 0$: $v(s)=-v(y_3)$ and $2v(y_3)\leq v(a)\leq v(x)+2v(y_3)$,
\item[$\ast$] For $v(x)<0$: $v(a)=v(y_3)-v(s)$ and $-v(y_3)\leq v(s)\leq-v(y_3) -v(x)$.
\end{itemize}
\end{itemize}
Finally, all the scopes  under the conditions
 (\ref{bsp1}) to (\ref{bsp4}) are treated. This proves the Claim.
\end{proof}
\subsection*{Computation of the inner integral}
We compute the inner integral incase $y\in NN'$. The scopes of integration are summed up in the Claim above. The variable of integration is $a\in F^\times$.

In case $x\in F^\times \backslash (1+\wp)$ we get:
\begin{align*}
& \chi_1(1-x)\vol^\times(\mathbf o_F^\times)(1-q^{-1})^{-1}\cdot\\
&\Biggl(
(\lvert v(x)\rvert +1)(1-q^{-1})\Bigl(\mathbf 1_{\mathbf o_F}(y_2)\mathbf 1_{\mathbf o_F}(y_3)+\mathbf 1_{F\backslash \mathbf o_F}(y_3)\mathbf 1_{-y_3^{-1}(1+\wp^{-v(y_3)})}(y_2)\Bigr)\Biggr.\\
& \quad\quad+
2q^{v(y_2)}\mathbf 1_{\wp^{-v(y_2)}}(y_3)\mathbf 1_{F\backslash \mathbf o_F}(y_2) +2q^{v(y_3)}\mathbf 1_{F\backslash \mathbf o_F}(y_3)\mathbf 1_{\wp^{-v(y_3)+1}}(y_2)\\
&\quad\quad\Biggl.
+2q^{v(y_3)+v(1+y_2y_3)}\mathbf 1_{F\backslash \mathbf o_F}(y_3)\mathbf 1_{-y_3^{-1}(\mathbf o_F^\times\backslash(1+\wp^{-v(y_3)}))}(y_2)\Biggr).
\end{align*}

In case $x\in 1+\wp$ we get:
\begin{align*}
 &\chi_1(1-x)\vol^\times(\mathbf o_F^\times)(1-q^{-1})^{-1}\cdot\\
&\Biggl( \mathbf 1_{F\backslash \wp^{-v(1-x)+1}}(y_2)q^{v(y_2)}\mathbf 1_{\wp^{-v(y_2)}}(y_3) 
+\mathbf 1_{F\backslash \wp^{-v(1-x)}}(y_2)q^{v(y_2)}\mathbf 1_{\wp^{-v(y_2)}}(y_3)\Biggr.\\
&\Biggl.+ \mathbf 1_{F\backslash \wp^{-v(1-x)+1}}(y_3)q^{v(y_3)}\mathbf 1_{\wp^{-v(y_3)+1}}(y_2)
+\mathbf 1_{F\backslash \wp^{-v(1-x)}}(y_3)q^{v(y_3)}\mathbf 1_{\wp^{-v(y_3)+1}}(y_2)\Biggr.\\
&\Biggl.\quad\quad+q^{v(y_3)+v(1+y_2y_3)}\mathbf 1_{F\backslash \wp^{-v(1-x)+1}}(y_3)\mathbf 1_{-y_3^{-1}(\mathbf o_F^\times\backslash(1+\wp^{-v(y_3)-v(1-x)+1}))}(y_2)\Biggr.\\
&\Biggl.\quad\quad  +q^{v(y_3)+v(1+y_2y_3)}\mathbf 1_{F\backslash \wp^{-v(1-x)}}(y_3)\mathbf 1_{-y_3^{-1}(\mathbf o_F^\times\backslash(1+\wp^{-v(y_3)-v(1-x)}))}(y_2)
\Biggr).
\end{align*}
For the inner integral in case  $y\in NwN$ we only have to consider such $(y_2,y_3)$ which satisfy $-v(y_3)<v(y_2)$, by the shape of the exterior function $\phi_2$.
Here, we get the following integrals:

In case $x\in F^\times\backslash (1+\wp)$:
\begin{align*}
 & \chi_1(1-x)\vol^\times(\mathbf o_F^\times)(1-q^{-1})^{-1}\cdot\\
&\Biggl(
(\lvert v(x)\rvert +1)(1-q^{-1})\mathbf 1_{\mathbf o_F\cap\wp^{-v(y_3)+1}}(y_2)\mathbf 1_{\mathbf o_F}(y_3)\Biggr.\\
& \quad\quad\Biggl. +
2q^{v(y_2)}\mathbf 1_{\wp^{-v(y_2)+1}}(y_3)\mathbf 1_{F\backslash \mathbf o_F}(y_2) +2q^{v(y_3)}\mathbf 1_{F\backslash \mathbf o_F}(y_3)\mathbf 1_{\wp^{-v(y_3)+1}}(y_2)\Biggr).
\end{align*}
In case $x\in 1+\wp$:
\begin{align*}
 & \chi_1(1-x)\vol^\times(\mathbf o_F^\times)(1-q^{-1})^{-1}\cdot\\
&\Biggl( \mathbf 1_{F\backslash \wp^{-v(1-x)+1}}(y_2)q^{v(y_2)}\mathbf 1_{\wp^{-v(y_2)+1}}(y_3)
+\mathbf 1_{F\backslash \wp^{-v(1-x)}}(y_2)q^{v(y_2)}\mathbf 1_{\wp^{-v(y_2)+1}}(y_3)\Biggr.\\
&\Biggl. +
\mathbf 1_{F\backslash \wp^{-v(1-x)+1}}(y_3)q^{v(y_3)}\mathbf 1_{\wp^{-v(y_3)+1}}(y_2)
+
\mathbf 1_{F\backslash \wp^{-v(1-x)}}(y_3)q^{v(y_3)}\mathbf 1_{\wp^{-v(y_3)+1}}(y_2)\Biggr).
\end{align*}

\subsection*{Computation of the exterior integral}
Now we integrate these inner integrals against the exterior function 
$\phi_i$.. Dabei treten f"ur $x\in 1+\wp$ folgende Summanden auf:

In case $x\in 1+\wp$ and $y\in NN'$, the exterior function is $\phi_1$.
We work off the inner integral term by term and get up to the factor
$\chi_1(b(1-x))\vol^\times(\mathbf o_F^\times)$:
\begin{align*}
 &(1-q^{-1})^{-1}\int\int\mathbf 1_{b\mathbf o_F\backslash\wp^{-v(1-x)+1}}(y_2)q^{v(y_2)}\mathbf 1_{b^{-1}\mathbf o_F\cap\wp^{-v(y_2)}}(y_3)~dy_2~dy_3\\ 
&\quad\quad=\mathbf 1_{\wp^{v(1-x)}}(b^{-1})\lvert b^{-1}\rvert \bigl(-v(b)-v(1-x)+1\bigr)\vol(\mathbf o_F)^2;
\end{align*}
\begin{align*}
 &(1-q^{-1})^{-1}\int\int\mathbf 1_{b\mathbf o_F\backslash\wp^{-v(1-x)}}(y_2)q^{v(y_2)}\mathbf 1_{b^{-1}\mathbf o_F\cap\wp^{-v(y_2)}}(y_3)~dy_2~dy_3\\ 
&\quad\quad=\mathbf 1_{\wp^{v(1-x)+1}}(b^{-1})\lvert b^{-1}\rvert \bigl(-v(b)-v(1-x)\bigr)\vol(\mathbf o_F)^2;
\end{align*}
\begin{align*}
 &(1-q^{-1})^{-1}\int\int\mathbf 1_{b\mathbf o_F\cap\wp^{-v(y_3)+1}}(y_2)q^{v(y_3)}\mathbf 1_{b^{-1}\mathbf o_F\backslash\wp^{-v(1-x)+1}}(y_3)~dy_2~dy_3\\ 
&\quad\quad=\mathbf 1_{\wp^{v(1-x)}}(b)\lvert b\rvert \bigl(q^{-1}+v(b)-v(1-x)\bigr)\vol(\mathbf o_F)^2;
\end{align*}
\begin{align*}
 &(1-q^{-1})^{-1}\int\int\mathbf 1_{b\mathbf o_F\cap\wp^{-v(y_3)+1}}(y_2)q^{v(y_3)}\mathbf 1_{b^{-1}\mathbf o_F\backslash\wp^{-v(1-x)}}(y_3)~dy_2~dy_3\\ 
&\quad\quad=\mathbf 1_{\wp^{v(1-x)+1}}(b)\lvert b\rvert \bigl(q^{-1}+v(b)-v(1-x)-1\bigr)\vol(\mathbf o_F)^2;
\end{align*}
\begin{align*}
& (1-q^{-1})^{-1}\int\int\Bigl( \mathbf 1_{b\mathbf o_F\cap-y_3^{-1}(\mathbf o_F^\times\backslash(1+\wp^{-v(y_3)-v(1-x)+1}))}(y_2)q^{v(y_3)+v(1+y_2y_3)}\cdot\Bigr.\\
&\Bigl.\quad\quad\quad\quad \mathbf 1_{b^{-1}\mathbf o_F\backslash\wp^{-v(1-x)+1}}(y_3)~dy_2~dy_3\Bigr)\\
&\quad\quad=\mathbf 1_{\wp^{v(1-x)}}(b)\lvert b\rvert \bigl(1-2q^{-1}+(1-q^{-1})(v(b)-v(1-x))\bigr)\vol(\mathbf o_F)^2;
\end{align*}
\begin{align*}
 &(1-q^{-1})^{-1}\int\int\mathbf 1_{b\mathbf o_F\cap-y_3^{-1}(\mathbf o_F^\times\backslash(1+\wp^{-v(y_3)-v(1-x)}))}(y_2)q^{v(y_3)+v(1+y_2y_3)}\\
&\quad\quad\quad\quad\quad\quad\quad\cdot\mathbf 1_{b^{-1}\mathbf o_F\backslash\wp^{-v(1-x)}}(y_3)~dy_2~dy_3\\ 
&=\mathbf 1_{\wp^{v(1-x)+1}}(b)\lvert b\rvert \bigl(1-2q^{-1}+(1-q^{-1})(v(b)-v(1-x)-1)\bigr)\vol(\mathbf o_F)^2;%\quad\quad&&
\end{align*}
In case $x\in 1+\wp$ and $y\in NwN$ the exterior function is $\phi_2$. We get up to the factor
 $\chi_1(b(1-x))\vol^\times(\mathbf o_F)$:
\begin{align*}
 &(1-q^{-1})^{-1}\int\int\mathbf 1_{b^{-1}\mathbf o_F\backslash\wp^{-v(1-x)+1}}(y_2)q^{v(y_2)}\mathbf 1_{b\wp\cap\wp^{-v(y_2)}}(y_3)~dy_2\:dy_3\\ 
&\quad\quad
=\mathbf 1_{\wp^{v(1-x)}}(b)\lvert b\rvert \bigl(v(b)-v(1-x)+1\bigr)q^{-1}\vol(\mathbf o_F)^2;
\end{align*}
\begin{align*}
 &(1-q^{-1})^{-1}\int\int\mathbf 1_{b^{-1}\mathbf o_F\backslash\wp^{-v(1-x)}}(y_2)q^{v(y_2)}\mathbf 1_{b\wp\cap\wp^{-v(y_2)+1}}(y_3)~dy_2~dy_3\\ 
&\quad\quad=\mathbf 1_{\wp^{v(1-x)+1}}(b)\lvert b\rvert \bigl(v(b)-v(1-x)\bigr)q^{-1}\vol(\mathbf o_F)^2;
\end{align*}
\begin{align*}
 &(1-q^{-1})^{-1}\int\int\mathbf 1_{b^{-1}\mathbf o_F\cap\wp^{-v(y_3)+1}}(y_2)q^{v(y_3)}\mathbf 1_{b\wp\backslash\wp^{-v(1-x)+1}}(y_3)~dy_2~dy_3\\ 
&\quad\quad=\mathbf 1_{\wp^{v(1-x)}}(b^{-1})\lvert b^{-1}\rvert \bigl(-v(b)-v(1-x)\bigr)\vol(\mathbf o_F)^2;
\end{align*}
\begin{align*}
 &(1-q^{-1})^{-1}\int\int\mathbf 1_{b^{-1}\mathbf o_F\cap\wp^{-v(y_3)+1}}(y_2)q^{v(y_3)}\mathbf 1_{b\wp\backslash\wp^{-v(1-x)}}(y_3)~dy_2~dy_3\\ 
&\quad\quad=\mathbf 1_{\wp^{v(1-x)+1}}(b^{-1})\lvert b^{-1}\rvert \bigl(-v(b)-v(1-x)-1\bigr)\vol(\mathbf o_F)^2;
\end{align*}
Summing up all terms and remembering the missing factor, we get the translated local linking number for $x\in 1+\wp$
and $\phi=\chi\cdot\mathbf 1_{\operatorname{GL}_2(\mathbf o_F)}$:
\begin{align*}
 & <\phi,\begin{pmatrix} b&0\\0&1\end{pmatrix}.\phi>_{x\in 1+\wp} =
\chi_1(1-x)\chi_1(b)\vol^\times (\mathbf o_F^\times)\vol(\mathbf o_F)^2\cdot\\
&\quad\quad\quad\Biggl(\mathbf 1_{\wp^{v(1-x)+1}}(b)\lvert b\rvert \bigl(2v(b)-2(v(1-x)+1)+1\bigr)\Biggr.\\
&\quad\quad\quad\quad\quad+\mathbf 1_{\wp^{v(1-x)+1}}(b^{-1})\lvert b^{-1}\rvert \bigl(-2v(b)-2(v(1-x)+1)+1\bigr)\\
&\quad\quad\quad\quad\quad
+\mathbf 1_{\wp^{v(1-x)}}(b)\lvert b\rvert \bigl(2v(b)-2v(1-x)+1\bigr)\\
&\quad\quad\quad\quad\quad\Biggl.+\mathbf 1_{\wp^{v(1-x)}}(b^{-1})\lvert b^{-1}\rvert \bigl(-2v(b)-2v(1-x)+1\bigr)\Biggr).
\end{align*}
This is the term claimed in Example~\ref{bsp_lln_nichtkompakt}.

In case $x\in F^\times\backslash(1+\wp)$, the occuring exterior integrals can mostly be read off those for $x\in 1+\wp$  substituting there $v(1-x)=0$.
The exeptional terms for $y\in NN'$ are:
\begin{align*}
&\int\int\mathbf 1_{\mathbf o_F\cap b\mathbf o_F}(y_2)\mathbf 1_{\mathbf o_F\cap b^{-1}\mathbf o_F}(y_3)~dy_2~dy_3
 =
\mathbf 1_{\mathbf o_F}(b)\lvert b\rvert +\mathbf 1_{\wp}(b^{-1})\lvert b^{-1}\rvert
\end{align*}
and
\begin{align*}
& \int\int\mathbf 1_{ b\mathbf o_F\cap -y_3^{-1}(1+\wp^{-v(y_3)})}(y_2)\mathbf 1_{ b^{-1}\mathbf o_F\backslash \mathbf o_F}(y_3)~dy_2~dy_3 =
\mathbf 1_{\wp}(b)\lvert b\rvert (1-q^{-1}).
\end{align*}
The exeptional term for $y\in NwN$ is:
\begin{align*}
& \int\int\mathbf 1_{\mathbf o_F\cap b^{-1}\mathbf o_F}(y_2)\mathbf 1_{\mathbf o_F\cap b\wp}(y_3)~dy_2~dy_3=
\mathbf 1_{\mathbf o_F}(b)\lvert b\rvert q^{-1} +\mathbf 1_{\wp}(b^{-1})\lvert b^{-1}\rvert.
\end{align*}
In these formulae we left the factor $\vol(\mathbf o_F)^2$ on the right and the factor 
 $\chi_1(1-x)\vol^\times(\mathbf o_F^\times)$.
Summing up all the terms, we get the translated local linking number for  $x\in F^\times\backslash(1+\wp)$
and $\phi=\chi\cdot\mathbf 1_{\operatorname{GL}_2(\mathbf o_F)}$:
\begin{align*}
 & <\phi,\begin{pmatrix} b&0\\0&1\end{pmatrix}.\phi>_x =
\chi_1(1-x)\chi_1(b)\vol^\times (\mathbf o_F^\times)\vol(\mathbf o_F)^2\cdot\\
&\quad\quad\quad\Biggl(\mathbf 1_{\mathbf o_F^\times}(b)\bigl(\lvert v(x)\rvert +1\bigr)(1+q^{-1}) +\mathbf 1_{\wp}(b)\lvert b\rvert \bigl(4v(b)+2\lvert v(x)\rvert\bigr)\Biggr.
\Biggr.\\
& \quad\quad\quad\quad\quad\quad\quad\quad\quad\quad\Biggl. +\mathbf 1_{\wp}(b^{-1})\lvert b^{-1}\rvert \bigl(-4v(b)+2\lvert v(x)\rvert\bigr)\Biggr).
\end{align*}
This is the result claimed in Example~\ref{bsp_lln_nichtkompakt}.
%%%%%%%%%%%%%%%%%%%%%%%%%
%%%
%%%
%%%%%%%%%%%%%%%%%%%%%%%%%%%%%%%

%\bibliographystyle{plain}
%\bibliography{reference_1}

\vspace*{.5cm}
\end{document}